\documentclass[12pt]{amsart}
\usepackage{graphicx}
\numberwithin{equation}{section}

\newcommand{\ben}{\begin{enumerate}}
\newcommand{\een}{\end{enumerate}}

\newcommand{\bea}{\begin{eqnarray}}
\newcommand{\ba}{\begin{array}}
\newcommand{\bean}{\begin{eqnarray*}}
\newcommand{\ea}{\end{array}}
\newcommand{\eea}{\end{eqnarray}}
\newcommand{\eean}{\end{eqnarray*}}
\newcommand{\beq}{\begin{equation}}
\newcommand{\eeq}{\end{equation}}
\newcommand{\bthm}{\begin{thm}}
\newcommand{\ethm}{\end{thm}}
\newcommand{\blem}{\begin{lem}}
\newcommand{\elem}{\end{lem}}
\newcommand{\bprop}{\begin{prop}}
\newcommand{\eprop}{\end{prop}}
\newcommand{\bcor}{\begin{cor}}
\newcommand{\ecor}{\end{cor}}
\newcommand{\bdfn}{\begin{dfn}}
\newcommand{\edfn}{\end{dfn}}
\newcommand{\brem}{\begin{rem}}
\newcommand{\erem}{\end{rem}}
\newcommand{\bpf}{\begin{proof}}
\newcommand{\epf}{\end{proof}}
\newcommand{\bfact}{\begin{fact}}
\newcommand{\efact}{\end{fact}}

\alph{enumii} \roman{enumiii}

\newtheorem{thm}{Theorem}[section]
\newtheorem{prop}[thm]{Proposition}
\newtheorem{lem}[thm]{Lemma}

\newtheorem{cor}[thm]{Corollary}
\newtheorem{dfn}[thm]{Definition}
\newtheorem{rem}[thm]{Remark}
\newtheorem{fact}[thm]{Fact}
\newtheorem{ex}[thm]{Example}

\def\N{{\mathbb N}}                      \def\R{{\mathbb R}}
\def\C{{\mathbb C}}                  \def\oc{{\hat \C}}
\def\CC{{\mathbb C}} 
\def\CCI{{\hat \CC }}
\def\NN{{\mathbb N}}
\def\RR{{\mathbb R}}
\def\ZZ{{\mathbb Z}}

\def\1{1\!\!1}
\def\and{\text{ and }}

        \def\diam{\text{\rm {diam}}}

      \def\Exp{\text{{\rm Exp}}}
\def\Epb{\text{{\rm Epb}}}

     \def\HD{\text{{\rm HD}}}

                \def\b{\beta}             
                         
\def\g{\gamma}                           \def\l{\lambda}
              \def\om{\omega}           
\def\Sg{\Sigma}               \def\sg{\sigma}

\def\({\bigl(}                \def\){\bigr)}

\def\ld{\ldots}                        \def\^{\tilde}

\def\es{\emptyset}

\def\om{\omega}

\begin{document}

\title[]
{\bf\large {\Large T}ransversality family  of Expanding 
Rational Semigroups}
\date{}
\author[\sc Hiroki SUMI]{\sc Hiroki SUMI}
%
\author[\sc Mariusz URBA\'NSKI]{\sc Mariusz URBA\'NSKI}
%
%
\thanks{The first author thanks University of North Texas for support and kind hospitality. 
The research of the first author was partially supported by JSPS KAKENHI 21540216.  
The research of the second author was supported in part by the
NSF Grant DMS 1001874.}
\thanks{\ \newline 
\noindent Hiroki Sumi\newline 
Department of Mathematics,
Graduate School of Science,
Osaka University, 
1-1 Machikaneyama,
Toyonaka,
Osaka, 560-0043, 
Japan\newline 
E-mail: sumi@math.sci.osaka-u.ac.jp\newline 
Web: http://www.math.sci.osaka-u.ac.jp/$\sim$sumi/\newline
\ \newline 
Mariusz Urba\'nski\newline Department of Mathematics,
 University of North Texas, Denton, TX 76203-1430, USA\newline  
E-mail: urbanski@unt.edu\newline
Web: http://www.math.unt.edu/$\sim$urbanski/}

\begin{abstract}
We study finitely generated expanding semigroups of rational maps with overlaps on the Riemann sphere. 
We show that if a $d$-parameter family of such semigroups satisfies the transversality condition, 
then for almost every parameter value the Hausdorff dimension of the Julia set 
is the minimum of $2$ and the zero of the pressure function. 
Moreover, the Hausdorff dimension of the exceptional set of parameters is 
estimated. We also show that if the zero of the pressure function is greater than $2$, 
then typically the $2$-dimensional Lebesgue measure of the Julia set 
is positive. Some sufficient conditions for a family to satisfy the transversality conditions are given. 
We give non-trivial examples of families of semigroups of non-linear polynomials with the transversality condition 
for which the Hausdorff dimension of the Julia set is typically equal to the zero of the pressure function and is less than $2$. 
We also show that a family of small perturbations of the Sierpinski gasket system 
satisfies that for a typical parameter value, the Hausdorff dimension of the Julia set (limit set) 
is equal to the zero of the pressure function, which is equal to the  similarity dimension.  
Combining the arguments on the transversality condition, thermodynamical formalisms and potential theory, 
we show that for each $a\in \CC $ with $|a|\neq 0,1$, 
 the family of small perturbations of the semigroup generated by $\{ z^{2}, az^{2}\} $ satisfies that  
for a typical parameter value,  the $2$-dimensional Lebesgue measure of the Julia set 
is positive.    
\end{abstract}
\maketitle
\noindent  Mathematics Subject Classification (2001). Primary 37F35; 
Secondary 37F15.\\ Date: January 8, 2013. Published in Adv. Math. 234 (2013) 697--734. 
\section{Introduction}
\label{Introduction}
A {\bf rational semigroup} 
is a semigroup generated by a family of 
non-constant rational maps $g:\oc \rightarrow \oc $, 
where $\oc $ denotes the Riemann sphere, 
with the semigroup operation being functional composition. 
A polynomial semigroup is a semigroup generated by a 
family of non-constant polynomial maps on $\oc .$ 
The work on the dynamics of rational semigroups was initiated 
by A. Hinkkanen and G. J. Martin (\cite{HM}), 
who were interested in the role of the dynamics of polynomial semigroups 
while studying various one-complex-dimensional moduli spaces for discrete 
groups of M\"{o}bius transformations, and by F. Ren's group 
(\cite{ZR}), who studied such semigroups from the perspective 
of random dynamical systems. 
  
 The theory of the dynamics of rational semigroups on $\oc $ 
has developed in many directions since the 1990s (\cite{HM, ZR, Stan, 
     sumihyp1, sumihyp2,  
hiroki1, sumi1, sumi2, sumi06, sumirandom,
 su1, sumid1, sumid3, SS, sumiintcoh, sumiprepare, sumisurvey, sumicp,sumird}). 
We recommend \cite{Stan} as an introductory article.     
For a rational semigroup $G$, we denote by $F(G)$ the maximal open
subset of $\CCI $ where  
$G$ is normal. The set $F(G)$ is called the Fatou set of $G$. 
The complement $J(G):= \CCI \setminus F(G)$ is called 
the Julia set of $G.$ 
Since the Julia set $J(G)$ of a rational semigroup 
$G=\langle f_{1},\ldots ,f_{m}\rangle $ 
generated by finitely many elements $f_{1},\ldots ,f_{m}$ 
has {\bf backward self-similarity} i.e. 
\begin{equation}
\label{bsseq}
 J(G)=f_{1}^{-1}(J(G))\cup \cdots \cup f_{m}^{-1}(J(G)),
\end{equation}  
(see \cite{sumihyp1, hiroki1}), rational semigroups 
can be viewed as a significant generalization and extension of 
both the theory of iteration of rational maps (see \cite{M,Be}) 
and conformal 
iterated function systems (see \cite{MU}). 
Indeed, because of (\ref{bsseq}), 
the analysis of the Julia sets of rational semigroups somewhat
resembles 
 ``backward iterated functions systems'', however since each map 
$f_{j}$ is not in general injective (critical points), some 
qualitatively different extra effort in the case  of semigroups is needed.
The theory of the dynamics of 
rational semigroups borrows and develops tools 
from both of these theories. It has also developed its own 
unique methods, notably the skew product approach 
(see \cite{hiroki1, sumi1, sumi2, sumi06, sumid1, sumid2, sumid3, sumiprepare,
sumisurvey, sumicp, sumird, suetds1,su1, subowen}).

The theory of the dynamics of rational semigroups is intimately 
related to that of the random dynamics of rational maps. 
The first study of random complex dynamics was given in \cite{FS}. 
In \cite{BBR, GQL}, random dynamics of quadratic polynomials were
investigated. The paper \cite{msu} develops the thermodynamic
formalism of random distance expanding maps and, in particular,
applies it to random polynomials.  
The deep relation between these fields 
(rational semigroups, random complex dynamics, and (backward) IFS) 
is explained in detail in the subsequent papers
 (\cite{sumirandom, sumid1, sumid2, sumid3,  sumiintcoh, sumiprepare,sumisurvey, sumicp, sumird}) of the first author. 
For a random dynamical system generated by a family of 
 polynomial maps on $\CCI $, 
 let $T_{\infty }:\CCI \rightarrow [0,1]$ 
be the function of probability of tending to $\infty \in \CCI .$ 
In \cite{sumiprepare, sumicp, sumird} 
it was shown that under certain conditions, 
$T_{\infty }$ is continuous on $\CCI $ and varies only on the Julia set of the associated rational semigroup 
(further results were announced in \cite{sumisurvey}).  
For example, for a random dynamical system in Remark~\ref{r:t0tinfty}, 
$T_{\infty }$ is continuous on $\CCI $ and the set of varying points of $T_{\infty }$ 
is equal to the Julia set of Figure \ref{fig:ConnOSC1}, 
which is a thin fractal set with Hausdorff dimension strictly less than $2$.  
From this point of view also, it is very interesting and important to 
investigate the figure and the dimension of the Julia sets of rational semigroups. 

 In this paper, 
for an expanding finitely generated rational semigroup $\langle f_{1},\ldots ,f_{m}\rangle $, 
we deal at length with  the relation between 
the Bowen parameter $\delta (f)$ (the unique zero of the pressure function, see Definition~\ref{d:Bpara}) 
of the multimap $f=(f_{1},\ldots ,f_{m})$ 
and the Hausdorff dimension of the Julia set 
of $\langle
f_{1},\ldots ,f_{m}\rangle $.   
In the usual iteration of a single expanding rational map, 
it is well known that the Hausdorff dimension of the Julia set is equal to 
the Bowen parameter and they are strictly less than two. For a general expanding finitely generated rational semigroup 
$\langle f_{1},\ldots ,f_{m}\rangle $,  
it was shown that the Bowen parameter is larger than or equal to the
Hausdorff dimension of the Julia set  
(\cite{sumihyp2,sumi2}). If we assume further that the semigroup satisfies
the ``open set condition'' (see Definition~\ref{d:osc}),  
then it was shown that they are equal (\cite{sumi2}). 
However, if we do not assume the open set condition, then 
there are a lot of examples for which the Bowen  parameter is strictly larger than 
the Hausdorff dimension of the Julia set. In fact, 
the Bowen parameter can be  strictly larger than two (\cite{sumi2, subowen}).  
Thus, it is very natural to ask when we have this situation and what
happens if we have such a case.  
Let Rat be the set of non-constant rational maps on $\CCI $ 
endowed with distance $\overline{d}$ defined by 
$\overline{d}(h_{1},h_{2}):= \sup _{z\in \CCI }\hat{\rho }(h_{1}(z),h_{2}(z))$,  
where $\hat{\rho }$ denotes the spherical distance on $\CCI .$ 
%
For each $m\in \NN $, we set 
$$
\Exp(m):= \{ (g_{1},\ldots ,g_{m})\in (\mbox{Rat})^{m}: \langle
g_{1},\ldots ,g_{m}\rangle \mbox{ is expanding}\} .
$$
Note that $\Exp(m)$ is an open subset of $(\mbox{Rat})^{m}$ (see Lemma~\ref{expopenlem}). 
Let $U$ be a non-empty bounded open subset of $\RR ^{d}$. 
For each $\l \in U$, let $f_{\l }= (f_{\l ,1},\ldots ,f_{\l ,m})$ be an element in 
$\Exp(m)$. We set 
$$
G_{\l }:=\langle f_{\l ,1},\ldots ,f_{\l ,m}\rangle
.
$$   
We assume that the map $\l \mapsto f_{\l ,j}\in \mbox{Rat}, \l \in U,$
is continuous for each $j=1,\ldots ,m.$   
For every $\l \in U$, let $s(\l )$ be the zero of the pressure function for the 
system generated by $f_{\l }.$ Note that the function $\l \mapsto s(\l ), \l \in
U,$ is continuous (see Theorem~\ref{t:delrp}).   
For a family $\{ f_{\l }\} _{\l \in U}$ in $\Exp(m)$,  
we define the {\bf transversality condition} (see Definition~\ref{d:tc}). 
The transversality condition was introduced and investigated 
for a family of contracting IFSs in 
\cite{PeS} (one of first studies of transversality type conditions and applications to Bernoulli convolutions), 
\cite{PS} (case of IFSs in $\RR $), 
\cite{Ram} (case of finite IFSs of similitudes in general Euclidean spaces $\RR ^{d}, d\geq 1$), 
\cite{SSU} (case of infinite hyperbolic or parabolic IFSs in $\RR $), 
\cite{SoUr} (case of finite parabolic IFSs in $\RR $), 
and 
\cite{MiU} (case of skew products and application to Bowen formulas, examples, partial derivative conditions, etc.).  
Among these papers there are several types of definitions of the transversality condition. 
Our definition of the transversality condition is similar to that
given in \cite{SSU}, though in the present paper we work on a family of
semigroups of rational maps which are not contracting and are not
injective.  Note that there are many works of contracting IFSs with overlaps. See the above 
papers and \cite{PPS, F}, etc. Some results of this paper are applicable to the study of contracting IFSs with overlaps 
and infinitely many new examples of contracting families
of IFSs that satisfy the transversality condition are found (see
Theorem~\ref{t:iautc},  
Examples~\ref{ex:iSierpatc}, \ref{ex:SFatc}, \ref{ex:Pentaatc}, \ref{ex:nkun}, Remarks~\ref{r:autc-cci},\ref{r:autcci}).  
 
For any $p\in \NN $, we denote by Leb$_{p}$ the $p$-dimensional Lebesgue
measure on a $p$-dimensional manifold.   
In this paper, we prove the following. 
\begin{thm}[Theorem~\ref{t:tcdimj}]
\label{t:itcdimj}    
Let $\{ f_{\l }\} _{\l \in U}$ be a family in $\Exp(m)$ as above. 
Suppose that $\{ f_{\l }\} _{\l \in U}$ satisfies the transversality condition. Then  
we have all of the following. 
\begin{itemize}
\item[(1)] $\HD (J(G_{\l }))=\min \{ s(\l ),2\} $ for {\em
    Leb}$_{d}$-a.e. $\l \in U$,  where $\HD $ denotes the Hausdorff dimension. 
\item[(2)] For {\em Leb}$_{d}$-a.e. $\l \in \{ \l\in U: s(\l )>2\} $
  we have that  {\em Leb}$_{2}(J(G_{\l }))>0.$  
\end{itemize}
\end{thm} 
It is very interesting to investigate the Hausdorff dimension of the 
exceptional set of parameters in the above theorem. In order to do that, we 
define the {\bf strong transversality condition} (see Definition~\ref{d:stc}), and 
we prove the following. 
\begin{thm}[Theorem~\ref{t:stcmain}]
\label{t:istcmain}
Let $\{ f_{\l }\}_{\l \in U}$ be a family in $\Exp (m)$ as above.
Suppose that $\{ f_{\l }\} _{\l \in U}$ satisfies the strong transversality condition. 
Let $G$ be a subset of $U$.  Let $\xi \geq 0$. 
Suppose $\min \{ \xi , \sup _{\lambda \in G}s(\lambda )\} +d-2\geq 0.$ 
Then we have 
$$\HD (\{ \l \in G: \HD (J(G_{\l }))<\min \{ \xi ,s(\l )\} \} )
\leq \min \{ \xi ,\sup _{\l \in G}s(\l )\} +d-2.
$$ 
\end{thm}
Since $\HD (J(G_{\l }))\leq s(\l )$ for each $\l \in U$, if we
further assume $\sup _{\l \in U}s(\l )<2$ in the above theorem, then  
$$ \HD (\{ \l \in U: \HD (J(G_{\l }))\neq s(\l )\} )<\HD (U)=d.$$ 
It is very important to study sufficient conditions for a family of
expanding semigroups
to satisfy the strong transversality condition. 
Let $U$ be a bounded open subset of $\CC ^{d}$. 
We say that a family $\{ f_{\l }\} _{\l \in U}$ in $\Exp(m)$ as above
is a holomorphic family in $\Exp(m)$  
if $(z,\l ) \mapsto f_{\l ,j}(z)\in \CCI , (z,\l )\in \CCI \times U,$
is holomorphic for each $j.$  
For a holomorphic family in $\Exp(m)$, we define the {\bf analytic
  transversality condition} (see Definition~\ref{d:analtc}).     
We prove the following.
\begin{prop}[Proposition~\ref{p:atctc}]
\label{p:iatctc}
Let $\{ f_{\l }\} _{\l \in U}$ be a holomorphic family in $\Exp(m)$.  
Suppose that $\{ f_{\l }\} _{\l\in U}$ satisfies the 
analytic transversality condition. 
Then for each non-empty, relatively compact, open subset $U'$ of $U$, the 
family $\{ f_{\l }\} _{\l \in U'}$ satisfies the strong transversality
condition and, hence, the transversality condition.  
\end{prop}
By using Proposition~\ref{p:iatctc},  some calculations involving 
partial derivatives of  conjugacy maps with respect to the parameters
(Lemma~\ref{l:hfdiff}--Corollary~\ref{c:hfdiff3}), and  
some  observation about  the combinatorics of  the Julia set
(Lemma~\ref{l:genprin}), we can produce an
abundance  of examples of holomorphic families satisfying the analytic
transversality condition, and hence 
the strong transversality condition and ultimately the
transversality condition.  
Combining the above and some further observations, we prove 
Theorem~\ref{t:id1d2ex} which is formulated below.    
We consider the space 
$$
{\mathcal P}:= \{ g: g \mbox{ is a polynomial},\,
\deg (g)\geq 2\} 
$$ 
endowed with the relative topology from Rat. 
We are interested in families of small perturbations of elements in
the boundary of  
the parameter space ${\mathcal A}$ in $\Exp(m)$, where 
$${\mathcal A}:=
\{ (g_{1},\ldots ,g_{m})\in \Exp(m):  
g_{i}^{-1}(J(\langle g_{1},\ldots ,g_{m}\rangle ))\cap
g_{j}^{-1}(J(\langle g_{1},\ldots ,g_{m}\rangle ))=\emptyset  
\mbox{ if } i\neq j \} .$$ 
\begin{thm}[Theorem~\ref{t:d1d2ex}]
\label{t:id1d2ex}
Let $(d_{1},d_{2})\in \NN ^{2}$ be such that  $d_{1},d_{2}\geq 2$ and
$(d_{1},d_{2})\neq (2,2).$  
Let $b=ue^{i\theta }\in \{ 0<|z|<1\}$, where $0<u<1$ and $\theta \in [0,2\pi ).$ 
Let $\alpha \in [0,2\pi )$ be a number such that 
there exists a number $n\in \ZZ $ with $d_{2}(\pi +\theta )+\alpha
=\theta +2n\pi .$    
Let $\b _{1}(z)=z^{d_{1}}.$ For each $t>0$, 
let $g_{t}(z)=te^{i\alpha }(z-b)^{d_{2}}+b.$ 
Then there exists a point $t_{1}\in (0,\infty )$ and an open 
neighborhood $U$ of $0$ in $\CC $ such that 
the family $\{ f_{\l }=(\b_{1},g_{t_{1}}+\l g_{t_{1}}')\} _{\l \in U}$
with $\l _{0}=0$   
satisfies all of the following conditions {\em (i)--(iv)}. 
\begin{itemize}
\item[(i)] 
$\{ f_{\l }\} _{\l \in U}$ is a holomorphic family in $\Exp(2)$
satisfying the analytic transversality condition,  
the strong transversality condition  
and the transversality condition. 

\item[(ii)] 
For each $\l\in U$, $s(\l )<2$. 

\item[(iii)] 
There exists a subset $\Omega $ of $U$ with 
$\HD (U\setminus \Omega )<\HD (U)=2$ such that for each $\l \in \Omega $, 
$$1<\frac{\log (d_{1}+d_{2})}{\sum
  _{j=1}^{2}\frac{d_{i}}{d_{1}+d_{2}}\log (d_{i})}< \HD (J(G_{\l
}))=s(\l )<2.$$    

\item[(iv)] 
$J(G_{\l _{0}})$ is connected and $\HD (J(G_{\l _{0}}))=s(\l _{0})<2.$ 
Moreover, $G_{\l _{0}}$ satisfies the open set condition.
Furthermore, for each $t\in (0,t_{1})$, the semigroup  
$\langle \b_{1},g_{t}\rangle $ satisfies the open set condition, 
$\b _{1}^{-1}(J(\langle \b _{1},g_{t}\rangle ))\cap 
g_{t}^{-1}(J(\langle \b _{1},g_{t}\rangle ))=\emptyset $, the Julia
set $J(\langle \b _{1},g_{t}\rangle )$ is disconnected, 
and 
$$1<\frac{\log (d_{1}+d_{2})}{\sum _{j=1}^{2}\frac{d_{i}}{d_{1}+d_{2}}\log (d_{i})}< 
\HD (J(\langle \b _{1},g_{t}\rangle ))=\delta (\b _{1},g_{t})<2,$$
where $\delta (\b _{1}, g_{t})$ denotes the Bowen parameter of $(\b
_{1},g_{t}).$   
\end{itemize}
Moreover, there exists an open neighborhood $Y$ of $(\b _{1},g_{t_{1}})$ in ${\mathcal P}^{2}$ 
such that the family $\{ \g=(\g_{1},\g_{2})\} _{\g\in Y}$ 
satisfies all of the following conditions  {\em (v)--(viii)}.
\begin{itemize}
\item[(v)] 
$\{ \g=(\g_{1},\g_{2})\} _{ \g\in Y} $ is a holomorphic family in $\Exp(2)$ 
satisfying the analytic transversality condition,
the strong transversality condition and
 the transversality condition.
\item[(vi)] 
For each $\g\in Y$, $\delta (\g)<2$, where $\delta (\g )$ is the Bowen parameter of $\g =(\g _{1},\g _{2}).$  
\item[(vii)] 
There exists a subset $\Gamma $ of $Y$ with $\HD (Y\setminus \Gamma )<\HD (Y)=2(d_{1}+d_{2}+2)$ 
such that for each $\l \in \Gamma $, 
$$ 1<\frac{\log (d_{1}+d_{2})}{\sum _{j=1}^{2}\frac{d_{i}}{d_{1}+d_{2}}\log (d_{i})}< 
\HD (J(\langle \g_{1} ,\g_{2}\rangle ))=\delta (\g)<2.$$ 
\item[(viii)] 
For each neighborhood $V$ of $(\b _{1},g_{t_{1}})$ in $Y$ there exists a non-empty 
open set $W$ in $V$ such that 
for each $\g =(\g_{1},\g_{2})\in W$, we have that 
$\gamma _{1}^{-1}(J(\langle \gamma _{1},\gamma _{2}\rangle ))\cap \gamma _{2}^{-1}(J(\langle \gamma _{1},\gamma _{2}\rangle ))\neq 
\emptyset $ and that 
$J(\langle \g _{1},\g_{2}\rangle )$ is connected. 
\end{itemize}

\end{thm}
\begin{rem}
\label{r:t0tinfty}
For each $\g =(\g _{1},\g_{2})\in {\mathcal P}^{2}$ and 
$p=(p_{1},p_{2})\in (0,1)^{2}$ with $p_{1}+p_{2}=1$, 
we consider the random dynamical system such that 
for each step, we choose $\g _{i}$ with probability $p_{i}.$ 
For each $z\in \CCI $, let $T_{\infty ,\g ,p} (z)$ be the 
probability of tending to $\infty $ starting with the initial value $z.$  
Then the function $T_{\infty ,\g ,p}:\CCI \rightarrow [0,1] $ is locally constant on
$F(\langle \g _{1},\g_{2}\rangle ).$  
Moreover, this function provides a lot of information about the random
dynamics generated by $(\g ,p).$ (See \cite{sumiprepare, sumird}.)  
Let $\{ f_{\l }\} _{\l \in U}$ be as in Theorem~\ref{t:id1d2ex}. 
Let $\zeta =(\zeta _{1},\zeta _{2})=(f _{\l _{0},1}, f_{\l _{0},2}).$ 
Let $p=(1/2,1/2).$ Then we can show that 
$T_{\infty ,\zeta, p}$ is continuous on $\CCI $ and the set of varying
points of $T_{\infty ,\zeta, p}$ is equal to 
$J(G_{\l _{0}})=J(\langle \zeta _{1},\zeta _{2}\rangle ).$ 
(For the figure of $J(G_{\l _{0}})$, see Figure~\ref{fig:ConnOSC1}.) 
Moreover, there exists a neighborhood $H$ of $(\zeta _{1},\zeta
_{2})$ in ${\mathcal P}^{2}$  
such that for each $\g =(\g _{1},\g _{2})\in H$, 
$T_{\infty ,\g ,p}$ is continuous on $\CCI $ and locally constant
on $F(\langle \g _{1},\g _{2}\rangle ).$  
It is a complex analogue of the devil's staircase and is called a
``devil's coliseum.'' 
(These results are announced in the first author's papers 
\cite{sumisurvey, sumiprepare}.) 
From this point of view also, it is very natural and important to
investigate the Hausdorff dimension of the  
Julia set of a rational semigroup.
\end{rem}
\begin{figure}[htbp]
\caption{The Julia set of the $2$-generator polynomial semigroup 
$G_{\l _{0}}$ with $(d_{1},d_{2})=(3,2), b=0.1$, in Theorem~\ref{t:id1d2ex}. 
$G_{\l _{0}}$ satisfies the open set condition, $J(G_{\l _{0}})$ is connected and 
$\HD (J(G_{\l _{0}}))=s(\l _{0})<2.$}  
\includegraphics[width=1.8cm,width=1.8cm]{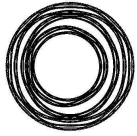}
\label{fig:ConnOSC1}
\end{figure} 

In Theorem~\ref{t:id1d2ex} we deal with $2$-generator polynomial semigroups 
$\langle \g _{1},\g _{2}\rangle $ with $\deg (\g _{1}),$ $ \deg (\g _{2})$ $\geq 2$, 
$(\deg (\g _{1}), \deg (\g _{2}))\neq (2,2)$ for which   
the planar postcritical set is bounded. 
In fact, it is very important to investigate the dynamics of polynomial semigroups with bounded planar 
postcritical set (see \cite{sumid1,sumid2,sumid3,SS}). 
There appear  many new phenomena (for example, the Julia sets of such  semigroups can be disconnected) 
in the dynamics of such semigroups which cannot hold in the usual iteration dynamics of a single 
polynomial. In the proof of Theorem~\ref{t:id1d2ex}, we use some idea from the study of dynamics of such semigroups.
In the family of Theorem~\ref{t:id1d2ex}, for a typical parameter value 
the Hausdorff dimension of the Julia set is strictly less than $2$ and
is equal to the Bowen parameter.  
Thus it is very natural to ask what happens for  polynomial semigroups 
$\langle \g _{1}, \g _{2}\rangle $ with $\deg (\g _{1})=\deg (\g _{2})=2$ 
for which the planar postcritical set is bounded. In this case, by
\cite[Theorem 2.15]{sumid1},  
$J(\langle \g _{1},\g_{2}\rangle )$ is connected and 
$\g_{1}^{-1}(J(\langle \g _{1},\g _{2}\rangle ))\cap \g
_{2}^{-1}(J(\langle \g _{1},\g _{2}\rangle ))\neq \emptyset .$  
Combining Proposition~\ref{p:iatctc} and the lower estimate of the
Bowen parameter from \cite{subowen},  
which was obtained by using thermodynamic formalisms,  potential theory, and 
some results from \cite{zdunik2}, we prove the following. 
\begin{thm}[Corollary~\ref{c:pmz2az2}] 
\label{t:ipmz2az2}
For each $a\in \CC $ with $|a|\neq 0,1$, 
there exists an open neighborhood $Y_{a}$ of $(az^{2}, z^{2})$ in 
${\mathcal P}^{2}$ such that 
$\{ g=(g_{1},g_{2})\} _{g\in Y_{a}} $ is a holomorphic family in $\Exp(2)$ 
satisfying the analytic transversality condition, 
the strong transversality condition and the transversality condition,  
and 
for a.e. $g=(g_{1},g_{2})\in Y_{a}$ with respect to the Lebesgue measure on 
${\mathcal P}^{2}$, we have that 
$\mbox{{\em Leb}}_{2}(J(\langle g_{1},g_{2}\rangle ))>0.$  
\end{thm}
Note that in the usual iteration dynamics of a single expanding
rational map $g$,  
the Hausdorff dimension of the Julia set is strictly less than two. In
particular, Leb$_{2}(J(g))=0.$ 

For any  $a\in \CC $ with $|a|\neq 0,1$, 
$J(\langle az^{2},z^{2}\rangle )$ is equal to the closed annulus between 
$\{ w\in \CC :|w|=1\} $ and $\{ w\in \CC: |w|=|a|^{-1}\} $, thus 
$\mbox{int}(J(\langle az^{2}, z^{2}\rangle ))\neq \emptyset .$ 
However, regarding Theorem~\ref{t:ipmz2az2},  
it is an open problem to determine,  for any other  
parameter value $(g_{1}, g_{2})\in Y_{a}$ with
$\mbox{Leb}_{2}(J(\langle g_{1},g_{2}\rangle ))>0$,  
whether $\mbox{int}(J(\langle g_{1},g_{2}\rangle ))= \emptyset $ or not. 
We have some partial answers though. At least we can show that
for each $a\in \CC $ with $|a|\neq 0,1$ and 
for each neighborhood $W$ of $(az^{2},z^{2})$ in $Y_{a}$ 
there exists a non-empty open subset 
$\tilde{W}$ of $W$ such that 
for each $(\g _{1},\g _{2})\in \tilde{W}$, the Fatou set 
$F(\langle \g _{1},\g _{2}\rangle )$ has at least three connected
components, and  thus the Julia set 
$J(\langle \g _{1},\g _{2}\rangle )$ is not a closed annulus. 
If $a\in \RR $ with $a>0,a\neq 1$, then we can show that 
for each neighborhood $W$ of $(az^{2},z^{2})$ in $Y_{a}$ and for each
$n\in \NN $ with $n\geq 3$,  
there exists a non-empty open subset 
$W_{n}$ of $W$ such that for each $(\g _{1},\g_{2})\in W_{n}$, 
$F(\langle \g _{1}, \g _{2}\rangle )$ has at least $n$ connected components and 
$J(\langle \g _{1},\g _{2}\rangle )$ is not a closed annulus (see Remark~\ref{r:Jholes}).  

\ 

We now consider the expanding semigroups generated by affine maps. 
 Let $m\geq 2.$ 
For each $j=1,\ldots ,m$, 
let $g_{j}(z)=a_{j}z+b_{j}$, 
where $a_{j},b_{j}\in \CC , |a_{j}|>1.$ 
Let $G=\langle g_{1},\ldots ,g_{m}\rangle .$ 
Since $|a_{j}|>1$, $\infty \in F(G).$ 
Hence, by (\ref{bsseq}), 
$J(G)$ is a compact subset of $\CC $ which satisfies 
$J(G)=\bigcup _{j=1}^{m}g_{j}^{-1}(J(G)).$ 
Since $g_{j}^{-1}$ is a contracting similitude on $\CC $, 
it follows that 
$J(G)$ is equal to the self-similar set constructed by the 
family $\{ g_{1}^{-1},\ldots ,g_{m}^{-1}\}$ of contracting similitudes. 
For the definition of self-similar sets, see \cite{F,F0,Ki}. 
Note that the Bowen parameter $\delta (g_{1},\ldots ,g_{m})$ of
$(g_{1},\ldots ,g_{m})$  
is equal to the unique solution of 
the equation $\sum _{i=1}^{m}|a_{i}|^{-t}=1,\, t\geq 0$.  
Thus $\delta (g_{1},\ldots ,g_{m})$ is the similarity dimension of 
$\{ g_{1}^{-1},\ldots ,g_{m}^{-1}\} .$ 
Conversely, any self-similar set constructed by 
a finite family $\{ h_{1},\ldots ,h_{m}\} $ of contracting similitudes on $\CC $ is equal to 
the Julia set of the rational semigroup $\langle h_{1}^{-1},\ldots ,h_{m}^{-1}\rangle .$    
By using Proposition~\ref{p:iatctc} and some calculations of the partial derivatives 
of the conjugacy maps with respect to the parameters, 
we prove the following. 
\begin{thm}[Theorem~\ref{t:autc}]
\label{t:iautc}
Let $m\in \NN $ with $m\geq 2.$ 
For each $i=1,\ldots, m,$, 
let $g_{i}(z)=a_{i}z+b_{i}, $ where 
$a_{i}\in \CC,\, |a_{i}|>1$, $b_{i}\in \CC $. 
Let $G:=\langle g_{1},\ldots ,g_{m}\rangle .$ 
We suppose all of the following conditions hold. 
\begin{itemize}
\item[(i)]
For each $(i,j)$ with $i\neq j$ and $g_{i}^{-1}(J(G))\cap
g_{j}^{-1}(J(G))\neq \emptyset $,   
there exists a number $\alpha _{ij}\in \{ 1,\ldots ,m\} $ 
such that 
$g_{i}(g_{i}^{-1}(J(G))\cap g_{j}^{-1}(J(G)))\subset \{
\frac{-b_{\alpha _{ij}}}{a_{\alpha _{ij}}-1}\} .$ 

\item[(ii)]
If $i,j,k$ are mutually distinct elements in $\{ 1,\ldots ,m\} $, 
then 
$$ g_{k}(g_{i}^{-1}(J(G))\cap g_{j}^{-1}(J(G)))\subset F(G).$$ 

\item[(iii)]
For each $(j,k)$ with $j\neq k$, we have 
$g_{k}(\frac{-b_{j}}{a_{j}-1})\in F(G).$ 
\end{itemize}
Then, there exists an open neighborhood $U$ of 
$(g_{1},\ldots ,g_{m})\in (\mbox{{\em Aut}}(\CC ))^{m}$,  
where {\em Aut}$(\CC ):= \{ az+b: a\in \CC \setminus \{ 0\} ,b\in \CC \}$,  
such that 
$\{ \g =(\g _{1},\ldots ,\g _{m})\} _{ \g \in U} $ 
is a holomorphic family in $\Exp(m)$ satisfying the analytic
transversality condition,  
the strong transversality condition and the transversality condition. 
\end{thm}
Note that in the above theorem, 
for each $j=1,\ldots ,m$, $J(g_{j})=\{ \frac{ -b_{j}}{a_{j}-1}\} .$ 

Note also that even if we replace ``$\mbox{Aut}(\CC )$'' by 
$\mbox{Aut}(\CCI ):=\{ \frac{az+b}{cz+d}: a,b,c,d,\in \CC, ad-bc\neq 0\}$, 
similar results hold (see Remark~\ref{r:autc-cci}).
  
By using Theorem~\ref{t:iautc}, we can obtain many examples of 
families of systems of affine maps satisfying the analytic
transversality condition.  
In fact, we have the following. 
\begin{ex}[Example~\ref{ex:Sierpatc}]
\label{ex:iSierpatc}
Let $p_{1},p_{2},p_{3}\in \CC $ be such that 
$p_{1}p_{2}p_{3}$ makes an equilateral triangle.  
For each $i=1,2,3$, let $g_{i}(z)=2(z-p_{i})+p_{i}.$ 
Let $G=\langle g_{1},g_{2},g_{3}\rangle .$ Then 
$J(G)$ is equal to the Sierpinski gasket.  
It is easy to see that $(g_{1},g_{2},g_{3})$ satisfies the assumptions of Theorem~\ref{t:iautc}. 
Moreover, $\delta (g_{1},g_{2},g_{3})=\HD (J(G))=\frac{\log 3}{\log 2}<2.$ 
By Theorems~\ref{t:iautc}, \ref{t:istcmain} and \ref{t:fundfact1}, 
there exists an open neighborhood $U$ of $(g_{1},g_{2},g_{3})$ 
in $(\mbox{{\em Aut}}(\CC ))^{3}$ and a Borel subset $A$ of $U$ 
with $\HD (U\setminus A)<\HD (U)=12$ 
such that 
{\em (1)} $\{ \g =(\g _{1},\g_{2},\g _{3})\} _{\g \in U}$ is a holomorphic family in $\Exp (3)$ 
satisfying the analytic transversality condition, the strong transversality condition and 
the transversality condition, and {\em (2)} for each $\g =(\g _{1},\g _{2},\g_{3})\in A$, 
$\HD (J(\langle \g _{1},  \g _{2}, \g_{3}\rangle ))=\delta (\g _{1},\g _{2},\g_{3})<2.$  
\end{ex}
For some  other examples including the families related to the Snowflake, Pentakun, 
Hexakun, Heptakun, Octakun and so on,   
see Examples~\ref{ex:m2autatc}, \ref{ex:SFatc}, \ref{ex:Pentaatc}, \ref{ex:nkun} and Remark~\ref{r:autcci}. 
(For the definition of Snowflake, Pentakun, etc., see \cite{Ki}.)  
We remark that, up to our best knowledge, these examples
(Examples~\ref{ex:iSierpatc}, etc.) have not been explicitly dealt with
in any literature of contracting IFSs with overlaps.   

In section~\ref{Preliminaries}, we introduce and collect some
fundamental concepts, notation, and definitions. 
In section~\ref{Results}, we prove the main results of this paper. 
In section~\ref{Applications}, we describe some applications and examples. 
In section~\ref{Remarks}, we make a remark on similar results for
families of conformal contracting iterated function systems in arbitrary dimensions. 
\section{Preliminaries}
\label{Preliminaries}
In this section we introduce notation and basic definitions. 
Throughout the paper, we frequently follow the notation 
from \cite{hiroki1} and \cite{sumi2}. 
\begin{dfn}[\cite{HM,ZR}] 
A ``rational semigroup" $G$ is a semigroup generated by a family of 
non-constant 
rational maps $g:\oc \rightarrow \oc$,\ where $\oc $ denotes the 
Riemann sphere,\ with the semigroup operation being functional 
composition. 
A ``polynomial semigroup'' is a semigroup generated by a 
family of non-constant polynomial maps of $\oc .$ 
For a rational semigroup $G$, we set 
$$
F(G):=\{ z\in \oc : G \mbox{ is normal in some neighborhood of } z\} 
$$
and we call $F(G)$ the {\bf Fatou set} of $G$. Its complement,
$$
J(G):=\oc \setminus F(G)
$$ 
is called the {\bf Julia set} of $G.$ 
If $G$ is generated by a family $\{ f_{i}\} _{i}$ 
(i.e., $G=\{ f_{i_{1}}\circ \cdots \circ f_{i_{n}}: n\in \NN , \forall f_{i_{j}}\in \{ f_{i}\} \}$),\ 
then we write $G=\langle f_{1},f_{2},\ldots \rangle .$ 
For each $g\in \mbox{{\em Rat}}$, we set $F(g):= F(\langle g\rangle )$
and $J(g):=J(\langle g\rangle ).$  
\end{dfn} 
Note that for each $h\in G$, $h(F(G))\subset F(G), h^{-1}(J(G))\subset J(G).$ 
For the fundamental properties of $F(G)$ and $J(G)$, see 
\cite{HM,Stan,hiroki1}. 
For the papers dealing with dynamics of rational semigroups, 
see for example \cite{HM,ZR,Stan,  sumihyp1,sumihyp2,
hiroki1,sumi1, sumi2,sumi06, sumirandom,  suetds1, su1, subowen, 
sumid1, sumid2, sumid3,SS,sumiintcoh, sumiprepare, sumisurvey, sumicp,
sumird}, etc.  

We denote by  Rat the set of all non-constant 
rational maps on $\oc $ endowed with 
distance $\overline{d}$ defined by 
$\overline{d}(h_{1},h_{2}):=\sup _{z\in \CCI }\hat{\rho }(h_{1}(z),h_{2}(z))$,
 where $\hat{\rho }$ denotes the spherical distance on $\CCI .$ 
For each $d\in \NN $, 
we set Rat$_{d}:=\{ g\in \mbox{Rat}: \deg (g)=d\} .$ 
Note that each $\mbox{Rat}_{d}$ is a connected component of Rat. 
Hence Rat has countably many connected 
components. In addition, each connected component 
$\mbox{Rat}_{d}$ of Rat is an open subset of Rat and 
$\mbox{Rat}_{d}$ has a structure of a finite dimensional complex manifold.   
Similarly, we denote by ${\mathcal P}$ the set of all polynomial maps
$g:\oc \rightarrow \oc $  
with $\deg (g)\geq 2$ endowed with the relative topology inherited from Rat. 
We set $\mbox{Aut}(\CC ):= \{ az+b: a,b\in \CC , a\neq 0\} $ endowed
with the relative topology inherited from Rat.  
For each $d\in \NN $ with $d\geq 2$, we set 
${\mathcal P}_{d}:=\{ g\in {\mathcal P}: \deg (g)=d\} .$ 
Note that each ${\mathcal P}_{d}$ is a connected component of 
${\mathcal P}.$ Hence 
${\mathcal P}$ has countably many connected 
components. In addition, each connected component 
${\mathcal P}_{d}$ of ${\mathcal P}$ is an open subset of ${\mathcal P}$ and 
${\mathcal P}_{d}$ has a structure of a finite dimensional complex manifold. 
Moreover, Aut$(\CC )$ is a connected, complex-two-dimensional complex manifold.  
We remark that $g_{n}\rightarrow g$ as $n\rightarrow \infty $ in
${\mathcal P}\cup \mbox{Aut}(\CC )$  
if and only if there exists a number $N\in \NN $ such that 
\begin{itemize}
\item[(i)] $\deg (g_{n})=\deg (g)$ for each $n\geq N$, and 
\item[(ii)] the coefficients of $g_{n} (n\geq N)$ converge to the
coefficients of $g$ appropriately as $n\rightarrow \infty .$ 
\end{itemize}
Thus 
$$
{\mathcal P}_{d}\cong (\CC \setminus \{ 0\} )\times \CC ^{d} \  \text{
  and } \
\mbox{Aut}(\CC )\cong (\CC \setminus \{ 0\} )\times \CC .
$$ 
For more information on the topology and complex structure of Rat and ${\mathcal P}\cup
\mbox{Aut}(\CC )$, the reader may consult \cite{Be}. 

For each $z\in \oc $, we denote by $T\oc _{z}$ the 
complex tangent space of $\oc $ at $z.$ 
 Let $\varphi :V\rightarrow \oc $ be a holomorphic map  
defined on an open set $V$ of $\oc $ and let $z\in V. $  
We denote by $D\varphi _{z}: T\oc _{z}\rightarrow T\CCI _{\varphi (z)}$ 
the derivative of $\varphi $ at $z.$ 
Moreover, we denote by 
 $\| \varphi '(z)\|  $ the norm of the derivative $D\varphi _{z}$ at $z$ 
 with respect to the spherical metric on $\CCI .$ 
\begin{dfn}
For each $m\in \Bbb{N}$, 
let $\Sigma _{m}:=\{ 1,\ldots ,m\} ^{\Bbb{N}}$ be the 
space of one-sided sequences of $m$-symbols endowed with the 
product topology. This is a compact metrizable space. 
For each $f=(f_{1},\ldots ,f_{m})\in (\mbox{{\em Rat}})^{m}$, 
we define a map   
$$
\tilde{f}:\Sg_{m}\times \oc \rightarrow \Sg_{m}\times \oc 
$$ 
by the formula 
$$
\tilde{f}(\om,z)=(\sg (\om ),\ f_{\om_{1}}(z)),
$$
where $(\om,z)\in \Sg _{m}\times \oc,\ \om=(\om_{1},\om_{2},\ldots ),$ and 
$\sg :\Sigma _{m}\rightarrow \Sg _{m}$ denotes the shift map. 
The transformation $\tilde{f} :\Sigma _{m}\times \oc \rightarrow 
\Sigma _{m}\times \oc $ is called the {\bf skew product map} associated 
with the multimap $f=(f_{1},\ldots ,f_{m})\in (\mbox{{\em Rat}})^{m} .$  
We denote by $\pi _{1}:\Sigma _{m}\times \oc \rightarrow \Sigma _{m}$ 
the projection onto $\Sg_{m}$ and by $\pi_{2}:\Sg _{m}\times\oc\rightarrow\oc$ 
the projection onto $\oc $. That is, $\pi _{1}(\om ,z)=\om $ and 
$\pi _{2}(\om ,z)=z.$  For each $n\in \Bbb{N} $ and $(\om ,z)\in 
\Sigma _{m}\times \oc $, we put 
$$
\| (\tilde{f}^{n})'(\om ,z)\| := \| (f_{\om_{n}}\circ \cdots \circ f_{\om _{1}})'(z)\| .
$$ 
We define 
$$J_{\om }(\tilde{f}):=\{ z\in \oc : 
 \{ f_{\om _{n}}\circ \cdots \circ f_{\om _{1}}\} _{n\in \Bbb{N}} \mbox{ is 
 not normal in any neighborhood of } z\} 
$$ 
for each $\om \in  \Sigma _{m}$ and we set 
$$
J(\tilde{f}):= \overline{\cup _{w\in \Sigma _{m}}\{ \om \} 
\times J_{\om }(\tilde{f}) },
$$ 
where the closure is taken with respect to the product topology on the space
$\Sigma _{m}\times \oc .$ $J(\tilde{f})$ is called the 
{\bf Julia set} of the skew product map $\tilde{f}.$ In addition, we set 
$F(\tilde{f}):=(\Sigma _{m}\times \oc )\setminus J(\tilde{f})$ and 
$
\deg(\tilde{f}):=\sum _{j=1}^{m}\deg(f_{j}).
$
We also set $\Sigma _{m}^{\ast }:= \cup _{j=1}^{\infty }
\{ 1,\ldots ,m\} ^{j}$ (disjoint union). 
For each 
$\om\in\Sg_m\cup\Sg_m^*$ let $|\om|$ be the length of $\om .$ 
For each $\om\in\Sg_m\cup\Sg_m^*$ we write $\om =(\om _{1},\om _{2},\ldots ).$ 
For each $f=(f_{1},\ldots ,f_{m})\in (\mbox{{\em Rat}})^{m}$ 
and each $\om =(\om _{1},\ldots ,\om _{n})\in \Sigma _{m}^{\ast }$,  
we put 
$$
f_{\om }:= f_{\om _{n}}\circ \cdots \circ f_{\om _{1}}.
$$ 
For every 
$n\le |\om|$ let $\om|_n=(\om_1,\om_2,\ld ,\om_n)$. If $\om\in\Sg_{m}^{*}$, we put
$$
[\om]=\{\tau\in\Sg_m:\tau|_{|\om|}=\om\}.
$$
If $\om,\tau\in\Sg_m\cup\Sg_m^*$, $\om\wedge\tau$ is the longest initial
subword common for both $\om$ and $\tau$. 
Let $\alpha $ be a fixed number with $0<\alpha <1/2.$ 
We endow the shift space $\Sg_m$
with the distance $\rho_\alpha $ defined as
$
\rho_{\alpha }(\om,\tau)=\alpha ^{|\om\wedge\tau|}
$
with the standard convention that $\alpha ^{\infty}=0$. The distance  
$\rho _\alpha $
induces the product topology on $\Sg_m$. Denote the spherical distance
on $\oc$ by $\hat\rho$ and equip the product
space $\Sg_m\times\oc$ with the distance $\rho$ defined as follows.
$$
\rho((\om,x),(\tau,y))=\max\{\rho_{\alpha }(\om,\tau),\hat\rho(x,y)\}.
$$
Of course $\rho$ induces the product topology on $\Sg_m\times\oc$.
If $\om =(\om _{1},\om _{2},\ldots ,\om _{n})\in \Sigma _{m}^{\ast }$ and 
$\tau =(\tau _{1},\tau _{2},\ldots )\in \Sigma _{m}^{\ast }\cup \Sigma _{m}$, 
we set $\om \tau := (\om _{1},\om _{2},\ldots ,\om _{n},\tau _{1},\tau
_{2},\ldots )\in \Sigma _{m}^{\ast }\cup \Sigma _{m}.$  
For a $j\in \{ 1,\ldots ,m\} $, we set 
$j^{\infty }:= (j,j,j,\ldots )\in \Sigma _{m}.$  
\end{dfn}

\begin{rem}
\label{rem1} 
By definition, the set
 $J(\tilde{f})$ 
 is compact. Furthermore,\ if we set 
 $G=\langle f_{1},\ldots ,f_{m}\rangle $, then, 
 by  \cite[Proposition 3.2]{hiroki1},\ the following hold:  
\begin{enumerate}
\item  
 $J(\tilde{f})$ is completely invariant under $\tilde{f}$;\ 
\item 
 $\tilde{f}$ is an open map on $J(\tilde{f})$;\ 
\item 
if 
 $\sharp J(G)\geq 3$ and $E(G):= \{ z\in \oc : 
 \sharp \cup _{g\in G}g^{-1}(\{ z\} )<\infty \} $ is contained in 
 $F(G)$,  then  the dynamical system $(\tilde{f},J(\tilde{f}))$ is topologically exact; 
\item  
 $J(\tilde{f})$ is equal to the closure of 
 the set of repelling periodic points of 
 $\tilde{f}$ if 
 $\sharp J(G)\geq 3$,\ where we say that a periodic point 
 $(\om ,z)$ of $\tilde{f}$ with 
 period $n$ is repelling if $\| (\tilde{f}^{n})'(\om ,z)\| >1$.    
\item $\pi _{2}(J(\tilde{f}))=J(G).$ 
\end{enumerate}
\end{rem} 
\begin{dfn}[\cite{sumi2}]
A finitely generated 
rational semigroup $G=\langle f_{1},\ldots ,f_{m}\rangle $ 
is said to be expanding provided that 
$J(G)\ne \es $ and the skew product map 
$\tilde{f}:\Sg_{m} \times \oc \rightarrow \Sg _{m}\times \oc $ 
associated with 
$f=(f_{1},\ldots ,f_{m}) $ is expanding along 
fibers of the Julia set $J(\tilde{f})$, 
meaning that there exist $\eta >1$ and 
$C\in(0,1]$ such that for all $n\ge 1$,
\begin{equation}
\label{1112505}
\inf \{ \| (\tilde{f}^{n})'(z)\|: z\in J(\tilde{f})\} \ge C\eta ^{n}.  
\end{equation}  
\end{dfn} 
\begin{dfn}
Let $G$ be a rational semigroup. 
We put
$$
P(G):=\overline{\cup _{g\in G}\{ \mbox{all critical values of }
g:\oc \rightarrow \oc \}} \ (\subset \oc )
$$ 
and we call $P(G)$ the {\bf postcritical set} of $G$. 
A rational semigroup $G$ is said to be {\bf hyperbolic} if 
$P(G)\subset F(G).$ 
\end{dfn}
We remark that if $\Gamma \subset \mbox{Rat}$ and $G$ is generated by $\Gamma $, 
then 
\begin{equation}
\label{eq:PG}
P(G)=\overline{\bigcup _{g\in G\cup \{ Id\} }g(\bigcup _{h\in \Gamma }\{ \mbox{all critical values of } h:\CCI \rightarrow \CCI \} )}. 
\end{equation}
Therefore for each $g\in G$, $g(P(G))\subset P(G).$ 
\begin{dfn}
Let $G$ be a polynomial semigroup. 
We set $P^{\ast }(G):= P(G)\setminus \{ \infty \}.$ 
This set is called the {\bf planar postcritical set} of $G.$ 
We say that $G$ is postcritically bounded if $P^{\ast }(G)$ is bounded in $\C .$ 
\end{dfn}
\begin{rem}
\label{exphyplem}
Let $G=\langle f_{1},\ldots ,f_{m}\rangle $ be a rational semigroup 
such that 
there exists an element $g\in G$ with $\deg (g)\geq 2$ and 
such that each M\"{o}bius transformation in $G$ is loxodromic. 
Then,  
it was proved in \cite{sumihyp2} that 
$G$ is expanding if and only if $G$ is hyperbolic. 
\end{rem}
\begin{dfn}
For each $m\in \NN $, 
we define
$$\Exp (m):=\{ (f_{1},\ldots ,f_{m})\in (\mbox{{\em Rat}})^{m}
: \langle f_{1},\ldots ,f_{m}\rangle \mbox{ is expanding} \}.
$$  
\end{dfn}
Then we have the following.
\begin{lem}[\cite{sumihyp1, suetds1}]
\label{expopenlem}
$\Exp(m)$ is an open subset of {\em (Rat}$)^{m}.$ 
\end{lem}
\begin{lem}[Theorem 2.14 in \cite{sumi1}]
\label{l:jtfwo} 
For each $f=(f_{1},\ldots ,f_{m})\in \Exp(m)$, 
$J(\tilde{f})=\bigcup _{\om \in \Sigma _{m}}(\{\omega\}\times J_{\om }(\tilde{f}))$ and 
$J(\langle f_{1},\ldots ,f_{m}\rangle )=\bigcup _{\om \in \Sigma _{m}}J_{\om }(\tilde{f}).$ 
\end{lem}
\begin{dfn}
We set 
$$\Epb (m):= \{ f=(f_{1},\ldots ,f_{m})\in \Exp (m)\cap {\mathcal P}^{m}
: \langle f_{1},\ldots ,f_{m}\rangle \mbox{ is postcritically bounded} \} .$$ 
\end{dfn}
\begin{lem}[\cite{sumid3, sumiprepare}]
\label{l:epbopen}
$\Epb (m)$ is open in ${\mathcal P}^{m}.$ 
\end{lem}
\begin{dfn}
\label{d:Bpara}
Let $f=(f_{1},\ldots ,f_{m})\in \Exp(m)$ and 
let $\tilde{f}:\Sigma _{m}\times \oc \rightarrow 
\Sigma _{m}\times \oc $ be the skew product map 
associated with $f=(f_{1},\ldots ,f_{m}) .$ 
For each $t\in \R $, 
let $P(t,f)$ be the topological pressure of 
the potential $\varphi(z):= -t\log \| \tilde{f}'(z)\|$ 
with respect to the map 
$\tilde{f}:J(\tilde{f})\rightarrow J(\tilde{f}).$  
(For the definition of the topological pressure, 
see \cite{pubook}.) 
We denote by $\delta (f)$ the unique zero 
of the function $\R\ni t\mapsto P(t,f)\in\R.$ Note that the existence and 
uniqueness of the zero of the function $P(t,f)$ was shown in 
\cite{sumi2}. The number $\delta (f)$ is called the 
{\bf Bowen parameter} of the multimap $f=(f_{1},\ldots ,f_{m})\in 
\Exp(m).$  

Let $u\geq 0$. 
A Borel probability measure $\mu $ on $J(\tilde{f})$ is said to be $u$-conformal for 
$\tilde{f}$ if the following holds.  
For any Borel subset $A$ of $J(\tilde{f})$ such that 
$\tilde{f}|_{A}:A\rightarrow J(\tilde{f})$ is injective, 
we have that 
$$
\mu (\tilde{f}(A))=\int _{A}\| \tilde{f}'(z)\| ^{u} d\mu (z).
$$ 
\end{dfn}

We remark that with the notation of Definition~\ref{d:Bpara}, 
there exists a unique $\delta (f)$-conformal measure for $\tilde{f}$ (see \cite{sumi2}). 
\begin{dfn}
For a subset $A$ of $\oc $, we denote by $\HD (A)$ the Hausdorff dimension of 
$A$ with respect to the spherical distance. 
For each $d\in \NN $, if $B$ is a subset of 
$\RR^ {d}$, 
we denote by $\HD (B)$ the Hausdorff dimension of $B$ with respect to 
the Euclidean distance on $\RR ^{d}.$ 
For a Riemann surface $S$, we denote by $\mbox{{\em Aut}}(S)$ the set of all 
holomorphic isomorphisms of $S.$ 
For a compact metric space $X$, we denote by $C(X)$ the Banach space of all continuous 
complex-valued functions on $X$, endowed with the supremum norm. 
\end{dfn}

A fundamental fact about the Bowen parameter is the following. 
\begin{thm}[\cite{sumi2,sumihyp2}]
\label{t:fundfact1}
For each $f=(f_{1},\ldots ,f_{m})\in \Exp (m)$, 
$\HD (J(\langle f_{1},\ldots ,f_{m}\rangle ))\leq \delta (f).$ 
\end{thm}

Another crucial property of the Bowen parameter is the following fact
proved as one of the main results of \cite{suetds1}. 

\begin{thm}[\cite{suetds1}]
\label{t:delrp}
The function $\Exp(m)\ni f\mapsto \delta (f)\in\R$ is real-analytic and plurisubharmonic.  
\end{thm} 
\begin{rem}[\cite{sumi2,subowen}]
\label{r:eqst}
Let $f=(f_{1},\ldots ,f_{m})\in \Exp(m).$ 
Then there exists a unique equilibrium state $\nu _{f}$ with respect to 
$\tilde{f}:J(\tilde{f})\rightarrow J(\tilde{f})$ for the potential function 
$-\delta (f)\log \| \tilde{f}'(z)\| .$ 
The $\tilde{f}$-invariant probability measure $\nu _{f}$ is equivalent to 
the $\delta (f)$-conformal measure for $\tilde{f}.$ 
We have that $\delta (f)=\frac{h_{\nu _{f}}(\tilde{f})}{\int \log \| \tilde{f}'\| d\nu _{f}}$, 
where $h_{\nu _{f}}(\tilde{f})$ denotes the metric entropy of  
$(\tilde{f},\nu _{f}).$ Moreover, $\delta (f)$ is equal to the ``critical exponent 
 of the Poincar\'{e} series'' of the multimap $f$. For the details, 
 see \cite{sumi2,subowen}. 
\end{rem}

\section{Proofs and Results}
\label{Results}
In this section we state and prove the main results of our paper. 
\begin{dfn}
\label{d:osc}
Let $f=(f_{1},\ldots ,f_{m})\in (\mbox{{\em Rat}})^{m}$
and let $G=\langle f_{1},\ldots ,f_{m}\rangle $.  
Let also $U$ be a non-empty open set in $\oc .$ 
We say that $f$ (or $G$) satisfies the open set condition (with $U$) 
if 
$$
\cup _{j=1}^{m}f_{j}^{-1}(U)\subset U \  \text{  and } \
f_{i}^{-1}(U)\cap f_{j}^{-1}(U)=\emptyset 
$$ 
for each $(i,j)$ with $i\neq j.$ 
There is also a stronger condition. Namely, we say that 
$f$ (or $G$) satisfies the separating open set condition 
(with $U$) if 
$$
\cup _{j=1}^{m}f_{j}^{-1}(U)\subset U \  \text{  and } \
f_{i}^{-1}(\overline{U})\cap f_{j}^{-1}(\overline{U})=\emptyset 
$$ 
for each $(i,j)$ with $i\neq j.$ 
\end{dfn}
We remark that the above concept of ``open set condition'' 
(for ``backward IFSs'') is an analogue 
of the usual open set condition in the theory of IFSs. 

The following theorem is important for our investigations.
\begin{thm}[\cite{sumi2}]
\label{t:deltaHD}
Let $f=(f_{1},\ldots ,f_{m})\in \Exp(m)$. If $f$ satisfies the open set condition,
then 
$\HD (J(\langle f_{1},\ldots ,f_{m}\rangle ))=\delta (f).$
\end{thm}

It is interesting to ask for an  estimate of the Hausdorff dimension
of the Julia set  
of $G$ in the case when it is not known whether $G$ satisfies the open
set condition or not. The goal of our paper is to provide answers to
this question. We start with introducing  the following setting. 

\ 

\noindent {\bf Setting $(\ast )$:} 
Let $d,m\in \N $. Let $U$ be a non-empty bounded open subset of $\R ^{d}.$ 
For each $\lambda \in U$, let 
$f_{\lambda }=(f_{\lambda ,1},\ldots ,f_{\lambda ,m})\in \Exp (m)$ 
and let $G_{\l }:= \langle f_{\l ,1},\ldots ,f_{\l ,m}\rangle .$ 
We suppose that 
$\{ f_{\l }\} _{\l\in U}$ is a continuous family of $\Exp (m)$, 
i.e., the map 
$U\ni\lambda \mapsto f_{\l }\in \Exp (m)$ is continuous. Fix a parameter
$\l _{0}\in U$. Suppose that for each $\l \in U$, 
there exists a homeomorphism $h_{\l }:J(\tilde{f}_{\l _{0}})\rightarrow 
J(\tilde{f}_{\l })$ of the form $h_{\l }(\om,z)=(\om,\overline{h}_{\l }(\om,z))$ such that  
$h_{\l _{0}}=Id|_{J(\tilde{f}_{\l _{0}})}$, 
$h_{\l }\circ \tilde{f}_{\l _{0}}=\tilde{f}_{\l }\circ h_{\l }$ on $J(\tilde{f}_{\l _{0}})$,  
 and such that 
the map $(\om,z,\l) \mapsto \overline{h}_{\l }(\om,z)\in \CCI , 
(\om,z,\l) \in J(\tilde{f} _{\l _{0}})\times U$, is continuous. 
The point $\l _{0}$ is called the base point of $\{ f_{\l }\} _{\l \in U}.$ 
Let $C>0, \eta >1$ be such that for each $n\in \NN $, 
$\inf _{(\om,z)\in J(\tilde{f}_{\l _{0}})}\| (\tilde{f}^{n}_{\l _{0}})'(\om,z)\| \geq C\eta ^{n}.$ 
For each $\l \in U$, we set $s(\l ):=\delta (f_{\l })$, where $\delta (f_{\l })$ is the Bowen parameter of 
the multimap $f_{\l }.$   

\ 

 We now will explain (in Definition~\ref{d:holofam} and Remark~\ref{r:setting}) 
that Setting $(\ast )$ is natural.
\begin{dfn}
\label{d:holofam}
Let $M$ be a finite dimensional complex manifold. 
Let $m\in \NN .$ 
For each $\l \in M$, let 
$f_{\l }=(f_{\l ,1},\ldots ,f_{\l, m})$ be an element of $\Exp(m).$ 
We say that $\{ f_{\l }\} _{\l \in M}$ is a holomorphic family in  
$\Exp (m) $ over $M$ if 
the map $\l \mapsto f_{\l }\in \Exp (m), \l \in M$, is holomorphic. 
If a holomorphic family $\{ f_{\l }\} _{\l \in M}$ in $\Exp(m)$ 
satisfies that $f_{\l }\in \Epb(m)$ for each $\l \in M$, 
then we say that $\{ f_{\l }\} _{\l \in M}$ is a holomorphic family in $\Epb(m).$ 

\end{dfn}
\begin{rem}
\label{r:setting}
Let $\{ f_{\l }\} _{\l \in M}$ be a holomorphic family in $\Exp(m)$ over a complex manifold $M$ and 
let $\l _{0}\in M.$ 
Then there exists a neighborhood $U$ of $\l _{0}$ 
such that for the holomorphic family $\{ f_{\l }\} _{\l \in U}$ over $U$, 
there exists a unique family $\{ h_{\l }\} _{\l \in U}$ of conjugacy
maps  as in Setting $(\ast ).$  
Moreover, $\l \mapsto \overline{h}_{\l }(\om ,z)$ is holomorphic. 
For the proof of this result, 
see \cite[Theorem 4.9, Lemma 6.2]{suetds1} and its proof (in fact, the
assumption `` $f$ is simple'' in  
\cite[Theorem 4.9]{suetds1} is not needed).   
\end{rem}

\begin{rem}
\label{r:conjG}
Let $\{ f_{\l }\} _{\l \in M}$ be a holomorphic family in 
$\Exp (m)$ over $M$ and let $\l _{0}\in M$. 
Since the map $\l \mapsto J(G_{\l })$ is continuous with respect to 
the Hausdorff metric (\cite[Theorem 2.3.4]{sumihyp1}, \cite[Lemma 4.1]{suetds1}), 
there exist a M\"{o}bius transformation $\alpha $, 
an open neighborhood $U$ of $\l _{0}$, 
and a compact subset $K$ of $\CC $  
such that setting $\tilde{G}_{\l }:=\{ \alpha \circ  g\circ \alpha
^{-1}: g\in G_{\l }\} $  
for each $\l \in U$, 
we have $J(\tilde{G}_{\l })\subset K $ for each $\l \in U.$ 
\end{rem}

From Lemma~\ref{l:vd} through Theorem~\ref{t:tcdimj}, 
we assume Setting $(\ast )$. 

\vspace{3mm}\noindent {\bf Notation:} 
For a $x\in \RR ^{d}$ and $r>0$, we denote by 
$B_{r}(x)$ the open $r$-ball with center $x$ with respect to the
Euclidean distance.  
For a $y\in \CC $ and $r>0$ we set $D_{r}(y):= \{ z\in \CC : |z-y|<r\} .$  
We denote by Leb$_{d}$ the $d$-dimensional Lebesgue measure on a
$d$-dimensional manifold.  

Under Setting $(\ast )$, the following lemma is immediate. 
\begin{lem}
\label{l:vd}
Let $s,\epsilon >0$ be given with $s>\epsilon .$ 
Then there exist constants $v>0$ and $\delta >0$ such that 
for any $(\om ,z,\om ',z',\l )\in J(\tilde{f}_{\l _{0}})^{2}\times U$, 
if $\rho ((\om,z),(\om',z'))<v $ and $\l \in B_{\delta }(\l _{0})$,
then 
\begin{itemize}
\item $(\eta ^{\frac{3\epsilon }{4(s-\epsilon )}})^{-1}\leq 
\frac{\| \tilde{f}_{\l }'(\om ', z')\| }{\| \tilde{f}_{\l _{0}}'(\om ,z)\| }\leq 
 \min \{ \eta ^{\frac{3\epsilon }{4(s-\epsilon )}}, \eta ^{\frac{\epsilon }{4}}\} $ and 
\item 
$\hat{\rho }(z,\overline{h}_{\l }(\om ,z))<\frac{1}{2}v.$  

\end{itemize} 

\end{lem}
 
We now give the definition of the transversality condition, the
concept of our primary interests in this paper.
    
\begin{dfn}
\label{d:tc}
Let $\{ f_{\l }\} _{\l \in U}$ be as in Setting $(\ast ).$ 
We say that $\{ f_{\l }\} _{\l \in U}$ 
satisfies the transversality condition {\em (TC)} 
if there exists a constant $C_{1}>0$ such that for each $r\in
(0,\mbox{{\em diam}}(\CCI ))$ and  
for each $(\om ,z),(\om ',z')\in J(\tilde{f}_{\l _{0}})$ with 
$\om _{1}\neq \om '_{1}$, 
\begin{equation}
\label{eq:tc} 
\mbox{{\em Leb}}_{d}(\{ \l \in U: \hat{\rho }(\overline{h}_{\l }(\om
,z),\overline{h}_{\l }(\om ',z')) 
\leq r\} )\leq  C_{1}r^{2}.
\end{equation} 
\end{dfn}
\begin{rem}
\label{r:tccon}
If $\{ f_{\l }\} _{\l \in U}$ with base $\l _{0}\in U$ satisfies the
transversality condition,  
then for any $\l _{1}\in U$, 
the family $\{ f_{\l }\} _{\l\in U}$ with base $\l _{1}$ satisfies the
transversality condition  
with the same constant $C_{1}$ (we just consider the family $\{ h_{\l
}h_{\l _{1}}^{-1}\} _{\l \in U}$  
of conjugacy maps). 
\end{rem}
\begin{lem}
\label{l:tcintc}
Suppose that $\{ f_{\l }\} _{\l \in U}$ satisfies the transversality condition. 
Let $\alpha \in (0,2).$ 
Then there exists a constant $C_{2}>0$ such that 
for each $(\om ,z),(\om ',z')\in J(\tilde{f}_{\l _{0}})$ with 
$\om _{1}\neq \om '_{1}$, 
$$
\int _{U}\frac{d\l }{\hat{\rho }(\overline{h}_{\l }(\om
  ,z),\overline{h}_{\l }(\om ',z'))^{\alpha }} 
\leq C_{2}.
$$ 
\end{lem}
\begin{proof}
Let $(\om ,z),(\om ',z')\in J(\tilde{f}_{\l _{0}})$ with $\om _{1}\neq
\om '_{1}.$ Then  
\begin{align*}
\ & \int _{U}\frac{d\l }{\hat{\rho }(\overline{h}_{\l }(\om
  ,z),\overline{h}_{\l }(\om ',z'))^{\alpha }}=\\  
= & \int _{0}^{\infty }\mbox{Leb}_{d}\left(\left\{ \l \in U: 
\frac{1}{\hat{\rho }(\overline{h}_{\l }(\om ,z),\overline{h}_{\l }(\om
  ', z'))^{\alpha }}\geq x\right\} \right) dx\\  
= & \alpha \int _{0}^{\infty }\mbox{Leb}_{d}(\{ \l \in U: 
\hat{\rho }(\overline{h}_{\l }(\om ,z), \overline{h}_{\l }(\om ',
z'))\leq r\} )r^{-\alpha -1} dr\\  
= & \alpha \int _{0}^{\mbox{diam}(\CCI ) }\mbox{Leb}_{d}(\{ \l \in U: 
\hat{\rho }(\overline{h}_{\l }(\om ,z), \overline{h}_{\l }(\om ',
z'))\leq r\} )r^{-\alpha -1} dr \\  
\ & \ \ + \alpha \int _{\mbox{diam}(\CCI )}^{\infty }\mbox{Leb}_{d}(\{ \l \in U: 
\hat{\rho }(\overline{h}_{\l }(\om ,z), \overline{h}_{\l }(\om ',
z'))\leq r\} )r^{-\alpha -1} dr\\  
\leq & \alpha \left( \int _{0}^{\mbox{diam}(\CCI )}C_{1}r^{2}\cdot r^{-\alpha -1}dr 
 + \mbox{Leb}_{d}(U)[\frac{1}{-\alpha }r^{-\alpha }]_{\mbox{diam}(\CCI )}^{\infty }\right) \\ 
= & \alpha \left( \frac{C_{1}}{2-\alpha }(\mbox{diam}(\CCI ))^{2-\alpha }
 +\mbox{Leb}_{d}(U)(\frac{1}{\alpha }(\mbox{diam}(\CCI ))^{-\alpha })\right). 
\end{align*}
Thus we have proved our lemma. 
\end{proof}
\begin{lem}
\label{l:djms2e}
Suppose that  $\{ f_{\l }\} _{\l \in U}$ satisfies the transversality condition. 
Then for each $\l _{1}\in U$ and for each $\epsilon >0$, there exists $\delta >0$ 
such that for {\em Leb}$_{d}$-a.e. $\l \in B_{\delta }(\l _{1})$, 
$\HD (J(G_{\l }))\geq \min \{ s(\l _{1}),2\} -\epsilon .$ 
\end{lem}
\begin{proof}
We may assume that $\l _{1}=\l _{0}.$ 
Since $\l \mapsto J(G_{\l })$ is continuous 
with respect to the Hausdorff metric in the space of all non-empty compact subsets of 
$\CCI $ (\cite[Theorem 2.3.4]{sumihyp1}, \cite[Lemma 4.1]{suetds1}), 
by conjugating $G_{\l _{0}}$ with a M\"{o}bius transformation, 
we may assume without loss of generality that there exists a compact
subset $K$ of $\CC $ such that  
for each $\l $ in a small neighborhood of $\l _{0}$, $J(G_{\l })\subset K.$ 
Let $s:=\min \{ s(\l _{0}),2\} .$ 
Let $\epsilon >0$ with $\epsilon <s.$ 
For this pair $(\epsilon ,s)$, let $v,\delta >0$ be as in
Lemma~\ref{l:vd}.  
We may assume that $v$ is small enough. 
Let $\mu $ be the $s(\l _{0})$-conformal measure for $\tilde{f}_{\l _{0}}.$ 
Let $\mu _{2}:=\mu \otimes \mu $. This is a Borel probability measure on 
$J(\tilde{f}_{\l _{0}})^{2}.$ 
For each $\l \in U$, let 
$$
R(\l ):=\int _{J(\tilde{f}_{\l _{0}})^{2}}\frac{d\mu _{2}(\om ,z,\om ', z')}
{|\overline{h}_{\l }(\om ,z)-\overline{h}_{\l }(\om ',
  z')|^{s-\epsilon }}.
$$ 
By \cite[Theorem 4.13]{F}, it suffices to show that 
\begin{equation}
\label{eq:rlfin}
R(\l )<\infty \mbox{ for Leb}_{d}\mbox{-a.e. } \l \in B_{\delta }(\l _{0}). 
\end{equation}
In order to prove (\ref{eq:rlfin}), assuming $v$ is small enough, 
for each $(\om ,z,\om ',z')\in J(\tilde{f}_{\l _{0}})^{2}$ with 
$(\om ,z)\neq (\om ',z')$, 
let $n=n(\om ,z,\om ', z')\in \NN \cup \{ 0\} $ be the minimum number 
such that 
$$
\mbox{either }\ |\pi _{2}(\tilde{f}_{\l _{0}}^{n}(\om ,z)) -\pi _{2}(\tilde{f}_{\l
  _{0}}^{n}(\om ',z') ) |\geq v \ \ 
\text{{\rm or }}  \ 
\om _{n+1}\neq \om '_{n+1}.
$$ 
For each $n\in \NN \cup \{ 0\} $, 
let $E_{n}:=\{ (\om ,z,\om ', z')\in J(\tilde{f}_{\l _{0}})^{2}:  
n(\om ,z,\om ',z')=n\} .$ 
Let $H:= \{ (\om ,z,\om ',z')\in J(\tilde{f}_{\l _{0}})^{2}:  
(\om ,z)=(\om ',z')\} .$ 
Then 
we have 
$J(\tilde{f}_{\l _{0}})^{2}=H\amalg \amalg _{n\geq 0}E_{n}$ (disjoint union).  
We obtain that  
\begin{align*}
\mu _{2}(H) & =  \int _{J(\tilde{f}_{\l _{0}})}\mu (\{ (\om ',z')\in J(\tilde{f}_{\l _{0}}):  
(\om ,z, \om ', z')\in H\} )d\mu (\om ,z)\\ 
& = \int_{J(\tilde{f}_{\l _{0}})}\mu (\{ (w,z)\} ) d\mu (\om ,z)=0. 
\end{align*}
Hence, by Lemma~\ref{l:vd} and Koebe's distortion theorem, we obtain that 
\begin{align*}
\int _{B_{\delta }(\l_{0})}R(\l )d\l & = 
\int _{B_{\delta }(\l_{0})}d\l \int _{J(\tilde{f}_{\l _{0}})^{2}}
\frac{d\mu _{2}(\om ,z,\om ',z')}{|\overline{h}_{\l }(\om ,z)-\overline{h}_{\l }(\om ',z')|^{s-\epsilon }}\\ 
& = \sum _{n=0}^{\infty }\int _{E_{n}}d\mu _{2}(\om ,z,\om ',z')
\int _{B_{\delta }(\l _{0})}\frac{d\l }{|\overline{h}_{\l }(\om ,z)-\overline{h}_{\l }(\om ',z')|^{s-\epsilon }}\\ 
& \leq \sum _{n=0}^{\infty }\int _{E_{n}}d\mu _{2}(\om ,z,\om ',z')
\int _{B_{\delta }(\l _{0})}\frac{\mbox{Const.}\| (f_{\l ,\om |_{n}})'(\overline{h}_{\l }(\om ,z))\| ^{s-\epsilon}d\l }
{|\overline{h}_{\l }(\tilde{f}_{\l _{0}}^{n}(\om ,z))-\overline{h}_{\l }(\tilde{f}_{\l _{0}}^{n}(\om ',z'))|^{s-\epsilon }}\\ 
& \leq \sum _{n=0}^{\infty }\int _{E_{n}}d\mu _{2}(\om ,z,\om ',z')
\int _{B_{\delta }(\l_{0})}\frac{\mbox{Const.}\| (f_{\l _{0},\om |_{n}})'(z)\| ^{s-\epsilon }
(\eta ^{\frac{3\epsilon }{4(s-\epsilon )}})^{(s-\epsilon )n} d\l }
{|\overline{h}_{\l }(\tilde{f}_{\l _{0}}^{n}(\om ,z))-\overline{h}_{\l }(\tilde{f}_{\l _{0}}^{n}(\om ',z'))|^{s-\epsilon }}\\ 
& = \sum _{n=0}^{\infty }\int _{E_{n}}d\mu _{2}(\om ,z,\om ',z')
\int _{B_{\delta }(\l_{0})}\frac{\mbox{Const.}\| (\tilde{f}_{\l _{0}}^{n})'(\om ,z)\| ^{s-\frac{\epsilon }{4}}
\| (\tilde{f}_{\l _{0}}^{n})'(\om ,z)\| ^{-\frac{3\epsilon }{4}}(\eta ^{\frac{3\epsilon }{4}})^{n}d\l }
{|\overline{h}_{\l }(\tilde{f}_{\l _{0}}^{n}(\om ,z))-\overline{h}_{\l }(\tilde{f}_{\l _{0}}^{n}(\om ',z'))|^{s-\epsilon }}\\ 
& \leq \sum _{n=0}^{\infty }\int _{E_{n}}d\mu _{2}(\om ,z,\om ',z')
\int _{B_{\delta }(\l _{0})} \frac{\mbox{Const.}\| (\tilde{f}_{\l _{0}}^{n})'(\om ,z)\| ^{s-\frac{\epsilon }{4}}d\l }
{|\overline{h}_{\l }(\tilde{f}_{\l _{0}}^{n}(\om ,z))-\overline{h}_{\l }(\tilde{f}_{\l _{0}}^{n}(\om ',z'))|^{s-\epsilon }},
\end{align*}
where Const. denotes a constant although all Const. above may be
mutually different, and $f_{\l _{0},\om |_{0}}=\mbox{Id}.$   
By Lemma~\ref{l:tcintc}, 
it follows that 
\begin{align*}
\int _{B_{\delta }(\l_{0})}R(\l )d\l & \leq 
\mbox{Const.} \sum _{n=0}^{\infty }\int _{E_{n}}\| (\tilde{f}_{\l
  _{0}}^{n})'(\om ,z)\| ^{s-\frac{\epsilon }{4}} 
d\mu _{2}(\om ,z, \om ',z')\\ 
& \leq \mbox{Const.} \sum _{n=0}^{\infty }(C\eta ^{n})^{-\frac{\epsilon }{4}}
\int _{E_{n}}\| (\tilde{f}_{\l _{0}}^{n})'(\om ,z)\| ^{s(\l _{0})}d\mu _{2}(\om ,z, \om ',z')\\ 
& = \mbox{Const.} \sum _{n=0}^{\infty }(C\eta ^{-\frac{\epsilon }{4}n})
\int _{J(\tilde{f}_{\l _{0}})}d\mu (\om ,z)\int _{E_{n,\om ,z}}\| (\tilde{f}_{\l _{0}}^{n})'(\om ,z)\| ^{s(\l _{0})}d\mu (\om ',z')\\ 
& = \mbox{Const.} \sum _{n=0}^{\infty }C\eta ^{-\frac{\epsilon }{4}n}
\int _{J(\tilde{f}_{\l _{0}})}(\| (\tilde{f}_{\l _{0}}^{n})'(\om ,z)\| ^{s(\l _{0})}\mu (E_{n,\om ,z}))d\mu (\om ,z), 
\end{align*} 
where $E_{n,\om ,z}:= \{ (\om ',z')\in J(\tilde{f}_{\l _{0}}): (\om
,z,\om ',z')\in E_{n}\} .$  
As, by Koebe's distortion theorem,
$\| (\tilde{f}_{\l _{0}}^{n})'(\om ,z)\| ^{s(\l _{0})}\mu (E_{n,\om
  ,z})$ is comparable with  
$\mu (\tilde{f}_{\l _{0}}^{n}(E_{n,\om ,z}))$, we 
therefore, obtain that 
\begin{align*}
\int _{B_{\delta }(\l _{0})}R(\l )d\l 
& \leq \mbox{Const.}\sum _{n=0}^{\infty }C\eta ^{-\frac{\epsilon }{4}n}<\infty .
\end{align*}
Hence, (\ref{eq:rlfin}) holds. Thus, we have proved Lemma~\ref{l:djms2e}.
\end{proof} 

\begin{lem}
\label{l:posmeas}
Suppose that $\{ f_{\l }\} _{\l \in U}$ satisfies the transversality condition. 
Suppose $s(\l _{0})>2.$ Let $\mu $ be the $s(\l _{0})$-conformal
measure on $J(\tilde{f}_{\l _{0}})$  
for $\tilde{f}_{\l _{0}}.$ 
Then there exists $\delta >0$ such that 
for $\mbox{{\em Leb}}_{d}$-a.e. $\l \in B_{\delta }(\l _{0})$, 
the Borel probability measure $(\overline{h}_{\l })_{\ast }(\mu )$ on $J(G_{\l })$ 
is absolutely continuous 
with  respect to {\em Leb}$_{2}$ with $L^{2}$ density and 
$\mbox{{\em Leb}}_{2}(J(G_{\l }))>0.$   
\end{lem}
\begin{proof}
As in the proof of Lemma~\ref{l:djms2e}, 
we may assume that there exists a compact subset $K_{0}$ of $\CC $ such that 
for each $\l \in U$, $J(G_{\l })\subset K_{0}.$ 
Take an $\epsilon >0$ with $s(\l _{0})-\epsilon >2.$ 
For this $\epsilon $ and $s=s(\l _{0})$, take a couple $(v, \delta )$
coming from Lemma~\ref{l:vd}.  
We use the notation and the arguments from the proof of Lemma~\ref{l:djms2e}. 
For each $\l \in B_{\delta }(\l _{0})$, 
let $\nu _{\l }:= (\overline{h}_{\l })_{\ast }(\mu ).$ 
Then supp $\nu _{\l }\subset J(G_{\l }).$ 
It is enough to show that $\nu _{\l }$ is absolutely continuous with respect to 
Leb$_{2}$ with $L^{2}$ density for Leb$_{d}$-a.e. $\l \in B_{\delta }(\l _{0}).$ 
In order to do that, we set 
$$
{\mathcal I}:= \int _{B_{\delta }(t_{0})}d\l \int _{\CC
}\underline{D}(\nu _{\l },x)d\nu _{\l }(x),
$$ 
where 
$$
\underline{D}(\nu _{\l },x):=\liminf _{r\rightarrow 0}\frac{\nu _{\l
  }(B(x,r))}{r^{2}}.
$$ 
We remark that if ${\mathcal I}<\infty $, 
then by \cite[p.36, p.43]{Ma}, 
for Leb$_{d}$-a.e. $\l \in B_{\delta }(\l _{0})$, 
$\nu _{\l }$ is absolutely continuous with respect to 
Leb$_{2}$ with $L^{2}$ density.  
Therefore, it is enough to show that ${\mathcal I}<\infty .$ 
In order to do that, by Fatou's lemma, we have 
\begin{equation}
\label{eq:IFatou}
{\mathcal I}\leq \liminf _{r\rightarrow 0}\int _{B_{\delta }(\l _{0})}\int _{\CC }
\frac{\nu _{\l }(B(x,r))}{r^{2}}d\nu _{\l }(x) d\l .
\end{equation}  
Moreover, we have 
$$ \int _{\CC }\nu _{\l }(B(x,r))d\nu _{\l }(x)=
\int _{J(\tilde{f}_{\l _{0}})^{2}}
1_{ \{ (\om, z,\om ', z')\in J(\tilde{f}_{\l _{0}})^{2}: 
|\overline{h}_{\l }(\om ,z)-\overline{h}_{\l }(\om ',z')|< r\} } d\mu _{2}(\om ,z,\om ',z'),$$ 
where $1_{A}$ denotes the characteristic function with respect to the set $A$, and $\mu _{2}:= \mu \otimes \mu .$ 
Hence, by using (\ref{eq:IFatou}),  
we obtain that 
\begin{align*}
{\mathcal I} & \leq \liminf _{r\rightarrow 0}\frac{1}{r^{2}}
\int _{J(\tilde{f}_{\l _{0}})^{2}}\mbox{Leb}_{d}(\{ \l \in B_{\delta }(\l _{0}): 
|\overline{h}_{\l }(\om ,z)-\overline{h}_{\l }(\om ',z')|<r\} )d\mu _{2}(\om ,z,\om ',z')\\ 
& = \liminf _{r\rightarrow 0} \frac{1}{r^{2}}\sum _{n=0}^{\infty }
\int _{E_{n}}\mbox{Leb}_{d}( \{ \l \in B_{\delta }(\l _{0}): 
|\overline{h}_{\l }(\om ,z)-\overline{h}_{\l }(\om ',z')|<r\} )d\mu _{2}(\om ,z,\om ',z').
\end{align*}
By Koebe's distortion theorem 
(we take $v$ and $\delta $ sufficiently
small), there exists a constant $K>0$ such that  
for each $n\in \NN \cup \{ 0\} $, for each $(\om ,z,\om ',z')\in E_{n}$ and for each $\l \in B_{\delta }(\l _{0})$, 
$$ |\overline{h}_{\l }(\om ,z)-\overline{h}_{\l }(\om ',z')|
\geq K \|(f_{\l ,\om |_{n}})'(z)\|^{-1}|\overline{h}_{\l }(\tilde{f}_{\l _{0}}^{n}(\om ,z))
-\overline{h}_{\l }(\tilde{f}_{\l _{0}}^{n}(\om ',z'))|.$$ 
Therefore, by Lemma~\ref{l:vd}, 
for each $n\in \NN \cup \{ 0\} $, for each $(\om ,z,\om ',z')\in E_{n}$ and for each $\l \in B_{\delta }(\l _{0})$, 
\begin{align*}
|\overline{h}_{\l }(\om ,z)-\overline{h}_{\l }(\om ',z')| 
& \geq K\| (\tilde{f}_{\l _{0}}^{n})'(\om ,z)\| ^{-1}(\eta ^{\frac{\epsilon }{4}})^{-n}
|\overline{h}_{\l }(\tilde{f}_{\l _{0}}^{n}(\om ,z))-\overline{h}_{\l }(\tilde{f}_{\l _{0}}(\om ',z'))|\\ 
& \geq K \| (\tilde{f}_{\l _{0}}^{n})'(\om ,z)\| ^{-1-\frac{\epsilon }{4}}(C\eta ^{n})^{\frac{\epsilon }{4}}
\eta ^{-\frac{\epsilon}{4}n}|\overline{h}_{\l }(\tilde{f}_{\l _{0}}^{n}(\om ,z))-\overline{h}_{\l }(\tilde{f}_{\l _{0}}(\om ',z'))|\\ 
& \geq KC^{\frac{\epsilon }{4}}\| (\tilde{f}_{\l _{0}}^{n})'(\om ,z)\| ^{-1-\frac{\epsilon }{4}}
|\overline{h}_{\l }(\tilde{f}_{\l _{0}}^{n}(\om ,z))-\overline{h}_{\l }(\tilde{f}_{\l _{0}}(\om ',z'))|.
\end{align*}
Hence, by transversality condition, for each $n$ and for each $(\om ,z,\om ',z')\in E_{n}$, 
\begin{align*}
\ & \ \ \ \  \mbox{Leb}_{d}(\{ \l \in B_{\delta }(\l _{0}): |\overline{h}_{\l }(\om ,z)-\overline{h}_{\l }(\om ',z')|<r\}) \\ 
& \leq 
\mbox{Leb}_{d}(\{ \l \in B_{\delta }(\l _{0}):
|\overline{h}_{\l }(\tilde{f}_{\l _{0}}^{n}(\om ,z))-\overline{h}_{\l }(\tilde{f}_{\l _{0}}^{n}(\om ',z'))|
\leq (KC^{\frac{\epsilon }{4}})^{-1}r\| (\tilde{f}_{\l _{0}}^{n})'(\om ,z)\| ^{1+\frac{\epsilon }{4}}\} )\\ 
& \leq \mbox{Const.} r^{2}\|(\tilde{f}_{\l _{0}}^{n})'(\om ,z) \|^{2+\frac{\epsilon }{2}}.
\end{align*}
Therefore, 
\begin{align*}
{\mathcal I} & \leq \mbox{Const.} \sum _{n=0}^{\infty }
\int _{E_{n}}\| (\tilde{f}_{\l _{0}}^{n})'(\om ,z)\|
^{2+\frac{\epsilon }{2}}d\mu _{2}(\om ,z,\om ',z')\\  
& =\mbox{Const.} \sum _{n=0}^{\infty }\int _{J(\tilde{f}_{\l _{0}})}d\mu (\om ,z)
\int _{E_{n,\om ,z}}\| (\tilde{f}_{\l _{0}}^{n})'(\om ,z)\|
^{2+\frac{\epsilon }{2}}d\mu (\om ',z'),  
\end{align*}
where $E_{n,\om ,z}=\{ (\om ',z')\in J(\tilde{f}_{\l _{0}}): (\om ,z,\om ',z')\in E_{n}\} .$ 
Thus, 
\begin{align*}
{\mathcal I} & \leq 
\mbox{Const.} \sum _{n=0}^{\infty }\int _{J(\tilde{f}_{\l _{0}})}
(\| (\tilde{f}_{\l _{0}}^{n})'(\om ,z)\| ^{s(\l _{0})} \cdot \mu (E_{n,\om ,z}))\cdot 
\| \tilde{f}_{\l _{0}}^{n}(\om ,z)\| ^{-\frac{\epsilon }{2}} d\mu (\om ,z)\\ 
& \leq  \mbox{Const.} \sum _{n=0}^{\infty }(C\eta ^{n})^{\frac{-\epsilon }{2}}<\infty .
\end{align*}
Hence we have proved Lemma~\ref{l:posmeas}. 
\end{proof}
\begin{thm}
\label{t:tcdimj}
Let $\{ f_{\l }\} _{\l \in U}$ be a family in $\Exp (m)$ satisfying Setting $(\ast ).$  
Suppose that $\{ f_{\l }\} _{\l \in U}$ satisfies the transversality condition. 
Let $\mu $ be the $s(\l _{0})$-conformal measure on $J(\tilde{f}_{\l _{0}})$ 
for $\tilde{f}_{\l _{0}}.$ 
Then we have the following. 
\begin{itemize}
\item[(1)]
$\HD (J(G_{\l }))=\min \{ s(\l ),2\} $ for 
{\em Leb}$_{d}$-a.e. $\l \in U.$ 
\item[(2)] 
For {\em Leb}$_{d}$-a.e. $\l \in \{ \l \in U: s(\l )>2\} $,  
the Borel probability measure $(\overline{h}_{\l })_{\ast }(\mu )$ on $J(G_{\l })$ 
is absolutely continuous 
with  respect to Lebesgue measure {\em Leb}$_{2}$ with $L^{2}$ density and 
$\mbox{{\em Leb}}_{2}(J(G_{\l }))>0.$
\end{itemize}
\end{thm}
\begin{proof}
We first prove (1). 
By \cite{sumi2}, 
we have that $\HD (J(G_{\l }))\leq \min \{ s(\l ),2\} $ for each 
$\l \in U.$ 
Hence it suffices to show that 
$\HD (J(G_{\l }))\geq \min \{ s(\l ),2\} $ for 
Leb$_{d}$-a.e. $\l \in U.$ 
Suppose that this is not true. 
Then, there exists an $\epsilon >0$ and a point 
$\l_{1} \in U$ such that 
$\l _{1}$ is a Lebesgue density point of the set 
$\{ \l \in U: \HD (J(G_{\l }))<\min \{ s(\l ),2\} -\epsilon \} .$ 
Then there exists $\delta _{0}>0$ such that 
for each $\delta \in (0,\delta _{0})$, 
\begin{equation}
\label{eq:Lebdmp}
\mbox{Leb}_{d}(\{ \l \in B_{\delta }(\l _{1}): \HD (J(G_{\l }))<\min \{ s(\l ),2\} -\epsilon \}) >0. 
\end{equation}
However, by the continuity of the function $\l \mapsto s(\l )$ 
(see Theorem~\ref{t:delrp}, \cite{suetds1}), 
if $\delta $ is small enough, 
then $s(\l )<s(\l _{1})+\frac{\epsilon }{2}$ for each 
$\l \in B_{\delta }(\l _{1}).$ 
Thus, 
for all $\delta $ sufficiently small, we obtain from (\ref{eq:Lebdmp}) 
that 
$$ 
\mbox{Leb}_{d}(\{ \l \in B_{\delta }(\l _{1}): \HD (J(G_{\l
}))<\min \{ s(\l _{1}),2\} -\frac{\epsilon }{2}\} )>0.
$$ 
This however contradicts Lemma~\ref{l:djms2e}. 
Thus, we have proved assertion (1). 
Statement (2) follows from Lemma~\ref{l:posmeas}.
Hence, we have proved our theorem.  
\end{proof} 
\begin{rem}
\label{r:acim}
Let $\{ f_{\l }\} _{\l \in U}$ be as in Theorem~\ref{t:tcdimj}. 
Let $\nu $ be the equilibrium state with respect to $\tilde{f}_{\l _{0}}:J(\tilde{f}_{\l _{0}})\rightarrow J(\tilde{f}_{\l _{0}})$ 
for the potential $-\delta (f_{\l _{0}})\log \| \tilde{f}_{\l _{0}}'\| $ (see Remark~\ref{r:eqst}). 
Then for each $\l \in U$, the Borel probability measure $(\overline{h}_{\l })_{\ast }(\mu ) $ in Theorem~\ref{t:tcdimj} is equivalent to 
$(\pi _{2})_{\ast }((h_{\l })_{\ast }(\nu ))$ and $(h_{\l })_{\ast }(\nu )$ is $\tilde{f}_{\l }$-invariant. 
Thus $(\overline{h}_{\l })_{\ast }(\mu ) $ is equivalent to the projection of an $\tilde{f}_{\l }$-invariant Borel probability measure on 
$J(\tilde{f}_{\l }).$ 
\end{rem}

We now define the strong transversality condition. 
\begin{dfn}
\label{d:nrf}
For each $r>0$ and each subset $F$ of $\RR^{d}$, 
we denote by $N_{r}(F)$ the minimal number of balls of radius $r$
needed to cover the set $F.$  

Let $\nu $ be a Borel probability measure in $\RR ^{d}.$ Let $u\geq 0.$ 
Let $E$ be a Borel subset of $\RR ^{d}.$  
We say that $\nu $ is a Frostman measure on $E$ with exponent $u$ 
if $\nu (E)=1$ and if there exists a constant $C>0$ such that for each
$x\in \RR ^{d}$ and for each $r>0$,  
$\nu (B_{r}(x))\leq Cr^{u}.$   
\end{dfn}
\begin{dfn}
\label{d:stc}
Let $d\in \NN .$ Let $U$ be a non-empty bounded open subset of $\RR ^{d}.$ 
Let $\{ f_{\l }\} _{\l \in U}$ be a family as in Setting $(\ast ).$ 
We say that $\{ f_{\l }\} _{\l \in U}$ satisfies the strong transversality condition (STC) 
if there exists a constant $C_{1}'>0$ such that for each 
$r\in (0,\diam (\CCI ))$ and for each 
$(\om ,z),(\om ',z')\in J(\tilde{f}_{\l _{0}})$ with 
$\om _{1}\neq \om '_{1},$ 
\begin{equation}
\label{eq:stc}
N_{r}(\{ \l \in U: \hat{\rho }(\overline{h}_{\l }(\om ,z),\overline{h}_{\l }(\om ',z'))\leq r\} )
\leq C_{1}'r^{2-d}.
\end{equation}
\end{dfn} 
\begin{rem}
\label{r:stctc}
The strong transversality condition implies the transversality condition. 
It is however not known whether or not there exists a 
family of multimaps of rational maps (or contracting conformal IFSs)  
which satisfies the transversality condition 
but fails to satisfy the strong transversality condition.  
\end{rem} 
In the same way as Lemma~\ref{l:tcintc} we can prove the following.
\begin{lem}
\label{l:stcintc}
Let $d\in \NN .$ Let $U$ be a non-empty bounded open subset of $\RR ^{d}.$ 
Let $\{ f_{\l }\} _{\l \in U}$ be a family as in Setting $(\ast ).$ 
Suppose that $\{ f_{\l }\} _{\l \in U}$ satisfies the 
strong transversality condition. 
Let $\nu $ be a Frostman measure in $\RR ^{d}$ with exponent $u.$ 
Suppose $u-d+2>0.$  
Then for each $\alpha \in (0,u-d+2)$ 
there exists a constant $C_{2}'>0$ such that 
for each $(\om ,z,\om ',z')\in J(\tilde{f}_{\l _{0}})$ with $\om _{1}\neq \om '_{1}$,  
$$\int _{U}\frac{d\nu (\l)}{\hat{\rho }(\overline{h}_{\l }(\om ,z),\overline{h}_{\l }(\om ',z'))^{\alpha }}
\leq C_{2}'.$$ 
\end{lem}
\begin{lem}
\label{l:sdjms}
Let $d\in \NN .$ Let $U$ be a non-empty bounded open subset of $\RR ^{d}.$ 
Let $\{ f_{\l }\} _{\l \in U}$ be a family as in Setting $(\ast ).$ 
Suppose that $\{ f_{\l }\} _{\l \in U}$ satisfies the 
strong transversality condition. 
Then for each $\l _{1}\in U$,  for each $\epsilon >0$, and for each $u\geq 0$,  
there exists $\delta >0$ such that 
if $\nu $ is a Frostman measure on $B_{\delta }(\l _{1})$ with exponent $u$, 
then 
$$\HD (J(G_{\l }))\geq \min \{ s(\l _{1}),u-d+2\} -\epsilon $$ 
for $\nu $-a.e. $\l \in B_{\delta }(\l _{1}).$  
\end{lem}
\begin{proof}
We may assume that $\l_{1}=\l _{0}$ and $u-d+2>0.$  
Let $s:=\min \{ s(\l _{0}),u-d+2\} .$ 
We repeat the proof of Lemma~\ref{l:djms2e}. 
The only change is that now we prove 
$\int _{B_{\delta }(\l _{0})}R(\l )d\nu (\l )<\infty $ by using 
Lemma~\ref{l:stcintc}.  
\end{proof}

We now give an upper estimate of the Hausdorff dimension of the set of exceptional parameters.
 Note that if $\{ f_{\l }=(f_{\l ,1},\ldots ,f_{\l, m})\} _{\l \in U}$ is a family in $\Exp(m)$, 
then by Theorem~\ref{t:fundfact1}, for each $\l \in U$, $\HD (J(G_{\l } ))\leq  s(\l ) ,$ where 
$G_{\l }:=\langle f_{\l ,1},\ldots ,f_{\l ,m}\rangle $ and $s(\l ):=\delta (f_{\l }).$    

\begin{thm}
\label{t:stcmain}
Let $d\in \NN .$ Let $U$ be a non-empty bounded open subset of $\RR ^{d}.$ 
Let $\{ f_{\l }\} _{\l \in U}$ be a family as in Setting $(\ast ).$ 
Suppose that $\{ f_{\l }\} _{\l \in U}$ satisfies the 
strong transversality condition. 
Let $G$ be a subset of $U$. 
Let $\xi \geq 0.$ 
Suppose $\min \{ \xi ,\sup _{\lambda \in G}s(\lambda )\}+d-2\geq 0.$ 
Then we have  
\begin{equation}
\label{eq:stcmain1}
\HD (\{ \l \in G: \HD (J(G_{\l }))<\min \{ \xi ,s(\l )\} \})
\leq \min \{ \xi ,\sup _{\l \in G}s(\l )\} +d-2.
\end{equation}
\end{thm}
\begin{proof}
We set $\kappa := \min \{ \xi ,\sup _{\l\in G}s(\l )\} +d-2.$ 
By the countable stability of Hausdorff dimension, it is enough
to prove that  for each $n\in \NN $, 
\begin{equation}
\label{eq:stcmain2}
\HD (\{ \l \in G: \HD (J(G_{\l }))<\min \{ \xi ,s(\l )\}
-\frac{1}{n}\} )\leq \kappa . 
\end{equation}
Fix $n\in \NN .$ In order to prove (\ref{eq:stcmain2})
it suffices to show that for each $\l _{1}\in G$ there exists a
$\delta =\delta _{\l _{1}}>0$  such that 
\begin{equation}
\label{eq:stcmain3}
\HD (\{ \l \in B_{\delta }(\l _{1}): \HD (J(G_{\l }))<\min \{ \xi ,
s(\l )\} -\frac{1}{n}\} )\leq \kappa. 
\end{equation}
To prove (\ref{eq:stcmain3}), suppose that it is false. 
Then there exists $\l _{1}\in G$ such that for each $\delta >0$, 
\begin{equation}
\label{eq:stcmain4}
\HD (\{ \l \in B_{\delta }(\l _{1}): \HD (J(G_{\l }))<\min \{ \xi
,s(\l )\} -\frac{1}{n}\} )>\kappa . 
\end{equation}
Choose $\delta >0$ so small that the statement of Lemma~\ref{l:sdjms} holds with 
$\epsilon =\frac{1}{2n}$ 
and $|s(\l )-s(\l _{1})|<\frac{1}{2n}$ for each $\l \in B_{\delta }(\l _{1})$ 
(by the continuity of $s(\l )$, see Theorem~\ref{t:delrp}). 
Then, 
\begin{align*}
\{ \l \in B_{\delta }(\l _{1}): \HD (J(G_{\l }))&<\min \{ \xi
,s(\l )\} -\frac{1}{n}\} \\ 
& \subset \{ \l \in B_{\delta }(\l _{1}): \HD (J(G_{\l }))<\min \{ \xi
,s(\l _{1})\} -\frac{1}{2n}\} :=E.
\end{align*}
Hence $\HD (E)>\kappa .$ 
By Frostman's Lemma (see \cite[Corollary 4.12]{F}), 
there exists a Frostman measure $\nu $ on the set $E$ with exponent $u=\kappa .$ 
By Lemma~\ref{l:sdjms}, for $\nu $-a.e. $\l $ we have 
$$\HD (J(G_{\l }))\geq \min \{ s(\l _{1}),\kappa -d+2\} -\frac{1}{2n}
=\min \{ s(\l _{1}),\min \{ \xi , \sup _{\l \in G}s(\l )\} \} -\frac{1}{2n}.$$ 
This is a contradiction since for each $\l\in E$ 
we have 
$\HD (J(G_{\l }))<\min \{ \xi ,s(\l _{1})\}-\frac{1}{2n}$ 
and 
$$\min \{ \xi ,s(\l _{1})\} \leq \min \{ s(\l _{1}),\min \{ \xi ,\sup _{\l \in G}s(\l )\} \} .$$ 
Thus we have proved Theorem~\ref{t:stcmain}. 
\end{proof}

\noindent By continuity of $s(\l )$ (see Theorem~\ref{t:delrp}, \cite{suetds1}),  
as an immediate consequence of Theorem~\ref{t:stcmain},
 we get the following estimate for the local dimension of the exceptional set.  
\begin{cor}
Let $d\in \NN .$ Let $U$ be a non-empty bounded open subset of $\RR ^{d}.$ 
Let $\{ f_{\l }\} _{\l \in U}$ be a family as in Setting $(\ast ).$ 
Suppose that $\{ f_{\l }\} _{\l \in U}$ satisfies the 
strong transversality condition. 
Let $\xi \geq 0.$ 
Suppose $\min \{ \xi ,s(\lambda _{1})\}+d-2\geq 0.$ 
Then, we have all of the following. 
\begin{itemize}
\item[(1)] 
For each $\l _{1}\in U$, we have 
$$ \lim _{r\rightarrow 0}\HD (\{ \l \in B_{r}(\l _{1}): 
\HD (J(G_{\l }))<\min \{ \xi ,s(\l )\} \} )\leq \min \{ \xi ,s(\l _{1})\} +d-2.$$ 
\item[(2)]
If, in addition to the assumptions of our corollary, $s(\l _{1})<2$, 
then 
$$\lim _{r\rightarrow 0} \HD (\{ \l \in B_{r}(\l _{1}): \HD (J(G_{\l }))\neq s(\l )\} )\leq 
d-(2-s(\l _{1}))<d=\HD (U).$$ 
\end{itemize}
\end{cor}
We now give a sufficient condition for a holomorphic family $\{ f_{\l }\} _{\l \in U}$ to satisfy 
the strong transversality condition. 
\begin{dfn}
\label{d:analtc}
Let $U$ be an open subset of $\CC ^{d}.$ 
Let $\{ f_{\l }\} _{\l \in U}=\{ (f_{\l ,1},\ldots ,f_{\l ,m})\} _{\l \in U}$ be a holomorphic family in $\Exp (m)$ 
over $U.$ We set 
$G_{\l }:= \langle f_{\l ,1},\ldots ,f_{\l ,m}\rangle $ for each $\l \in U.$ 
Let $\l _{0}\in U$ be a point. 
Suppose that for each $\l \in U$, 
there exists a homeomorphism $h_{\l }:J(\tilde{f}_{\l _{0}})\rightarrow 
J(\tilde{f}_{\l })$ of the form $h_{\l }(\om,z)=(\om,\overline{h}_{\l }(\om,z))$ such that  
$h_{\l _{0}}=Id|_{J(\tilde{f}_{\l _{0}})}$, 
$h_{\l }\circ \tilde{f}_{\l _{0}}=\tilde{f}_{\l }\circ h_{\l }$ on $J(\tilde{f}_{\l _{0}})$,  
 and such that for each $(\om,z)\in J(\tilde{f}_{\l _{0}})$ 
the map $(\om,z,\l) \mapsto \overline{h}_{\l }(\om,z)\in \CCI , 
(\om,z,\l) \in J(\tilde{f} _{\l _{0}})\times U$, is continuous and 
the map $\l \mapsto \overline{h}_{\l }(\om ,z)$ is holomorphic. 
We say that the family $\{ f_{\l }\} _{\l \in U}$ satisfies the
analytic transversality condition (ATC) if the following hold.
\begin{itemize}
\item[{\em (a)}] 
$J(G_{\l })\subset \CC $ for each $\l \in U$. 
\item[{\em (b)}] 
For each $(\om ,z,\om ',z',\l )\in J(\tilde{f}_{\l _{0}})^{2}\times U$, 
let $g_{\om ,z,\om ',z'}(\l ):= \overline{h}_{\l }(\om ,z)-\overline{h}_{\l }(\om ',z').$ 
Then for each $(\om ,z,\om ',z',\l )\in J(\tilde{f}_{\l _{0}})^{2}\times U$ with 
$g_{\om ,z,\om ',z'}(\l )=0$ and $\om _{1}\neq \om '_{1}$, 
we have 
$\bigtriangledown _{\l }g_{\om ,z,\om ',z'}(\l )\neq 0$, 
where 
 $\bigtriangledown _{\l }g_{\om ,z,\om ',z'}(\l ):= 
(\frac{\partial g_{\om ,z,\om ',z'}}{\partial \l _{1}}(\l ),\ldots , 
\frac{\partial g_{\om ,z,\om ',z'}}{\partial \l _{d}}(\l )).$  
\end{itemize}
\end{dfn}
\begin{prop}
\label{p:atctc}
Let $U$ be a bounded open subset of $\CC ^{d}.$ 
Let $\{ f_{\l }\} _{\l \in U}$ be a holomorphic family in $\Exp (m)$ 
over $U.$ Suppose that $\{ f_{\l }\} _{\l \in U}$ satisfies the
analytic transversality condition.  
Then for each non-empty, relative compact, open subset $U'$ of $U$, 
the family $\{ f_{\l }\} _{\l \in U'}$ satisfies the strong
transversality condition, and consequently, the transversality condition.  
\end{prop} 
\begin{proof}
Let $\l _{0}\in U$ and let  
$h_{\l }$ and $g_{\om ,z,\om ',z'}(\l )$ be as in Definition~\ref{d:analtc}.
We set 
$$
W:= \{ (\om ,z,\om ',z',\zeta )\in J(\tilde{f}_{\l _{0}})^{2}\times
U: g_{\om ,z,\om ',z'}(\zeta )=0 \text{ and } \om _{1}\neq \om '_{1}\} .
$$  
For each $\l \in U$  write $\l =(\l _{1},\ldots ,\l _{d}).$ 
Let $(\om ,z,\om ',z',\zeta )\in W.$ 
Then $\bigtriangledown _{\l }g_{\om ,z,\om ',z'}(\zeta )\neq 0.$ 
Without loss of generality, we may assume that 
$\frac{\partial g_{\om ,z,\om ',z'}}{\partial \l _{1}}(\zeta )\neq 0.$ 
 Then by the arguments in \cite[page 154]{A}, 
there exists a neighborhood $A_{0}$ of 
$(\om ,z,\om ',z')$, a constant $\delta >0$, and a constant $r_{0}>0$, 
such that 
for each $(x,y,x',y')\in A_{0}$ and  for each 
$(\l _{2},\ldots ,\l _{d})\in D_{2\delta }(\zeta _{2})\times \cdots \times D_{2\delta }(\zeta _{d})$, 
setting 
$g_{x,y,x',y',\l _{2},\ldots ,\l _{d}}(\l _{1}):= 
g_{x,y,x',y'}(\l _{1},\ldots ,\l _{d})$ for each 
$\l _{1}\in D_{2\delta }(\zeta _{1})$, 
we have that 

\vspace{2mm}\noindent(i) $g_{x,y,x',y',\l _{2},\ldots ,\l _{d}}$ is injective
on $D_{2\delta }(\zeta _{1})$, and \\  
(ii) there exists a holomorphic function 
$\alpha _{x,y,x',y',\l _{2},\ldots ,\l _{d}}:D_{2r_{0}}(0)\rightarrow D_{2\delta }(\zeta _{1})$ 
such that 
$$
g_{x,y,x',y',\l _{2},\ldots ,\l _{d}}\circ \alpha _{x,y,x',y',\l
  _{2},\ldots \l _{d}}=\mbox{Id} \mbox{ on } D_{2r_{0}}(0).
$$
We may assume that there exists a constant $C_{0}>0$ such that 
for each $(x,y,x',y')\in A_{0}$, 
for each $(\l _{2},\ldots ,\l _{d},z)\in \prod _{j=2}^{d}D_{2\delta }(\zeta _{j})\times D_{2r_{0}}(0)$, 
and for each $j=2,\ldots ,d$, we have
\begin{equation}
\label{eq:contialp}
|\alpha _{x,y,x'y',\l _{2},\ldots ,\l _{d}}'(z)|\leq C_{0}, \mbox{ and } 
\bigg|\frac{\partial \alpha _{x,y,x',y',\l _{2},\ldots ,\l _{d}}(z)}{\partial \l _{j}}\bigg|\leq C_{0}. 
\end{equation}

For every $(x,y,x',y')\in A_{0}$ and for every $r\in (0,r_{0})$, 
\begin{align*}
\ \ & \ \ \ \ \ \{ (\l _{1},\ldots ,\l _{d})\in \prod _{j=1}^{d} D_{\delta }(\zeta _{j}):  
|g_{x,y,x',y'}(\l _{1},\ldots ,\l _{d})|<r\} \\ 
& = \{ (\alpha _{x,y,x',y', \l _{2},\ldots ,\l _{d}}(z),\l _{2},\ldots ,\l _{d})
: (\l _{2},\ldots ,\l _{d})\in \prod _{j=2}^{d}D_{\delta }(\zeta _{j}), z\in D_{r}(0)\} \\ 
& =\Psi _{x,y,x'y'}(\prod _{j=2}^{d}D_{\delta }(\zeta _{j})\times D_{r}(0)),
\end{align*}
where $\Psi _{x,y,x',y'}(\l _{2},\ldots ,\l _{d},z):=
(\alpha _{x,y,x',y',\l _{2},\ldots ,\l _{d}}(z),\l _{2},\ldots ,\l _{d}).$ 
Let $A_{r}:=\prod _{j=2}^{d}D_{\delta }(\zeta _{j})\times D_{r}(0).$ 
Then there exists a constant $C_{1}>0$ such that 
for each $r>0$, 
$N_{r}(A_{r})\leq C_{1}(\frac{1}{r})^{2(d-1)}.$ 
Let $\{ E_{j}\} _{j=1}^{N_{r}(A_{r})}$ 
be a family of $r$-balls with 
$A_{r}\subset \bigcup _{j=1}^{N_{r}(A_{r})}E_{j}.$ 
By (\ref{eq:contialp}), 
there exists a constant $C_{2}>0$  such that 
for each $(x,y,x',y')\in A_{0}$, for each $r\in (0,r_{0})$ and for
each $j\in \{ 1,\ldots ,N_{r}(A_{r})\} ,$   
$\Psi _{x,y,x',y'}(E_{j})$ is included in a $C_{2}r$-ball. 
Therefore, there exists a constant $C_{3}>0$ such that 
for each $(x,y,x',y')\in A_{0}$ and $r\in (0,r_{0})$, 
$N_{r}(\Psi _{x,y,x'y'}(A_{r}))\leq C_{3}r^{2-2d}.$  
Hence, 
we obtain 
$$N_{r}(\{ (\l _{1},\ldots ,\l _{d})\in \prod _{j=1}^{d}D_{\delta }(\zeta _{j}): 
|g_{x,y,x',y'}(\l _{1},\ldots ,\l _{d})|<r\} )\leq C_{3}r^{2-2d}.$$
Therefore, for each non-empty relative compact open subset $U'$ of
$U$, the family $\{ f_{\l }\} _{\l\in U}$ 
satisfies the strong transversality condition and, consequently, the
transversality condition.   
\end{proof}
\begin{rem}
\label{r:stcatc}
If $d=1$ and the strong transversality condition holds (which is equivalent to 
that $\inf \{ \hat{\rho }(a,b) : a\in f_{\l, i}^{-1}(J(G_{\l })), b\in
f_{\l ,j}^{-1}(J(G_{\l })), \l \in U, i\neq j\}  
> 0\} $), then the analytic transversality condition is not satisfied. However, 
it is not known whether or not there exists a holomorphic
family of multimaps of rational maps (or contracting conformal  
IFSs on $\CC $) which satisfies the strong transversality condition
but fails to satisfy the analytic transversality condition. 
\end{rem}
Looking at Proposition~\ref{p:atctc} we see that in order to obtain a
sufficient condition for a holomorphic family  
$\{ f_{\l }\} _{\l \in U}$ in $\Exp (m)$ to satisfy the strong
transversality condition, it is important to calculate
$\frac{\partial g_{\om ,z,\om ',z'}(\l )}{\partial \l _{j}}.$ 
We give now several methods of doing this. 
\begin{lem}
\label{l:hfdiff}
Let $U$ be a bounded open set subset of $\CC $. 
Let $\l _{0}\in U.$  
Let $\{ f_{\l }\} _{\l \in U}=\{ f_{\l ,1},\ldots ,f_{\l ,m}\} _{\l
  \in U}$ be a holomorphic family in  
$\Exp(m).$ 
For each $\l \in U$, let $G_{\l }, h_{\l },\overline{h}_{\l }$ be as
in Setting $(\ast ).$  
Suppose that for each $\l \in U$, $J(G_{\l })\subset \CC .$ 
Then for each $(\om ,z)\in J(\tilde{f}_{\l _{0}})$,  
\begin{equation}
\label{eq:hfdiffeq}
\frac{\partial \overline{h}_{\l }(\om ,z)}{\partial \l }|_{\l =\l _{0}} 
=\sum _{n=1}^{\infty }\frac{1}{f_{\l _{0},\om |_{n}}'(z)}
\left(-\frac{\partial f_{\l ,\om _{n}}(f_{\l _{0},\om
    |_{n-1}}(z))}{\partial \l }\Big|_{\l =\l _{0}}\right),  
\end{equation}
where $f_{\l _{0},\omega |_{0}}$ is the identity map.  
\end{lem}
\begin{proof}
Since $\tilde{f_{\l }}\circ h_{\l }=h_{\l }\circ \tilde{f}_{\l _{0}}$, 
we have that 
for each $\l \in U$ and for each $(\om ,z)\in J(\tilde{f}_{\l _{0}})$,   
$f_{\l ,\om _{1}}(\overline{h}_{\l }(\om ,z))=\overline{h}_{\l
}(\sigma (\om ),f_{\l _{0},\om _{1}}(z)).$  
Hence 
$$ \frac{\partial f_{\l ,\om _{1}}}{\partial \l }(\overline{h}_{\l }(\om ,z))
+f_{\l ,\om _{1}}'(\overline{h}_{\l }(\om ,z))\frac{\partial
  \overline{h}_{\l }(\om ,z)}{\partial \l } 
=\frac{\partial \overline{h}_{\l }(\sigma (\om ), f_{\l _{0},\om _{1}}(z))}{\partial \l }.$$ 
Therefore, 
\begin{equation}
\label{eq:hfdiffeq2}
\frac{\partial \overline{h}_{\l }(\om ,z)}{\partial \l }|_{\l =\l _{0}}
=\frac{1}{f_{\l _{0},\om _{1}}'(z)}
\left( -\frac{ \partial f_{\l ,\om _{1}}(z)}{\partial \l }|_{\l =\l _{0}} 
+\frac{ \partial \overline{h}_{\l }(\sigma (\om ),f_{\l _{0},\om
    _{1}}(z))}{\partial \l }|_{\l =\l _{0}} 
\right).  
\end{equation}
Iterating this calculation,  since the right hand side of
(\ref{eq:hfdiffeq}) converges due to the expandingness of $G_{\l
_{0}}$, we obtain  equation (\ref{eq:hfdiffeq}). 
\end{proof}
We remark that the calculation like (\ref{eq:hfdiffeq}) is a well-known technique 
in contracting IFSs with overlaps (e.g. \cite{SSU}), though in Lemma~\ref{l:hfdiff} 
we deal with ``expanding'' semigroups in which each map may not be injective.

We now provide several corollaries of Lemma~\ref{l:hfdiff}. 
\begin{cor}
\label{c:hfdiff1}
Let $(g_{1},\ldots ,g_{m})\in \Exp (m).$ 
Let $U$ be a bounded open subset of $\CC .$ 
Let $\l _{0}\in U.$ For each $\l \in U$, 
let $\alpha _{\l }\in \mbox{{\em Aut}}(\CCI )$. We assume that 
the map $ \CCI \times U\ni(z,\l )\mapsto \alpha _{\l }(z)\in \CCI$
is holomorphic, and that $\alpha _{\l _{0}}=Id.$ For each $\l \in U$ let 
$$
f_{\l }:=(g_{1},\ldots, g_{m-1}, \alpha _{\l
}\circ g_{m}\circ \alpha _{\l }^{-1}).
$$  
Suppose that $\{ f_{\l }\} _{\l \in U}$ is a holomorphic family in
$\Exp (m)$ which satisfies the Setting $(\ast )$. Further, 
letting $G_{\l }, h_{\l },\overline{h}_{\l }$ be as in the Setting
$(\ast )$ assume that  
$U\ni\l \mapsto \overline{h}_{\l }(\om ,z)$ is holomorphic.  
Note that if $U$ is small enough, then we do not need any extra
hypotheses, namely, by Lemma~\ref{expopenlem} and
Remark~\ref{r:setting},  
$\{ f_{\l }\} _{\l \in U}$ is automatically a holomorphic family in $\Exp (m)$
satisfying Setting $(\ast )$, 
and the map $U\ni\l \mapsto \overline{h}_{\l }(\om ,z)$ is holomorphic.   
In any case we also extra assume that for each $\l \in U$, $J(G_{\l
})\subset \CC $ 
(see Remark~\ref{r:conjG}).   
For each $\om =(\om _{1},\ldots ,\om _{n})\in \Sigma _{m}^{\ast }$, let 
$g_{\om } =g_{\om _{n}}\circ \cdots \circ g_{\om _{1}}.$ 
Then, we have all of the following. 
\begin{itemize}
\item[(1)]
For each $(\om ,z)\in J(\tilde{f}_{\l _{0}})$, 
$$\frac{\partial \overline{h}_{\l }(\om ,z)}{\partial \l }\bigg|_{\l =\l _{0}}
=\sum _{n=1}^{\infty }
\frac{1}{g_{\om |_{n}}'(z)}a_{n}(z),$$
where 
$$a_{n}(z):=\begin{cases}
0  \mbox{ if } \om _{n}=1,\ldots ,m-1\\ 
g_{m}'(g_{\om |_{n-1}}(z))(-\frac{\partial \alpha _{\l}(g_{\om
    |_{n-1}}(z))}{\partial \l }\big|_{\l =\l _{0}} ) 
+\frac{\partial \alpha _{\l }(g_{\om |_{n}}(z))}{\partial \l
}\big|_{\l=\l _{0}}       
\mbox{ if } \om _{n}=m. \ (g_{\om |_{0}}:= Id.)
\end{cases}
$$ 
\item[(2)] 
Let $j\neq m$, $\beta =jm^{\infty } $ and $\gamma =mj^{\infty }.$  
Then for each $z\in \CCI $ with $(\beta ,z)\in J(\tilde{f}_{\l _{0}})$, 
$$
\frac{\partial \overline{h}_{\l }(\beta ,z)}{\partial \l }\bigg|_{\l =\l _{0}}=
\frac{1}{g_{j}'(z)}\frac{\partial \alpha _{\l} (g_{j}(z))}{\partial \l
}\bigg|_{\l =\l _{0}},
$$ 
and for each $z\in \CCI $ with $(\gamma , z)\in J(\tilde{f}_{\l _{0}})$, 
$$ 
\frac{\partial \overline{h}_{\l }(\gamma ,z)}{\partial \l }\bigg|_{\l =\l _{0}}=
\frac{\partial \alpha _{\l }(z)}{\partial \l }\bigg|_{\l =\l _{0}}
-\frac{1}{g_{m}'(z)}\frac{\partial \alpha _{\l}(g_{m}(z))}{\partial \l
}\bigg|_{\l =\l _{0}}. 
$$
\end{itemize}
\end{cor}
\begin{proof}
It is easy to see that 
\begin{equation}
\label{eq:conjdiff}
\frac{\partial (\alpha _{\l }g_{m}\alpha _{\l }^{-1}(z))}{\partial \l
}\bigg|_{\l =\l _{0}} 
=g_{m}'(z)\left(-\frac{\partial \alpha _{\l }(z)}{\partial
    \l}\big|_{\l =\l _{0}}\right)+ 
\frac{\partial \alpha _{\l }(g_{m}(z))}{\partial \l }\big|_{\l =\l _{0}}.
\end{equation}
By Lemma~\ref{l:hfdiff} and (\ref{eq:conjdiff}), statement (1) holds. 
We now prove statement (2). 
By the uniqueness of the conjugacy map $h_{\l }$ (\cite[Theorem 4.9]{suetds1}), 
we have for each $\l $ close to $\l _{0}$ and for each $j\neq m$,  that
$\overline{h}_{\l }(j^{\infty },z)=z\ (z\in J_{j^{\infty
  }}(\tilde{f}_{\l _{0}})=J(g_{j}))$  
and $\overline{h}_{\l }(m^{\infty },z)=\alpha _{\l }(z)\ (z\in
J_{m^{\infty }}(\tilde{f}_{\l _{0}})=J(g_{m})).$  
Therefore, by (\ref{eq:hfdiffeq2}) and (\ref{eq:conjdiff}), 
statement (2) holds.    
\end{proof}

\begin{cor}
\label{c:hfdiff2}
Let $(g_{1},\ldots ,g_{m})\in \Exp (m)\cap (\mbox{{\em Aut}}(\CC )\cup {\mathcal P})^{m}.$ 
Let $U$ be a bounded open subset of $\CC $ with $0\in U.$   
Let $\l _{0}=0\in U.$ 
Let $j\in \NN \cup \{ 0\} $ with $0\leq j\leq \deg (g_{m}).$  
For each $\l \in U$, let 
$$
f_{\l }:=(g_{1},\ldots, g_{m-1}, g_{m}+\l z^{j}).
$$ 
Assume that $\{ f_{\l }\} _{\l \in U}$ is a holomorphic family in
$\Exp (m)$ satisfying the 
Setting $(\ast )$. Further, letting $G_{\l }, h_{\l },\overline{h}_{\l
}$ be as in the Setting $(\ast )$ suppose that the map
$U\ni\l \mapsto \overline{h}_{\l }(\om ,z)$ is holomorphic. 
Note that if the open set $U$ is small enough, then by
Lemma~\ref{expopenlem} and Remark~\ref{r:setting},  
$\{ f_{\l }\} _{\l \in U}$ is automatically a holomorphic family in
$\Exp (m)$ satisfying the Setting $(\ast )$ and the map
$U\ni\l \mapsto \overline{h}_{\l }(\om ,z)$ is holomorphic.
For each $\om =(\om _{1},\ldots \om _{n})\in \Sigma _{m}^{\ast }$, let 
$g_{w} =g_{\om _{n}}\circ \cdots \circ g_{\om _{1}}.$ 
Then, for each $(\om ,z)\in J(\tilde{f}_{\l _{0}})$, 
$$ 
\frac{\partial \overline{h}_{\l }(\om ,z)}{\partial \l }\bigg|_{\l =\l _{0}}
=\sum _{n=1}^{\infty }\frac{1}{g_{\om |_{n}}'(z)}a_{n}(z),
$$
where 
$$
a_{n}(z)=\begin{cases}
-(g_{\om |_{n-1}}(z))^{j} \mbox{ if } \om _{n}=m\\ 
0 \mbox{ if } \om _{n}\neq m.
\end{cases}
$$
\end{cor}
\begin{proof}
The proof follows immediately from Lemma~\ref{l:hfdiff}.
\end{proof}

\begin{cor}
\label{c:hfdiff3}
Let $(g_{1},\ldots ,g_{m})\in \Exp (m)\cap (\mbox{{\em Aut}}(\CC )\cup {\mathcal P})^{m}.$ 
Let $U$ be a bounded open subset of $\CC $ with $0\in U.$   
Let $\l _{0}=0\in U.$ 
For each $\l \in U$, let 
$$
f_{\l }:=(g_{1},\ldots, g_{m-1}, g_{m}+\l g_{m}').
$$ 
Assume that $\{ f_{\l }\} _{\l \in U}$ is a holomorphic family in
$\Exp (m)$ satisfying the Setting $(\ast )$. Further,
letting $G_{\l }, h_{\l },\overline{h}_{\l }$ be as in Setting $(\ast
)$ suppose that $\l \mapsto \overline{h}_{\l }(\om ,z)$ is holomorphic.    
Note that if the open set $U$ is small enough, then by
Lemma~\ref{expopenlem} and Remark~\ref{r:setting},  
$\{ f_{\l }\} _{\l \in U}$ is automatically a holomorphic family in
$\Exp (m)$ satisfying Setting $(\ast )$ and the map
$U\ni\l \mapsto \overline{h}_{\l }(\om ,z)$ is holomorphic.   
For each $\om =(\om _{1},\ldots \om _{n})\in \Sigma _{m}^{\ast }$, let 
$g_{w} =g_{\om _{n}}\circ \cdots \circ g_{\om _{1}}.$ 
Then, for each $(\om ,z)\in J(\tilde{f}_{\l _{0}})$, 
$$ 
\frac{\partial \overline{h}_{\l }(\om ,z)}{\partial \l }\bigg|_{\l =\l _{0}}
=\sum _{n=1}^{\infty }\frac{1}{g_{\om |_{n}}'(z)}a_{n}(z),
$$
where 
$$
a_{n}(z)=\begin{cases}
-g_{m}'(g_{\om |_{n-1}}(z)) \mbox{ if } \om _{n}=m\\ 
0 \mbox{ if } \om _{n}\neq m.
\end{cases}
$$

\end{cor}
\begin{proof}
By Lemma~\ref{l:hfdiff}, our Corollary holds.
\end{proof}
\begin{lem}
\label{l:genprin}
Let $U$ be a bounded open set in $\CC ^{d}$. 
Let $\l _{0}\in U.$  
Let $\{ f_{\l }\} _{\l \in U}=\{ f_{\l ,1},\ldots ,f_{\l ,m}\} _{\l
  \in U}$ be a holomorphic family in  
$\Exp(m)$ satisfying Setting $(\ast )$. 
Letting $G_{\l }, h_{\l },\overline{h}_{\l }$ be as in Setting $(\ast )$ 
we suppose that $U\ni\l \mapsto \overline{h}_{\l }(\om ,z)$ is holomorphic. 
Note that if $U$ is small enough, then by Lemma~\ref{expopenlem} and
Remark~\ref{r:setting},  
$\{ f_{\l }\} _{\l \in U}$ is automatically a holomorphic family in
$\Exp (m)$ satisfying Setting $(\ast )$  
and $\l \mapsto \overline{h}_{\l }(\om ,z)$ is holomorphic. 
Suppose that for each $\l \in U$, $J(G_{\l })\subset \CC .$ 
We also require all of the following conditions to hold. 
\begin{itemize}
\item[(i)]
For each $(i,j)$ with $i\neq j$ and 
$f_{\l _{0},i}^{-1}(J(G_{\l _{0}}))\cap f_{\l _{0},j}^{-1}(J(G_{\l _{0}}))\neq \emptyset $, 
there exists a number $\alpha _{ij}\in \{ 1,\ldots, m\} $ such that 
$f_{\l _{0},i}(f_{\l _{0},i}^{-1}(J(G_{\l _{0}}))\cap f_{\l _{0},j}^{-1}(J(G_{\l _{0}})))\subset J(f_{\l _{0},\alpha _{ij}}).$   

\item[(ii)] 
If $i,j,k$ are mutually distinct elements in $\{ 1,\ldots ,m\} $, 
then $$f_{\l _{0},k}(f_{\l _{0},i}^{-1}(J(G_{\l _{0}}))\cap f_{\l _{0},j}^{-1}(J(G_{\l _{0}})))\subset F(G_{\l _{0}}).$$
\item[(iii)] 
For each $(j,k)$ with $j\neq k$, 
$f_{\l _{0},k}(J(f_{\l _{0},j}))\subset F(G_{\l _{0}}).$ 
\item[(iv)] 
If $i\neq j$ and if $z\in f_{\l _{0},i}^{-1}(J(G_{\l _{0}}))\cap f_{\l _{0},j}^{-1}(J(G_{\l _{0}}))$ 
(note: for such $z$, by {\em (i)-- (iii)} we have $z\in J_{i\alpha _{ij}^{\infty }}(\tilde{f}_{\l _{0}})\cap 
J_{j\alpha _{ji}^{\infty }}(\tilde{f}_{\l _{0}})$),  
then 
$$ \nabla _{\l } (\overline{h}_{\l }(i\alpha _{ij}^{\infty },z)-\overline{h}_{\l }(j\alpha _{ji}^{\infty },z))|_{\l =\l _{0}}\neq 0.$$ 
\end{itemize} 
Then, there exists an open neighborhood $U_{0}$ of $\l _{0}$ in $U$ such that 
$\{ f_{\l }\} _{\l \in U_{0}}$ satisfies the analytic transversality condition, the strong transversality condition 
and the transversality condition. 
\end{lem} 
\begin{proof}
By conditions (i),(ii), (iii), Lemma~\ref{l:jtfwo} and Remark~\ref{rem1}(1),  
we obtain that 
\begin{gather}
\label{eq:g0gen}
\begin{split}
\ & \ \ \ \   \{ (\om ,z,\om ',z')\in J(\tilde{f}_{\l _{0}})^{2}: \om _{1}\neq \om '_{1}, 
\overline{h}_{0}(\om ,z)-\overline{h}_{0}(\om ',z')=0\} \\  
& \subset \bigcup _{(i,j):i\neq j}\{ (i\alpha _{ij}^{\infty }, z, j\alpha _{ji}^{\infty },z')\in J(\tilde{f}_{\l _{0}})^{2}: 
z=z'\in f_{\l _{0},i}^{-1}(J(G_{\l _{0}}))\cap f_{\l _{0},j}^{-1}(J(G_{\l _{0}}))\} . 
\end{split} 
\end{gather}
From (\ref{eq:g0gen}) and condition (iv), 
we conclude that there exists an open neighborhood $U_{0}$ of $\l
_{0}$ in $U$ such that  
$\{ f_{\l }\} _{\l \in U_{0}}$ satisfies the analytic transversality condition. 
By Proposition~\ref{p:atctc}, shrinking $U_{0}$ if necessary, 
it follows that $\{ f_{\l }\} _{\l \in U_{0}}$ satisfies the strong transversality condition and 
the transversality condition. 
\end{proof}
\begin{lem}
\label{l:atcemb}
Let $d_{1},d_{2}\in \NN $ with $d_{1}\leq d_{2}.$ 
Let $U$ be a bounded open subset of $\CC^{d_{1}}$ and 
let $V$ be a bounded open subset of $\CC ^{d_{2}}.$  
Let $\{ f_{\l }\} _{\l \in U}$ be a holomorphic family in $\Exp (m)$ over $U$ with base point $\l _{0}$ 
satisfying the analytic transversality condition. 
Let $\{ g_{\g }\} _{\g \in V}$ be a holomorphic family in $\Exp(m)$ over $V$ and let $\g _{0}\in V.$ 
Suppose that there exists a holomorphic embedding $\eta :U\rightarrow V$ with $\eta (\l _{0})=\g _{0}$ 
such that $g_{\eta (\l )}=f_{\l }$ for each $\l \in U.$  
Then there exists an open neighborhood $W$ of $\g _{0}$ in $V$ such that 
$\{ g_{\g }\} _{\g \in W}$ is a holomorphic family in $\Exp (m)$ over $W$ with base point $\g _{0}$ 
satisfying the analytic transversality condition, the strong transversality condition, and the transversality condition.   
\end{lem}
\begin{proof}
By Remark~\ref{r:setting}, there exists an open neighborhood $W$ of $\g _{0}$ in $V$ such that 
$\{ g_{\g }\} _{\g \in W}$ satisfies Setting $(\ast )$ and 
letting $h_{\g }, \overline{h}_{\g }$ be as in Setting $(\ast )$, 
for each $(\om, z)\in J(\tilde{g}_{\g _{0}})$ 
the map $W\ni\g \mapsto \overline{h}_{\g }(\om ,z)$ is holomorphic. 
Let $h_{\l }^{0}(\om ,z)=(\om ,\overline{h}_{\l }^{0}(\om ,z))$ 
be the conjugacy map as in the Setting $(\ast )$ for the family $\{ f_{\l }\} _{\l \in U}.$ 
Then shrinking $U$ if necessary, by the uniqueness of the family of conjugacy maps (see Remark~\ref{r:setting}), 
we obtain $h_{\eta (\l )}=h_{\l }^{0}$ for each $\l\in U.$ 
Since $\{ f_{\l }\} _{\l \in U}$ satisfies the analytic transversality condition, 
shrinking $W$ if necessary,  it follows that $\{ g_{\g }\} _{\g \in W}$ satisfies the 
analytic transversality condition. 
By Proposition~\ref{p:atctc}, shrinking $W$ if necessary again, 
we obtain that $\{ g_{\g }\} _{\g \in W}$ satisfies the strong transversality condition and the transversality condition. 
\end{proof}

\begin{rem}
\label{r:atcex}
By Lemma~\ref{l:hfdiff}, Corollaries~\ref{c:hfdiff1},
\ref{c:hfdiff2},\ref{c:hfdiff3}, Lemmas~\ref{l:genprin},
\ref{l:atcemb},    
and Proposition~\ref{p:atctc},  
we can obtain many examples of holomorphic families $\{ f_{\l }\} _{\l
  \in U}$ in $\Exp(m)$ satisfying the  
analytic transversality conditions, the strong transversality
condition and the transversality condition.  
In the following section we will provide various kinds of 
examples of the holomorphic 
families satisfying the analytic transversality condition.  
\end{rem}
\section{Applications and Examples}
\label{Applications}
In this section, we apply the results of the previous one to describe
various examples and to solve a variety of emerging problems. 
For a polynomial $g\in {\mathcal P}$, we set 
$$
K(g):=\{ z\in \CC : \{ g^{n}(z)\} _{n\in \NN } \mbox{ is bounded in }
\CC \} 
$$ 
and we recall that $K(g)$ is commonly referred to as the filled in 
Julia set of the polynomial $g$.
\begin{thm}
\label{t:d1d2ex}
Let $(d_{1},d_{2})\in \NN ^{2}$ be such that  $d_{1},d_{2}\geq 2$ and
$(d_{1},d_{2})\neq (2,2).$  
Let $b=ue^{i\theta }\in \{ 0<|z|<1\}$, where $0<u<1$ and $\theta \in [0,2\pi ).$ 
Let $\alpha \in [0,2\pi )$ be a real number such that 
there exists an integer $n\in \ZZ $ with $d_{2}(\pi +\theta )+\alpha
=\theta +2n\pi .$    
Let $\b _{1}(z)=z^{d_{1}}.$ For each $t>0$, 
let $g_{t}(z)=te^{i\alpha }(z-b)^{d_{2}}+b.$ 
Then there exists a point $t_{1}\in (0,\infty )$ and an open 
neighborhood $U$ of $0$ in $\CC $ such that 
the family $\{ f_{\l }=(\b_{1},g_{t_{1}}+\l g_{t_{1}}')\} _{\l \in U}$ with $\l _{0}=0$
satisfies all the conditions {\em (i)--(iv)}. 
\begin{itemize}
\item[(i)] 
$\{ f_{\l }\} _{\l \in U}$ is a holomorphic family in $\Epb (2)$ satisfying the analytic transversality condition, 
the strong transversality condition  
and the transversality condition. 

\item[(ii)] 
For each $\l\in U$, $s(\l )<2$, where we recall that $s(\l )=\delta (f_{\l }).$  

\item[(iii)] 
There exists a subset $\Omega $ of $U$ with 
$\HD (U\setminus \Omega )<\HD (U)=2$ such that for each $\l \in \Omega $, 
$$1<\frac{\log (d_{1}+d_{2})}{\sum _{j=1}^{2}\frac{d_{i}}{d_{1}+d_{2}}\log (d_{i})}< \HD (J(G_{\l }))=s(\l )<2.$$   

\item[(iv)] 
$J(G_{\l _{0}})$ is connected and $\HD (J(G_{\l _{0}}))=s(\l _{0})<2.$ 
Moreover, $G_{\l _{0}}$ satisfies the open set condition.
Furthermore, for each $t\in (0,t_{1})$, 
$\langle \b_{1},g_{t}\rangle $ satisfies the separating open set condition, 
$\b _{1}^{-1}(J(\langle \b _{1},g_{t}\rangle ))\cap 
g_{t}^{-1}(J(\langle \b _{1},g_{t}\rangle ))=\emptyset $, 
$J(\langle \b _{1},g_{t}\rangle )$ is disconnected, 
and 
$$1<\frac{\log (d_{1}+d_{2})}{\sum _{j=1}^{2}\frac{d_{i}}{d_{1}+d_{2}}\log (d_{i})}< 
\HD (J(\langle \b _{1},g_{t}\rangle ))=\delta (\b _{1},g_{t})<2.$$ 
\end{itemize}
Moreover, there exists an open connected neighborhood $Y$ of $(\b
_{1},g_{t_{1}})$ in ${\mathcal P}^{2}$  
such that the family $\{ \g=(\g_{1},\g_{2})\} _{\g\in Y} $ 
satisfies all the conditions {\em (v)--(viii)}.
\begin{itemize}
\item[(v)] 
$\{ \g=(\g_{1},\g_{2})\} _{\g\in Y} $ is a holomorphic family 
in $\Epb (2)$ satisfying the analytic transversality condition,
the strong transversality condition and
 the transversality condition.
\item[(vi)] 
For each $\g\in Y$, $\delta (\g)<2$. 
\item[(vii)] 
There exists a subset $\Gamma $ of $Y$ with $\HD (Y\setminus \Gamma )<\HD (Y)=2(d_{1}+d_{2}+2)$ 
such that for each $\l \in \Gamma $, 
$$ 
1<\frac{\log (d_{1}+d_{2})}{\sum _{j=1}^{2}\frac{d_{i}}{d_{1}+d_{2}}\log (d_{i})}< 
\HD (J(\langle \g_{1} ,\g_{2}\rangle )=\delta (\g)<2.
$$ 
\item[(viii)] 
For each neighborhood $V$ of $(\b _{1},g_{t_{1}})$ in $Y$ there exists
a non-empty open set $W$ in $V$ such that 
for each $\g =(\g_{1},\g_{2})\in W$, we have that 
$\g _{1}^{-1}(J(\langle \g _{1},\g _{2}\rangle ))\cap \g _{2}^{-1}(J(\langle \g _{1},\g _{2}\rangle ))\neq \emptyset $ 
and that 
$J(\langle \g _{1},\g_{2}\rangle )$ is connected. 
\end{itemize}

\end{thm}
\begin{proof}
Let $z_{0}\in \{ |z|=1\} =J(\b_{1})$ be a point such that 
$|z_{0}-b|=\sup _{z\in J(\b_{1})}|z-b|.$ 
Then $z_{0}=e^{i(\pi +\theta )}.$  
Let $v:=|z_{0}-b|=1+|b|.$ 
Let $z_{1}:=2b-z_{0}. $ Then $z_{1}\in \{ z: |z-b|=v\} \setminus J(\b_{1})$.  
We note that 
\begin{equation}
\label{eq:gtnote}
g_{(\frac{1}{v})^{d_{2}-1}}(z_{0})=z_{1}.
\end{equation}
Let $r\in (1-u,1)$. Then $D(b,r)\subset \mbox{int}(K(\b_{1})).$ 
We also note that for each $t>0$, 
\begin{equation}
\label{eq:gtnote2}
g_{t}^{-1}(D(b,r))=D(b,(r/t)^{\frac{1}{d_{2}}}).
\end{equation}
Let $R\in \RR $ be any real number such that 
\begin{equation}
\label{eq:Rdef}
R>\exp \left(\frac{1}{d_{1}d_{2}-d_{1}-d_{2}}(-d_{1}\log
  r+d_{1}d_{2}\log 2)\right).  
\end{equation} 
We take $R$ satisfying (\ref{eq:Rdef}) so large that 
\begin{equation}
\label{eq:Rin}
D\left(b,\frac{3}{4}R^{\frac{1}{d_{1}}}\right)\subset \b_{1}^{-1}(D(b,R))
\subset D\left(b,\frac{3}{2}
  R^{\frac{1}{d_{1}}}\right)\subset \subset D(b,R), 
\end{equation}
where $A\subset \subset B$ denotes that $A$ is contained  in a compact
subset of $B.$  Let $a_R=1/R^{d_{2}-1}.$ 
By (\ref{eq:Rdef}), we obtain 
\begin{equation}
\label{eq:rt0d}
\left(\frac{r}{a_R}\right)^{\frac{1}{d_{2}}}>2R^{\frac{1}{d_{1}}}. 
\end{equation} 
We remark that 
\begin{equation}
\label{eq:gt0r}
J(g_{a_{R}})=\{ z: |z-b|=(1/a_R)^{\frac{1}{d_{2}-1}}\} =\{ z: |z-b|=R\} .
\end{equation}
We take a large $R$ so that 
\begin{equation}
\label{eq:dbrkh1}
D\left(b,\frac{1}{2}R^{\frac{1}{d_{1}}}\right)\supset K(\b_{1}).
\end{equation}
Then by (\ref{eq:dbrkh1}), (\ref{eq:Rin}), (\ref{eq:rt0d}),
(\ref{eq:gtnote2}) and (\ref{eq:gt0r}), we get that 
 \begin{equation}
\label{eq:kh1dbr}
\aligned
K(\b_{1}) &\subset D\left(b,\frac{1}{2}R^{\frac{1}{d_{1}}}\right)
\subset \subset D\left(b,\frac{3}{4}R^{\frac{1}{d_{1}}}\right)  
  \subset \b_{1}^{-1}(K(g_{a_R}))
  \subset D\left(b,\frac{3}{2}R^{\frac{1}{d_{1}}}\right)\\ 
& \subset \subset D(b,(r/a_R)^{\frac{1}{d_{2}}})=g_{a_R}^{-1}(D(b,r))
\subset g_{a_R}^{-1}(K(\b_{1})) \\
&\subset \subset \mbox{int}(K(g_{a_R})).
\endaligned
\end{equation}
Since the function $R\mapsto a_R$ is continuous and
$\lim_{R\to+\infty}a_R=0$, it follows from \eqref{eq:kh1dbr} that 
\begin{gather}
\label{eq:t1def}
\aligned
t_{1}:=
\sup \big\{ t\in [0,1/v^{d_{2}-1}]: &\forall c\in (0,t), 
K(\b_{1}) \subset \mbox{int}(\b_{1}^{-1}(K(g_{c}))) \\
&\subset \subset \mbox{int}(g_{c}^{-1}(K(\b_{1})))\subset \subset
\mbox{int}(K(g_{c}))\big\}>0.  
\endaligned
\end{gather}
By the definition of $t_{1}$, we get that 
\begin{equation}
\label{kh1h1}
K(\b_{1})\subset \b_{1}^{-1}(K(g_{t_{1}}))\subset
g_{t_{1}}^{-1}(K(\b_{1}))\subset K(g_{t_{1}}).  
\end{equation}
Therefore, by (\ref{eq:PG}), 
\begin{equation}
\label{eq:Pastlh}
P^{\ast }(\langle \b_{1},g_{t}\rangle )\subset K(\b_{1}) \mbox{ for
  each } t\in (0,t_{1}]. 
\end{equation} 
In addition, for each $t\in (0,t_{1})$, 
\begin{equation}
\label{eq:h1kgt}
\b_{1}^{-1}(K(g_{t})\setminus \mbox{int}(K(\b_{1})))
\amalg g_{t}^{-1}(K(g_{t})\setminus \mbox{int}(K(\b_{1})))
\subset K(g_{t})\setminus \mbox{int}(K(\b_{1})).
\end{equation}
In particular, for each $t\in (0,t_{1})$, the multimap
$(\b _{1},g_{t})$ satisfies the separating open set condition with 
$A_{t}:=\mbox{int}(K(g_{t}))\setminus K(\b _{1}).$ 
Moreover, by (\ref{eq:h1kgt}), (\ref{bsseq}) and \cite[Corollary 3.2]{HM}, 
for each $t\in (0,t_{1})$, the Julia set $J(\langle \b
_{1},g_{t}\rangle )$ is disconnected.   
Furthermore, by the definition \eqref{eq:t1def} of $t_1$, for each
$t\in (0,t_{1})$, we have that
$g_{t}(K(\b _{1}))\subset \mbox{int}(K(\b _{1}))$. 
Therefore, by (\ref{eq:PG}), for every $t\in (0,t_{1})$, 
$P^{\ast }(\langle \b _{1},g_{t}\rangle )\subset 
\mbox{int}(K(\b _{1}))\subset F(\langle \b _{1},g_{t}\rangle ).$ 
Thus for each $t\in (0,t_{1})$, 
$(\b _{1},g_{t})\in \Epb (2).$ 
Since $(\b _{1},g_{t})$ satisfies the open set condition,
 \cite[Theorem 1.2]{sumi2} implies that for every $t\in (0,t_{1})$, 
 $\HD (J(\langle \b _{1},g_{t}\rangle ))=\delta (\b_{1},g_{t}).$ 
Moreover, by (\ref{eq:h1kgt}), \cite[Corollary 3.2]{HM}, and (\ref{bsseq}),  
$J(\langle \b_{1},g_{t}\rangle )$ is a proper subset of 
$\overline{A_{t}}$ for each $t\in (0,t_{1}).$ 
Thus by \cite[Theorem 1.25]{sumi06}, 
$\HD (J(\langle \b_{1},g_{t}\rangle ))<2$ for each $t\in (0,t_{1}).$

We now prove the following claim.\\ 
 Claim 1: We have $t_{1}< \frac{1}{v^{d_{2}-1}}.$ In particular, 
$J(\b_{1})\cap J(g_{t_{1}})=\emptyset .$ 

In order to prove this claim, 
suppose on the contrary that $t_{1}=\frac{1}{v^{d_{2}-1}}.$ Then 
$J(g_{t_{1}})=\{ z: |z-b|=v\} $ and 
$z_{0}\in J(\b_{1})\cap J(g_{t_{1}}).$ 
By (\ref{kh1h1}), $g_{t_{1}}(K(\b_{1}))\subset K(\b_{1}).$ 
Hence $g_{t_{1}}(z_{0})\in K(\b_{1})\cap J(g_{t_{1}}).$ 
Since $g_{t_{1}}(z_{0})=z_{1}\not\in J(\b_{1})$, 
we obtain $J(g_{t_{1}})\cap \mbox{int}(K(\b_{1}))\neq \emptyset .$ 
However, since $K(\b_{1})\subset K(g_{t_{1}})$ (see (\ref{kh1h1})), 
we obtain a contradiction. Thus, we have proved Claim 1.  

 We now prove the following claim.\\ 
Claim 2: We have $K(\b_{1})\subset \mbox{int}(\b_{1}^{-1}(K(g_{t_{1}})))$ and 
$g_{t_{1}}^{-1}(K(\b_{1}))\subset \mbox{int}(K(g_{t_{1}})).$ 
In particular, $K(\b_{1})\subset \mbox{int}(g_{t_{1}}^{-1}(K(\b_{1})))$ and 
$g_{t_{1}}(K(\b_{1}))\subset \mbox{int}(K(\b_{1})).$ 
 
 To prove Claim 2, suppose $J(\b_{1})\cap
 \b_{1}^{-1}(J(g_{t_{1}}))\neq \emptyset .$  
 Then $J(\b_{1})\cap J(g_{t_{1}})\neq \emptyset $, and this contradicts Claim 1. 
 Similarly, we must have that $g_{t_{1}}^{-1}(J(\b_{1}))\cap
 J(g_{t_{1}})=\emptyset .$  
 Therefore, we have proved Claim 2. 
 
 Since $g_{t_{1}}(K(\b_{1}))\subset \mbox{int}(K(\b_{1}))$ (Claim 2), 
from (\ref{eq:PG})  
 it is easy to see that $P^{\ast }(\langle \b_{1},g_{t_{1}}\rangle
 )\subset \mbox{int}(K(\b_{1})) 
\subset F(\langle \b_{1},g_{t_{1}}\rangle ).$ 
 Therefore, $(\b_{1},g_{t_{1}})\in \Epb(2).$ 
We now prove the third  claim.

\vspace{1mm}\noindent Claim 3: $\b_{1}^{-1}(J(g_{t_{1}}))\neq
g_{t_{1}}^{-1}(J(\b_{1})).$  

 To prove Claim 3, 
let $\varphi _{1}$ be Green's function on 
$\CCI \setminus K(\b_{1})$ (with pole at infinity) 
and $\varphi _{2}$ be Green's function on 
$\CCI \setminus K(g_{t_{1}}).$ 
Then $\varphi _{1}(z)=\log |z|$ and 
$\varphi _{2}(z)=\log |z| +\frac{1}{d_{2}-1}\log t_{1}+O(\frac{1}{|z|}).$ 
Note that since $J(\b_{1})\subset \mbox{int}(K(g_{t_{1}}))$ (Claim 2), 
we have $\frac{1}{d_{2}-1}\log t_{1}<0.$ 
It is easy to see that Green's function $\varphi _{3}$ on 
$\CCI \setminus g_{t_{1}}^{-1}(K(\b_{1}))$ 
satisfies 
$\varphi _{3}(z)=\frac{1}{d_{2}}(\varphi _{1}(g_{t_{1}}(z)))
=\log |z|+\frac{1}{d_{2}}\log t_{1}+O(\frac{1}{|z|}).$ 
Similarly, 
Green's function $\varphi _{4}$ on 
$\CCI \setminus \b_{1}^{-1}(K(g_{t_{1}}))$ 
satisfies 
$\varphi _{4}(z)=\frac{1}{d_{1}}\varphi _{2}(\b_{1}(z))
=\log |z|+\frac{1}{d_{1}(d_{2}-1)}\log t_{1}+O(1/|z|).$ 
Therefore, if $\b_{1}^{-1}(J(g_{t_{1}}))= g_{t_{1}}^{-1}(J(\b_{1}))$, 
then $\frac{1}{d_{2}}\log t_{1}=\frac{1}{d_{1}(d_{2}-1)}\log t_{1}.$ 
Since $(d_{1},d_{2})\neq (2,2)$, we obtain $\log t_{1}=0.$ 
However, this contradicts $\frac{1}{d_{2}-1}\log t_{1}<0.$ Thus 
we have proved Claim 3. 
 
Let $A:= \mbox{int}(K(g_{t_{1}}))\setminus K(\b_{1}).$ 
By (\ref{kh1h1}) and Claim 2, 
$A$ is a non-empty open set in $\CC $ and 
$\b_{1}^{-1}(A)\cup g_{t_{1}}^{-1}(A)\subset A $ 
and $\b_{1}^{-1}(A)\cap g_{t_{1}}^{-1}(A)=\emptyset .$ 
Hence $(\b_{1},g_{t_{1}})$ satisfies the open set condition with 
$A.$ Combining it with the expandingness of $\langle \b_{1},g_{t_{1}}\rangle $, 
\cite[Theorem 1.2]{sumi2} implies that $\HD (J(\langle
\b_{1},g_{t_{1}}\rangle ))=\delta (\b_{1},g_{t_{1}}).$  
Moreover, by Claim 3, 
we have that $\b_{1}^{-1}(\overline{A})\cup
g_{t_{1}}^{-1}(\overline{A})$ is a proper subset of $\overline{A}.$  
Therefore by \cite[Corollary 3.2]{HM} and (\ref{bsseq}),  $J(\langle
\b_{1},g_{t_{1}}\rangle )$ is a proper subset of $\overline{A}.$  
Combining it with the expandingness of $\langle \b_{1},g_{t_{1}}\rangle $ again and 
\cite[Theorem 1.25]{sumi06}, 
we obtain $\HD (J(\langle \b_{1},g_{t_{1}}\rangle ))<2.$ 
Hence, $\delta (\b_{1},g_{t_{1}}) =\HD (J(\langle \b_{1},g_{t_{1}}\rangle ))<2.$
By Lemma~\ref{l:epbopen} and Theorem~\ref{t:delrp}, 
there exists an open neighborhood $Y_{0}$ of 
$(\b_{1},g_{t_{1}})$ in ${\mathcal P}^{2}$ such that 
for each $\g=(\g_{1},\g_{2})\in Y_{0}$, 
$\g\in \Epb (2)$ and $\delta (\g)<2.$ 
 
 We now consider the holomorphic family 
 $\{ f_{\l }\} _{\l \in U}$ in $\Epb(2)$,  
 where $U$ is a small open neighborhood of $0.$  
 Let $\l _{0}=0.$ 
 Let $G_{\l },h_{\l },\overline{h}_{\l }$ be as in the Setting $(\ast )$ 
(see Remark~\ref{r:setting}).  
By (\ref{kh1h1}) and Claim 2, it is easy to see that $\{ f_{\l }\} _{\l \in U}$ satisfies 
conditions (i),(ii),(iii) in Lemma~\ref{l:genprin} with $\alpha _{12}=2,\alpha _{21}=1.$  
Let $z\in f_{\l _{0},1}^{-1}(J(G_{\l _{0}}))\cap f_{\l _{0},2}^{-1}(J(G_{\l _{0}}))
=\b _{1}^{-1}(J(g_{t_{1}}))\cap g_{t_{1}}^{-1}(J(\b _{1})).$ 
Then by Corollary~\ref{c:hfdiff3}, 
\begin{equation}
\label{eq:2112ex}
\frac{\partial (\overline{h}_{\l }(21^{\infty },z)-\overline{h}_{\l }(12^{\infty },z))}{\partial \l}\bigg|_{\l =\l _{0}}
=-1-\sum _{n=2}^{\infty }\frac{-g_{t_{1}}'(g_{t_{1}}^{n-2}(\b _{1}(z)))}{(g_{t_{1}}^{n-1}\circ \b _{1})'(z)}
=-1-\sum _{n=2}^{\infty }\frac{-1}{(g_{t_{1}}^{n-2}\circ \b _{1})'(z)}.
\end{equation}
Since $\sum _{n=2}^{\infty }|\frac{-1}{(g_{t_{1}}^{n-2}\circ \b _{1})'(z)}|
=\sum _{n=2}^{\infty }\frac{1}{|(g_{t_{1}}^{n-2})'(\b _{1}(z))||\b _{1}'(z)|}
= \sum _{n=2}^{\infty }\frac{1}{d_{2}^{n-2}d_{1}|z|^{d_{1}-1}}<1, $ 
it follows that 
$$\frac{\partial (\overline{h}_{\l }(21^{\infty },z)-\overline{h}_{\l
  }(12^{\infty },z))}{\partial \l}\Big|_{\l =\l _{0}}\neq 0. 
$$  
Therefore, by Lemma~\ref{l:genprin}, 
shrinking $U$ if necessary, we obtain that 
$\{ f_{\l }\} _{\l \in U}$ satisfies the analytic transversality condition,  
the strong transversality
condition and the transversality condition.  
Since $\delta (\b _{1},g_{t_{1}})=s(\l _{0 })<2$ and 
$\l \mapsto s(\l )$ is continuous, 
shrinking $U$ if necessary, we obtain that 
for each $\l \in U$, 
$s(\l )<2.$ 
Therefore, by Theorems~\ref{t:stcmain} and \ref{t:fundfact1},  
there exists a subset $\Omega $ of $U$ with 
$\HD (U\setminus \Omega )<\HD (U)=2$ such that 
for each $\l \in \Omega $, 
$\HD (J(G_{\l }))=s(\l )<2.$ 

 By the definition of $t_{1}$, we have 
 $\b _{1}^{-1}(J(g_{t_{1}}))\cap g_{t_{1}}^{-1}(J(\b_{1}))\neq \emptyset .$ 
 In particular, 
$$
\b _{1}^{-1}(J(\langle \b_{1},g_{t_{1}}\rangle ))\cap
g_{t_{1}}^{-1}(J(\langle \b_{1},g_{t_{1}}\rangle ))\neq \emptyset .
$$ 
Combining this with the fact that the semigroup $\langle
\b_{1},g_{t_{1}}\rangle $ is postcritically bounded, 
\cite[Theorem 1.7, Theorem 1.5(2)]{sumiintcoh} implies that the Julia set
$J(\langle \b_{1},g_{t_{1}}\rangle )=J(G_{\l _{0}})$ is connected.  
Since $\{ f_{\l }\} _{\l \in U}$ satisfies the analytic transversality condition, 
by using Lemma~\ref{l:atcemb} and shrinking $Y_{0}$ if
necessary, we obtain that  
$\{ \g =(\g _{1},\g _{2})\} _{\g \in Y_{0}} $ satisfies the analytic
transversality condition,   
the strong
transversality condition and  
the transversality condition. 
Since $\delta (\g)<2$ for each $\g\in Y_{0}$, 
 Theorems~\ref{t:stcmain}, \ref{t:fundfact1} and \ref{t:delrp} imply that   
there exists a subset $\Gamma $ of $Y_{0}$ with 
$\HD (Y_{0}\setminus \Gamma )<\HD (Y_{0})=2(d_{1}+d_{2}+2)$ such that 
for each $\gamma =(\gamma _{1},\gamma _{2}) \in \Gamma $, 
$\HD ( J(\langle \g _{1},\g _{2}\rangle ))=\delta (\g)<2.$ 
Let $c_{0}\in \b_{1}^{-1}(J(g_{t_{1}}))\cap g_{t_{1}}^{-1}(J(\b_{1}))$. 
Let $w_{0}=\b_{1}(c_{0})\in J(g_{t_{1}}).$ 
There exists an open neighborhood $Y_{1}$ of $g_{t_{1}}$ in ${\mathcal P}$ 
and a holomorphic map $\zeta :Y_{1}\rightarrow \CCI $ such that 
$\zeta (g_{t_{1}})=w_{0}$ and 
$\zeta (\g_{2})\in J(\g_{2})$ for each $\g_{2}\in Y_{1}.$ 
Let $\xi $ be a well-defined inverse branch of $\b_{1}$ defined on a
neighborhood $D_{0}$  
of $w_{0}$ in $\CCI $ such that $\xi (w_{0})=c_{0}$.  
Let $\eta (\g_{2}):=\g_{2}\circ \xi \circ \zeta (\g_{2})$, 
which is defined on an open neighborhood $B_{0}$ of $g_{t_{1}}$ in $Y_{1}.$  
Then $\eta $ is holomorphic on $B_{0}.$ 
Moreover, $\eta (g_{t_{1}})\in J(\b_{1}).$ 
Furthermore, by the definition of $t_{1}$, 
for each $t$  close to $t_{1}$ with $t<t_{1}$, 
we have $\eta (g_{t})\not\in J(\b_{1}).$    
Hence $\eta $ is not constant on $B_{0}$. 
Therefore, for each neighborhood $V$ of $(\b _{1},g_{t_{1}})$ in $Y_{0}$,  
there exists an element $\overline{\g}_{2}$ with 
$(\b _{1}, \overline{\g}_{2})\in V$ such that 
$\eta (\overline{\g}_{2})\in \CC \setminus K(\b_{1}).$ 
In particular, 
\begin{equation}
\label{eq:b1-1j}
\b_{1}^{-1}(J(\overline{\g}_{2}))\cap \overline{\g}_{2}^{-1}(\CC \setminus K(\b _{1}))\neq \emptyset .
\end{equation}
Moreover, by (\ref{kh1h1}) and Claim 3, 
$\b_{1}^{-1}(J(g_{t_{1}}))\cap \mbox{int}(g_{t_{1}}^{-1}(K(\b_{1})))\neq \emptyset .$ 
Therefore, we may assume that 
\begin{equation}
\label{eq:b1-1g}
\b_{1}^{-1}(J(\overline{\g}_{2}))\cap \mbox{int}(\overline{\g}_{2}^{-1}(K(\b_{1})))\neq \emptyset .
\end{equation}
By (\ref{eq:b1-1j}) and (\ref{eq:b1-1g}), 
there exists an open neighborhood $W$ of $(\b_{1},\overline{\g}_{2})$ in $V$ 
such that for each $(\psi _{1},\psi _{2})\in W$, 
$$\psi _{1}^{-1}(J(\psi _{2}))\cap \psi _{2}^{-1}(J(\psi _{1}))\neq \emptyset .$$  
In particular, 
$$\psi _{1}^{-1}(J(\langle \psi _{1},\psi _{2}\rangle ))\cap \psi _{2}^{-1}(J(\langle \psi _{1},\psi _{2}\rangle ))\neq \emptyset .$$ 
Combining this with the fact that the semigroup $\langle \psi_{1},\psi
_{2}\rangle $ is postcritically bounded,  
\cite[Theorem 1.7, Theorem 1.5(2)]{sumiintcoh} implies that the Julia set 
$J(\langle \psi _{1},\psi _{2}\rangle )$ is connected for each $(\psi _{1},\psi _{2})\in W.$ 
  
Finally, we remark that by \cite[Theorem 3.15]{subowen}, 
for any $(\g _{1},\g_{2})\in \Epb(2) $ with $\deg (\g _{1})=d_{1},
\deg (\g _{2})=d_{2}$,  
if $\g _{1}(z)=z^{d_{1}} $ and $\g _{2}(z)=a(z-b)^{d_{2}}+b$ with $b\neq 0$, then 
we have 
$$1<\frac{\log (d_{1}+d_{2})}{\sum _{j=1}^{2}\frac{d_{i}}{d_{1}+d_{2}}\log (d_{i})}< 
\delta (\g _{1},\g _{2}).$$ 
   Thus we have proved Theorem~\ref{t:d1d2ex}. 
\end{proof}
Figure~\ref{fig:ConnOSC1} represents  the Julia set of the
$2$-generator polynomial semigroup  
$G_{\l _{0}}$ with $(d_{1},d_{2})=(3,2), b=0.1$. 
 For the relation between Theorem~\ref{t:d1d2ex} and 
random complex dynamics, see Remark~\ref{r:t0tinfty}.  

We now fix a complex number $a$ as required in the proposition below
and we consider a family of small perturbations of the multimap
$(z^{2},az^{2})$.  In the following we will see that 
for a typical value of the perturbation parameter, 
the $2$-dimensional Lebesgue measure of the Julia set of the
corresponding semigroup is positive.  
\begin{prop}
\label{p:d2pm}
Let $A:=\{ a\in \CC : |a|\neq 0,1, \mbox{ and } |2+a+\frac{1}{a}|\neq 4\} .$ 
Let $a\in A$ be a point. 
For each $b\in \CC $, let $f_{b,1}(z):=az^{2}$ (independent of $b$) and 
$f_{b,2}(z):= (z-b)^{2}+b$ and let  
$f_{b}:=(f_{b,1},f_{b,2})\in {\mathcal P}^{2}.$ 
For each $b\in \CC $, let $G_{b}:= \langle f_{b,1},f_{b,2}\rangle .$ 
Then there exists an open neighborhood $U$ of $0$ in $\CC $ such that 
$\{ f_{b}\} _{b \in U}$ is a holomorphic family in $\Epb(2)$
satisfying Setting $(\ast )$ with base point $0$  
and all of the following hold. 
\begin{itemize}
\item[(1)] 
The family $\{ f_{b}\} _{b\in U}$ satisfies the analytic
transversality condition,  
the strong transversality condition and
 the transversality condition.  
\item[(2)] 
For \mbox{{\em Leb}}$_{2}$-a.e. $b\in U$, 
{\em Leb}$_{2}(J(G_{b }))>0.$ 
\item[(3)] 
For each $b\in U$, 
let $h_{b}$ be 
the conjugacy map of the form 
$h_{b }(\om ,z)=(\om ,\overline{h}_{b}(\om ,z))$ 
between $\tilde{f}_{0}:J(\tilde{f}_{0})\rightarrow J(f_{0})$ 
and $\tilde{f}_{b}:J(\tilde{f}_{b})\rightarrow J(\tilde{f}_{b})$ as in
Setting ($\ast $).   
Let $\mu $ be the $s(0)$-conformal measure on $J(\tilde{f}_{0})$ for
$\tilde{f}_{0}.$  
Then for \mbox{{\em Leb}}$_{2}$-a.e. $b\in U$, the Borel probability measure 
$(\overline{h}_{b})_{\ast }(\mu )$ on $J(G_{b})$ is absolutely
continuous with respect to  
{\em Leb}$_{2}$ with $L^{2}$ density.   
\end{itemize}
\end{prop} 
\begin{proof}
It is easy to see that 
$P^{\ast }(G_{0})=\{ 0\} .$ Therefore $f_{0}\in \Epb (2).$ 
By Lemma~\ref{l:epbopen}, 
there exists an open neighborhood $U$ of $0$ such that 
for each $b\in U$, $f_{b}\in \Epb (2).$ 
By Remark~\ref{r:setting}, 
shrinking $U$ if necessary, 
for each $b\in U$, there exists a unique conjugacy map 
$h_{b }$ of the form $h_{b}(\om ,z)=(\om ,\overline{h}_{b}(\om ,z))$ 
between $\tilde{f}_{0}: J(\tilde{f}_{0})\rightarrow J(\tilde{f}_{0})$ and 
$\tilde{f}_{b}:J(\tilde{f}_{b})\rightarrow J(\tilde{f}_{b})$ as in
Setting ($\ast $),   
and $b\mapsto \overline{h}_{b}(\om ,z),b\in U,$ is
holomorphic for  
each $(\om ,z)\in J(\tilde{f}_{0}).$  
It is easy to see that 
$J(G_{0})$ is equal to the closed annulus between 
$J(f_{0,1})=\{ z\in \CC : |z|=1/|a|\} $ and 
$J(f_{0,2})=\{ z\in \CC : |z|=1\} $, and 
that 
$$
f_{0,1}^{-1}(J(G_{0}))\cap f_{0,2}^{-1}(J(G_{0}))=
\{ z\in \CC : |z|=|a|^{-\frac{1}{2}}\} 
=f_{0,1}^{-1}(J(f_{0,2}))=f_{0,2}^{-1}(J(f_{0,1})).
$$ 
Therefore, 
\begin{gather}
\label{eq:g0set}
\begin{split}
\ & \ \ \ \   \{ (\om ,z,\om ',z')\in J(\tilde{f}_{0})^{2}: \om
_{1}\neq \om '_{1},  
\overline{h}_{0}(\om ,z)-\overline{h}_{0}(\om ',z')=0\} \\  
& \subset \{ (12^{\infty }, z, 21^{\infty },z'): z=z'\in \{w\in \CC :
|w|=|a|^{-\frac{1}{2}}\} \}.   
\end{split} 
\end{gather}
By Corollary~\ref{c:hfdiff1},  for each 
$z\in \{w\in \CC : |w|=|a|^{-\frac{1}{2}}\} $, 
$$
\frac{\partial (\overline{h}_{b}(21^{\infty
  },z)-\overline{h}_{b}(12^{\infty },z))}{\partial b}\bigg|_{b=0}= 
1-\frac{1}{2z}-\frac{1}{2za}.
$$ 
Since $a\in A$, it is easy to see that 
for each $z\in \{ w\in \CC : |w|=|a|^{-\frac{1}{2}}\} $, 
$1-\frac{1}{2z}-\frac{1}{2za}\neq 0.$ 
Therefore, for each $z\in \{w\in \CC : |w|=|a|^{-\frac{1}{2}}\} $, 
$$
\frac{\partial (\overline{h}_{b}(21^{\infty
  },z)-\overline{h}_{b}(12^{\infty },z))}{\partial b}\bigg|_{b=0}\neq 0.
$$ 
Combining this with (\ref{eq:g0set}), and shrinking $U$ if necessary, 
we obtain that the family $\{ f_{b}\} _{b\in U}$ satisfies the
analytic transversality condition.  
 By Proposition~\ref{p:atctc}, shrinking $U$ if necessary, 
the family $\{ f_{b}\} _{b\in U}$ satisfies the strong transversality
condition and  the transversality condition.  
By \cite[Corollary 3.19]{subowen}, 
for each $b\in U\setminus \{ 0\} $, 
$s(b)>2.$ 
Hence, by Theorem~\ref{t:tcdimj}, 
 statements (2) and (3) of our proposition hold. 
Thus, we have proved our proposition.  
\end{proof}

\begin{thm}
\label{t:d2pm2}
Let $a\in \CC $ with $|a|>1.$ For each $\l \in \CC $, let
$f_{\l,1}(z):=az^{2}$ (independent of $\l $) and  
$f_{\l,2}(z):= z^{2}+\l $ and let  
$f_{\l }:=(f_{\l,1},f_{\l,2})\in {\mathcal P}^{2}.$ 
For each $\l \in \CC $, let $G_{\l }:= \langle f_{\l ,1},f_{\l ,2}\rangle .$ 
Then there exists an open neighborhood $U$ of $0$ in $\CC $ such that 
$\{ f_{\l }\} _{\l  \in U}$ is a holomorphic family in $\Epb(2)$
satisfying Setting $(\ast )$ with base point $0$  
and all of the following hold. 
\begin{itemize}
\item[(1)] 
The family $\{ f_{\l }\} _{\l \in U}$ satisfies the analytic
transversality condition, the strong transversality condition and
 the transversality condition.  
\item[(2)] 
For \mbox{{\em Leb}}$_{2}$-a.e. $\l \in U$, 
{\em Leb}$_{2}(J(G_{\l }))>0.$ 
\item[(3)] 
For each $\l \in U$, 
let $h_{\l }$ be 
the conjugacy map of the form 
$h_{\l }(\om ,z)=(\om ,\overline{h}_{\l }(\om ,z))$ 
between $\tilde{f}_{0}:J(\tilde{f}_{0})\rightarrow J(f_{0})$ 
and $\tilde{f}_{\l }:J(\tilde{f}_{\l })\rightarrow J(\tilde{f}_{\l })$
as in Setting ($\ast $) (with $\l _{0}=0$).
Let $\mu $ be the $s(0)$-conformal measure on $J(\tilde{f}_{0})$ for
$\tilde{f}_{0}.$  
Then for \mbox{{\em Leb}}$_{2}$-a.e. $\l \in U$, the Borel probability measure 
$(\overline{h}_{\l})_{\ast }(\mu )$ on $J(G_{\l})$ is absolutely
continuous with respect to  {\em Leb}$_{2}$ with $L^{2}$ density.   
\end{itemize}
\end{thm} 

\begin{proof}
It is easy to see that 
$P^{\ast }(G_{0})=\{ 0\} \subset F(G_{0}).$ Therefore $f_{0}\in \Epb (2).$ 
By Lemma~\ref{l:epbopen}, 
there exists an open neighborhood $U$ of $0$ such that 
for each $\l \in U$, $f_{\l }\in \Epb (2).$ By Remark~\ref{r:setting}, 
shrinking $U$ if necessary, 
for each $\l \in U$, there exists a unique conjugacy map 
$h_{\l }$ of the form $h_{\l }(\om ,z)=(\om ,\overline{h}_{\l }(\om ,z))$ 
between $\tilde{f}_{0}: J(\tilde{f}_{0})\rightarrow J(\tilde{f}_{0})$ and 
$\tilde{f}_{\l }:J(\tilde{f}_{\l })\rightarrow J(\tilde{f}_{\l })$ as
in Setting ($\ast $)  
with $\l_{0}=0$, and $\l \mapsto \overline{h}_{\l }(\om ,z)$ is holomorphic.  
It is easy to see that 
$J(G_{0})$ is equal to the closed annulus between 
$J(f_{0,1})=\{ z\in \CC : |z|=1/|a|\} $ and 
$J(f_{0,2})=\{ z\in \CC : |z|=1\} $, and 
that $f_{0,1}^{-1}(J(G_{0}))\cap f_{0,2}^{-1}(J(G_{0}))=
\{ z\in \CC : |z|=|a|^{-\frac{1}{2}}\} 
=f_{0,1}^{-1}(J(f_{0,2}))=f_{0,2}^{-1}(J(f_{0,1})).$ 
Therefore, 
\begin{gather}
\label{eq:g0set2}
\begin{split}
\ & \ \ \ \   \{ (\om ,z,\om ',z')\in J(\tilde{f}_{0})^{2}: \om
_{1}\neq \om '_{1},  
\overline{h}_{0}(\om ,z)-\overline{h}_{0}(\om ',z')=0\} \\  
& \subset \{ (12^{\infty }, z, 21^{\infty },z'): z=z'\in \{ w\in \CC :
|w|=|a|^{-\frac{1}{2}}\} \}.   
\end{split} 
\end{gather}
By Corollary~\ref{c:hfdiff2}, 
we obtain that for each $z\in J_{21^{\infty }}(\tilde{f}_{0})=\{w\in
\CC : |w|=|a|^{-\frac{1}{2}}\} $,  
$$ \frac{\partial \overline{h}_{\l }(21^{\infty },z)}{\partial \l }\bigg|_{\l =0} 
=\frac{-1}{2z},$$ 
 and for each $z\in J_{12^{\infty }}(\tilde{f}_{0})=\{w\in \CC :
 |w|=|a|^{-\frac{1}{2}}\} $,  
$$
 \frac{\partial \overline{h}_{\l }(12^{\infty },z)}{\partial \l }\bigg|_{\l =0} 
=\sum _{n=2}^{\infty }\frac{-1}{f_{0,(12^{\infty })|_{n}}'(z)}
=\sum _{n=2}^{\infty }\frac{-1}{2^{n}az\prod
  _{j=1}^{n-1}f_{0,(12^{\infty })_{j}}(z)}.
$$  
Therefore, for each $z\in \{w\in \CC : |w|=|a|^{-\frac{1}{2}}\} $, 
$$
\left| \frac{\partial \overline{h}_{\l }(21^{\infty },z)}{\partial \l
  }\bigg|_{\l =0}\right| 
=\frac{1}{2}|a|^{\frac{1}{2}} \mbox{ and } 
\left| \frac{\partial \overline{h}_{\l }(12^{\infty },z)}{\partial \l
  }\bigg|_{\l =0}\right|  
\leq \frac{1}{2}|a|^{-\frac{1}{2}}.
$$
Thus, for each $z\in \{w\in \CC : |w|=|a|^{-\frac{1}{2}}\} $, 
$$ 
\frac{\partial \overline{h}_{\l }(21^{\infty },z)}{\partial \l }\bigg|_{\l =0}-
\frac{\partial \overline{h}_{\l }(12^{\infty },z)}{\partial \l }\bigg|_{\l
  =0}\neq 0.
$$ 
Combining it with (\ref{eq:g0set2}), and shrinking $U$ if necessary, 
we obtain that the family $\{ f_{\l }\} _{\l \in U}$ satisfies the
analytic transversality condition.  
 By Proposition~\ref{p:atctc}, shrinking $U$ if necessary, 
the family $\{ f_{\l }\} _{\l \in U}$ satisfies the strong transversality condition and the transversality condition. 
By \cite[Corollary 3.19]{subowen}, 
for each $\l \in U\setminus \{ 0\} $, 
$s(\l )>2.$ 
Hence, by Theorem~\ref{t:tcdimj}, 
 statements (2) and (3) of our theorem hold. 
Thus, we have proved our theorem.  
\end{proof} 
\begin{cor}
\label{c:d2pm3}
Let $a\in \CC $ with $|a|>1.$ 
Let $V$ be an open subset of $\CC ^{d}.$ 
Let $\l _{0}\in V.$ 
Let $\{ f_{\l }=(f_{\l ,1},f_{\l ,2})\} _{\l \in V}$ be a holomorphic family 
in $\Exp (2)\cap {\mathcal P}^{2}.$ Suppose that there exists an open neighborhood $W$ of $0$ in $\CC $ and  
a holomorphic embedding $\eta :W\rightarrow V$ with $\eta (0)=\l _{0}$  
such that for each $c\in W$, $f_{\eta (c)}(z)=(az^{2}, z^{2}+c).$ 
Then  
there exists an open neighborhood $U$ of $\l _{0}$ in $V $ such that 
$\{ f_{\l }\} _{\l  \in U}$ is a holomorphic family in $\Epb(2)$ satisfying Setting $(\ast )$ 
with base point $\l _{0}$  
and all of the following hold. 
\begin{itemize}
\item[(1)] 
The family $\{ f_{\l }\} _{\l \in U}$ satisfies the analytic transversality condition, 
the strong transversality condition and
 the transversality condition.  
\item[(2)] 
For \mbox{{\em Leb}}$_{2d}$-a.e. $\l \in U$, 
{\em Leb}$_{2}(J(G_{\l }))>0.$ 
\item[(3)] 
For each $\l \in U$, 
let $h_{\l }$ be 
the conjugacy map of the form 
$h_{\l }(\om ,z)=(\om ,\overline{h}_{\l }(\om ,z))$ 
between $\tilde{f}_{\l _{0}}:J(\tilde{f}_{\l _{0}})\rightarrow J(\tilde{f}_{\l _{0}})$ 
and $\tilde{f}_{\l }:J(\tilde{f}_{\l })\rightarrow J(\tilde{f}_{\l })$ as in Setting ($\ast $).  
Let $\mu $ be the $s(\l _{0})$-conformal measure on $J(\tilde{f}_{\l _{0}})$ for $\tilde{f}_{\l _{0}}.$ 
Then for \mbox{{\em Leb}}$_{2d}$-a.e. $\l \in U$, the Borel probability measure 
$(\overline{h}_{\l})_{\ast }(\mu )$ on $J(G_{\l})$ is absolutely continuous with respect to 
{\em Leb}$_{2}$ with $L^{2}$ density.   
\end{itemize}
\end{cor}
\begin{proof}
By Theorem~\ref{t:d2pm2}, 
there exists an open neighborhood $W_{1}$ of $0$ in $\CC $ such that 
$\{ (az^{2},z^{2}+c)\} _{c\in W_{1}}$ is a holomorphic family in $\Epb (2)$ satisfying the analytic transversality condition. 
Hence, by Lemma~\ref{l:atcemb}, there exists an open disk neighborhood $U$ of $\l _{0}$ in $\CC ^{d}$ such that 
$\{ f_{\l }\} _{\l \in U}$ is a holomorphic family in $\Epb (2)$ satisfying the analytic transversality condition,  
the strong transversality condition and the transversality condition. 
For each $\l \in U$, we set $\Psi (\l ) = f_{\l }\in \Epb(2)\cap {\mathcal P}_{2}^{2}.$ 
By \cite[Corollary 3.19]{subowen}, 
\begin{align*}
\ & \ \ \  \ \{ g=(g_{1},g_{2})\in \Epb(2)\cap {\mathcal P}_{2}^{2}: \delta (g)\leq 2\} \\ 
 & =\{ (\alpha _{1}(z-b)^{2}+b, \alpha _{2}(z-b)^{2}+b): \alpha _{1},\alpha _{2}\in \CC \setminus \{ 0\} ,b\in \CC\} .
\end{align*}   
Let $A:= \{  (\alpha _{1}(z-b)^{2}+b, \alpha _{2}(z-b)^{2}+b): \alpha _{1},\alpha _{2}\in \CC \setminus \{ 0\} ,b\in \CC \} .$ 
Then $A$ is a holomorphic subvariety of $\Epb (2)\cap {\mathcal P}_{2}^{2}.$ 
Hence $\Psi ^{-1}(A)$ is a proper holomorphic subvariety of $U.$ 
Therefore $\mbox{Leb} _{2d} (\{ \l \in U: s(\l )\leq 2\} )=0. $ 
Thus, by Theorem~\ref{t:tcdimj}, 
statements (2) and (3) of our corollary hold. 
\end{proof}
From Corollary~\ref{c:d2pm3} we immediately obtain the following. 
\begin{cor}
\label{c:pmz2az2}
 For each $a\in \CC $ with $|a|\neq 0,1$, 
there exists an open neighborhood $Y_{a}$ of $(az^{2}, z^{2})$ in 
${\mathcal P}^{2}$ such that 
$\{ g=(g_{1},g_{2})\} _{g\in Y_{a}} $ is a holomorphic family 
in $\Epb (2)\cap {\mathcal P}_{2}^{2}$ satisfying Setting $(\ast )$ 
with base point $(az^{2},z^{2})$ and all of the following hold.
\begin{itemize}
\item[(1)] 
The family $\{ g=(g_{1},g_{2})\} _{g \in Y_{a}}$ satisfies the analytic transversality condition, 
the strong transversality condition and
 the transversality condition.  
\item[(2)] 
For a.e. $g =(g_{1},g_{2})\in Y_{a}$ with respect to the Lebesgue measure on ${\mathcal P}_{2}^{2}$, \\  
{\em Leb}$_{2}(J(\langle g_{1},g_{2}\rangle ))>0.$  
\item[(3)] 
Let $\l _{0}=(az^{2},z^{2})\in Y_{a}$ and 
for each $g =(g_{1},g_{2})\in Y_{a}$, 
let $h_{g }$ be 
the conjugacy map of the form 
$h_{g }(\om ,z)=(\om ,\overline{h}_{g }(\om ,z))$ 
between $\tilde{f}_{\l _{0}}:J(\tilde{f}_{\l _{0}})\rightarrow J(\tilde{f}_{\l _{0}})$ 
and $\tilde{f}_{g }:J(\tilde{f}_{g })\rightarrow J(\tilde{f}_{g })$ as in Setting ($\ast $).  
Let $\mu $ be the $s(\l _{0})$-conformal measure on $J(\tilde{f}_{\l _{0}})$ for $\tilde{f}_{\l _{0}}.$ 
Then for a.e. $g \in Y_{a}$ with respect to the Lebesgue measure on ${\mathcal P}^{2}_{2}$, the Borel probability measure 
$(\overline{h}_{g})_{\ast }(\mu )$ on $J(\langle g_{1},g_{2}\rangle )$ is absolutely continuous with respect to 
{\em Leb}$_{2}$ with $L^{2}$ density.   
\end{itemize}
\end{cor}
\begin{rem}
\label{r:Jholes}
For an $a\in \CC $ with $|a|\neq 0,1$, 
$J(\langle az^{2},z^{2}\rangle )$ is equal to the closed annulus between 
$\{ w\in \CC :|w|=1\} $ and $\{ w\in \CC: |w|=|a|^{-1}\} $, thus 
$\mbox{{\em int}}(J(\langle az^{2}, z^{2}\rangle ))\neq \emptyset .$ 
However, regarding Corollary~\ref{c:pmz2az2},  
it is an open problem to determine  for any other  
parameter value $(g_{1}, g_{2})\in Y_{a}$ with $\mbox{{\em Leb}}_{2}(J(\langle g_{1},g_{2}\rangle ))>0$, 
whether $\mbox{{\em int}}(J(\langle g_{1},g_{2}\rangle ))= \emptyset $ or not. 
(By \cite[Theorem 2.15]{sumid1}, at least we know that for each 
$(\g _{1},\g _{2})\in Y_{a}$, $J(\langle \g _{1},\g _{2}\rangle )$ is connected.) 
Let $a\in (0,1)\subset \RR .$ 
It is easy to see that for a small $\epsilon >0$, 
setting $g_{1,\epsilon }(z)=a(z+\epsilon )^{2}-\epsilon $ and 
$g_{2}(z)=z^{2}$, 
we have $J(g_{1,\epsilon })=\{ w\in \CC : |w+\epsilon |=a^{-1}\} $, 
$J(g_{2})=\{ z\in \CC : |w|=1\} $, 
$g_{2}|_{\{ x>0\} }^{-1}(a^{-1}-\epsilon )<g_{1,\epsilon }|_{\{ x>0\} }^{-1}(1)$ 
and $g_{2}|_{\{ x>0\} }^{-1}([1,a^{-1}-\epsilon ])\amalg g_{1,\epsilon }|_{\{ x>0\} }^{-1}([1,a^{-1}-\epsilon ])\subset [1,a^{-1}-\epsilon ].$   
Thus for each $n\in \NN $ with $n\geq 3$ 
there exists a small neighborhood $V_{n}$ of the above $(g_{1,\epsilon }, g_{2})$ in $Y_{a}$ 
such that for each $(\g _{1},\g _{2})\in V$, 
$F(\langle \g _{1},\g _{2}\rangle )$ has at least $n$ 
connected components and  
$J(\langle \g _{1},\g _{2}\rangle )$ is not a closed annulus.  
Since $\epsilon >0$ can be taken arbitrary small, we can deduce that 
for any $a\in \RR $ with $a>0,a\neq 1$, 
for each neighborhood $W$ of $(az^{2},z^{2})$ in $Y_{a}$ and for each $n\in \NN $ with $n\geq 3$, 
there exists a non-empty open subset 
$W_{n}$ of $W$ such that for each $(\g _{1},\g_{2})\in W_{n}$, 
$F(\langle \g _{1}, \g _{2}\rangle )$ has at least $n$ connected components and 
$J(\langle \g _{1},\g _{2}\rangle )$ is not a closed annulus. 
A similar argument shows that for any $a\in \CC $ with $|a|\neq 0,1$, 
for each neighborhood $W$ of $(az^{2},z^{2})$ in $Y_{a}$ there exists a non-empty open subset 
$\tilde{W}$ of $W$ such that 
for each $(\g _{1},\g _{2})\in \tilde{W}$, 
$F(\langle \g _{1},\g _{2}\rangle )$ has at least three connected components and 
$J(\langle \g _{1},\g _{2}\rangle )$ is not a closed annulus. 
\end{rem}

We now consider families of systems of affine maps.  
\begin{rem}
\label{r:Jssset}
Let $m\geq 2.$ 
For each $j=1,\ldots ,m$, 
let $g_{j}(z)=a_{j}z+b_{j}$, 
where $a_{j},b_{j}\in \CC , |a_{j}|>1.$ 
Let $G=\langle g_{1},\ldots ,g_{m}\rangle .$ 
Since $|a_{j}|>1$, $\infty \in F(G).$ 
Hence, by (\ref{bsseq}), 
$J(G)$ is a compact subset of $\CC $ which satisfies 
$J(G)=\bigcup _{j=1}^{m}g_{j}^{-1}(J(G)).$ 
Since $g_{j}^{-1}$ is a contracting similitude on $\CC $, 
it follows that 
$J(G)$ is equal to the self-similar set constructed by the 
family $\{ g_{1}^{-1},\ldots ,g_{m}^{-1}\}$ of contracting similitudes. 
For the definition of self-similar sets, see \cite{F,F0,Ki}. 
Note that $\delta (g_{1},\ldots ,g_{m})$ is equal to the unique solution of 
the equation $\sum _{i=1}^{m}|a_{i}|^{-t}=1, t\geq 0$. 
Thus $\delta (g_{1},\ldots ,g_{m})$ is the similarity dimension of 
$\{ g_{1}^{-1},\ldots ,g_{m}^{-1}\} .$ 
Conversely, any self-similar set constructed by 
a finite family $\{ h_{1},\ldots ,h_{m}\} $ of contracting similitudes on $\CC $ is equal to 
the Julia set of the rational semigroup $\langle h_{1}^{-1},\ldots ,h_{m}^{-1}\rangle .$    
\end{rem}
\begin{thm}
\label{t:autc}
Let $m\in \NN $ with $m\geq 2.$ 
For each $i=1,\ldots, m$,  
let $g_{i}(z)=a_{i}z+b_{i}, $ where 
$a_{i}\in \CC ,|a_{i}|>1$, $b_{i}\in \CC $. 
Let $G:=\langle g_{1},\ldots ,g_{m}\rangle .$ 
We suppose all of the following conditions.
\begin{itemize}
\item[(i)]
For each $(i,j)$ with $i\neq j$ and $g_{i}^{-1}(J(G))\cap
g_{j}^{-1}(J(G))\neq \emptyset ,$   
there exists a number $\alpha _{ij}\in \{ 1,\ldots ,m\}$ such that 
$$
g_{i}(g_{i}^{-1}(J(G))\cap g_{j}^{-1}(J(G)))\subset \left\{
\frac{-b_{\alpha _{ij}}}{a_{\alpha _{ij}}-1}\right\} .
$$

\item[(ii)]
If $i,j,k$ are mutually distinct elements in $\{ 1,\ldots ,m\} $, 
then 
$$ g_{k}(g_{i}^{-1}(J(G))\cap g_{j}^{-1}(J(G)))\subset F(G).$$ 

\item[(iii)]
For each $(j,k)$ with $j\neq k$, 
$g_{k}\left(\frac{-b_{j}}{a_{j}-1}\right)\in F(G).$ 
\end{itemize}
Then, there exists an open neighborhood $U$ of 
$(g_{1},\ldots ,g_{m})\in (\mbox{{\em Aut}}(\CC ))^{m}$ 
such that 
$\{ \g =(\g _{1},\ldots ,\g _{m})\} _{ \g \in U} $ 
is a holomorphic family in $\Exp(m)$ satisfying the analytic
transversality condition,  
the strong transversality condition and the transversality condition. 
\end{thm}
\begin{proof}
We first note that for each $j$, 
$J(g_{j})=\{ \frac{-b_{j}}{a_{j}-1}\} .$ 
By conditions (i) and (iii), $\alpha _{ij}\neq i$ for each 
$(i,j)$ with $i\neq j.$ 
By Lemma~\ref{expopenlem} and Remark~\ref{r:setting}, 
there exists a small open neighborhood $U$ of $(g_{1},\ldots ,g_{m})$ in 
$(\mbox{Aut}(\CC ))^{m}$ such that 
$\{ \g \} _{\g \in U} $ is a holomorphic family in $\Exp(m)$ satisfying Setting $(\ast )$ with base 
point $\g _{0}=(g_{1},\ldots ,g_{m})$ and letting $h_{\g },\overline{h}_{\g }, G_{\g }$ be as in Setting $(\ast )$,  
the map $\g \mapsto \overline{h}_{\g }(\om ,z), \g \in U,$ is holomorphic.   
We shall prove the following claim.

\vspace{2mm} Claim 1:  
If $i\neq j$ and  
$z_{0}\in g_{i}^{-1}(J(G))\cap g_{j}^{-1}(J(G))$, then 
\begin{equation}
\label{eq:autc1}
\nabla _{\g }(\overline{h}_{\g }(i\alpha _{ij}^{\infty },z_{0})-\overline{h}_{\g }(j\alpha _{ji}^{\infty },z_{0}))|
_{\g =\g _{0}} \neq 0.
\end{equation}
In order to prove Claim 1, let $i\neq j$ and  
$z_{0}\in g_{i}^{-1}(J(G))\cap g_{j}^{-1}(J(G))$. 
To show (\ref{eq:autc1}), 
by conjugating $G$ by a map $z\mapsto z-\frac{-b_{i}}{a_{i}-1}$, 
we may assume that $b_{i}=0.$ 
Let $V$ be a small open neighborhood of $0$ in $\CC $ and 
let ${\mathcal A}:= \{(g_{1},\ldots g_{i-1},g_{i}+\l z, g_{i+1},\ldots ,g_{m})\} _{\l \in V}.$ 
For this holomorphic family in $\Exp(m)$, 
let $h_{\l }^{0},\overline{h}_{\l }^{0}$ be the conjugating maps as in
Setting $(\ast )$ with base point $\l _{0}=0.$     
By Corollary~\ref{c:hfdiff2} and that $b_{i}=0$, we have  
$$
\frac{\partial \overline{h}_{\l }^{0}(i\alpha _{ij}^{\infty },
  z_{0})}{\partial \l}\bigg|_{\l =0} 
=\frac{-z_{0}}{a_{i}} \mbox{ and } 
\frac{\partial \overline{h}_{\l }^{0}(j\alpha _{ji}^{\infty
  },z_{0})}{\partial \l}\bigg|_{\l =0}=0.
$$ 
By (iii), 
we have $z_{0}\neq 0.$ 
Therefore, 
$$ 
\frac{\partial \overline{h}_{\l }^{0}(i\alpha _{ij}^{\infty },
  z_{0})}{\partial \l}\bigg|_{\l =0} 
- 
\frac{\partial \overline{h}_{\l }^{0}(j\alpha _{ji}^{\infty },
  z_{0})}{\partial \l}\bigg|_{\l =0}\neq 0.
$$ 
Thus, we have proved Claim 1. From this claim and from Lemma~\ref{l:genprin}, 
 shrinking $U$ if necessary, we obtain that 
 $\{ \g \} _{\g \in U} $ satisfies the analytic transversality condition, 
 the strong transversality condition and the transversality condition. 
 Thus we have proved Theorem~\ref{t:autc}. 
\end{proof}
\begin{rem}
\label{r:autc-cci}
Regarding Theorem~\ref{t:autc}, 
even if we replace ``$\mbox{\em Aut}(\CC )$'' by ``$\mbox{{\em Aut}}(\CCI )$'', 
we obtain similar results by using Lemma~\ref{l:atcemb}. 
\end{rem}

We give some examples to which we can apply Theorem~\ref{t:autc}. 
It seems true that those examples have not been dealt with explicitly in any literature of contracting IFSs with overlaps. 
\begin{ex}
\label{ex:m2autatc}
Let $g_{1}(z)=2z$ and $g_{2}(z)=2z-1.$ Let $G=\langle g_{1},g_{2}\rangle .$ Then 
$J(G)=[0,1].$ 
It is easy to see that $(g_{1},g_{2})$ satisfies the assumptions of Theorem~\ref{t:autc}. 
Moreover, $\delta (g_{1},g_{2})=\HD (J(G))=1<2.$ 
By Theorems~\ref{t:autc}, \ref{t:stcmain} and \ref{t:fundfact1}, 
there exists an open neighborhood $U$ of $(g_{1},g_{2})$ 
in $(\mbox{{\em Aut}}(\CC ))^{2}$ and a subset $A$ of $U$ 
with $\HD (U\setminus A)<\HD (U)=8$ 
such that 
{\em (1)} $\{ \g =(\g _{1},\g_{2})\} _{\g \in U}$ is a holomorphic family in $\Exp (2)$ 
satisfying the analytic transversality condition, the strong transversality condition and 
the transversality condition, and {\em (2)} for each $\g =(\g _{1},\g _{2})\in A$, 
$\HD (J(\langle \g _{1},  \g _{2}\rangle ))=\delta (\g _{1},\g _{2})<2.$  
\end{ex}
\begin{ex}
\label{ex:Sierpatc}
Let $p_{1},p_{2},p_{3}\in \CC $ be such that 
$p_{1}p_{2}p_{3}$ makes an equilateral triangle.  
For each $i=1,2,3$, let $g_{i}(z)=2(z-p_{i})+p_{i}.$ 
Let $G=\langle g_{1},g_{2},g_{3}\rangle .$ Then 
$J(G)$ is equal to the Sierpinski gasket.  
It is easy to see that $(g_{1},g_{2},g_{3})$ satisfies the assumptions of Theorem~\ref{t:autc}. 
Moreover, $\delta (g_{1},g_{2},g_{3})=\HD (J(G))=\frac{\log 3}{\log 2}<2.$ 
By Theorems~\ref{t:autc}, \ref{t:stcmain} and \ref{t:fundfact1}, 
there exists an open neighborhood $U$ of $(g_{1},g_{2},g_{3})$ 
in $(\mbox{{\em Aut}}(\CC ))^{3}$ and a subset $A$ of $U$ 
with $\HD (U\setminus A)<\HD (U)=12$ 
such that 
{\em (1)} $\{ \g =(\g _{1},\g_{2},\g _{3})\} _{\g \in U}$ is a holomorphic family in $\Exp (3)$ 
satisfying the analytic transversality condition, the strong transversality condition and 
the transversality condition, and {\em (2)} for each $\g =(\g _{1},\g _{2},\g_{3})\in A$, 
$\HD (J(\langle \g _{1},  \g _{2}, \g_{3}\rangle ))=\delta (\g _{1},\g _{2},\g_{3})<2.$  
\end{ex}
\begin{rem}
\label{r:Sierpatc}
Regarding Example~\ref{ex:Sierpatc}, 
for each open neighborhood $U$ of $(g_{1},g_{2},g_{3})$ in $(\mbox{{\em Aut}}(\CC ))^{3}$,  
there exists an open set $V$ in $U$ such that for each $\gamma =(\g _{1},\g_{2},\g_{3})\in V$, 
$\HD (J(\langle \g _{1},\g_{2},\g_{3}\rangle ))=\delta (\g _{1},\g_{2},\g_{3})<2.$ 
However, 
we can show 
that for each open neighborhood $U$ of $(g_{1},g_{2},g_{3})$ in $(\mbox{{\em Aut}}(\CC ))^{3}$,  
$$\HD (\{ \g =(\g _{1},\g _{2},\g _{3})\in U: \HD (J(\langle \g _{1},\g_{2},\g_{3}\rangle ))\neq \delta (\g _{1},\g _{2},\g_{3})\} ) \geq 10.$$
\end{rem}
\begin{ex}
\label{ex:SFatc}
For each $j=1,\ldots ,6$, let 
$p_{j}:=\exp (2j\pi \sqrt{-1}/6).$ Let $p_{7}:=0.$ 
For each $j=1,\ldots ,7$, 
let $g_{j}(z)=3(z-p_{j})+p_{j}.$ 
Let $G=\langle g_{1},\ldots ,g_{7}\rangle .$ 
Then $J(G)$ is equal to the Snowflake (see \cite[Example 3.8.12]{Ki}, Figure~\ref{fig:snowflake2}). 
It is easy to see that $(g_{1},\ldots ,g_{7})$ satisfies the assumptions of Theorem~\ref{t:autc} 
(see Figure~\ref{fig:snowflake2}). 
Moreover, $\delta (g_{1},\ldots ,g_{7})=\HD (J(G))=\frac{\log 7}{\log 3}<2.$ 
By Theorems~\ref{t:autc}, \ref{t:stcmain} and \ref{t:fundfact1}, 
there exists an open neighborhood $U$ of $(g_{1},\ldots ,g_{7})$ 
in $(\mbox{{\em Aut}}(\CC ))^{7}$ and a subset $A$ of $U$ 
with $\HD (U\setminus A)<\HD (U)=28$ 
such that 
{\em (1)} $\{ \g =(\g _{1},\ldots \g_{7})\} _{\g \in U}$ is a holomorphic family in $\Exp (7)$ 
satisfying the analytic transversality condition, the strong transversality condition and 
the transversality condition, and {\em (2)} for each $\g =(\g _{1},\ldots ,\g_{7})\in A$, 
$\HD (J(\langle \g _{1},  \ldots , \g_{7}\rangle ))=\delta (\g _{1},\ldots ,\g_{7})<2.$  
 
\end{ex}
\begin{ex}
\label{ex:Pentaatc}
For each $j=1,\ldots ,5$, 
let $p_{j}:=\exp(2j\pi \sqrt{-1}/5).$ 
For each $j=1,\ldots ,5$, let 
$g_{j}(z)=\frac{2}{3-\sqrt{5}}(z-p_{j})+p_{j}.$ 
Let $G=\langle g_{1},\ldots ,g_{5}\rangle .$ 
Then $J(G)$ is equal to the Pentakun (\cite[Example 3.8.11]{Ki}, Figure~\ref{fig:snowflake2}). 
It is easy to see that $(g_{1},\ldots ,g_{5})$ satisfies the assumptions of Theorem~\ref{t:autc} 
(see Figure~\ref{fig:snowflake2}). 
Moreover, $\delta (g_{1},\ldots ,g_{5})=\HD (J(G))=\frac{\log 5}{\log (\frac{2}{3-\sqrt{5}})}<2.$ 
By Theorems~\ref{t:autc}, \ref{t:stcmain} and \ref{t:fundfact1}, 
there exists an open neighborhood $U$ of $(g_{1},\ldots ,g_{5})$ 
in $(\mbox{{\em Aut}}(\CC ))^{5}$ and a subset $A$ of $U$ 
with $\HD (U\setminus A)<\HD (U)=20$ 
such that 
{\em (1)} $\{ \g =(\g _{1},\ldots \g_{5})\} _{\g \in U}$ is a holomorphic family in $\Exp (5)$ 
satisfying the analytic transversality condition, the strong transversality condition and 
the transversality condition, and {\em (2)} for each $\g =(\g _{1},\ldots ,\g_{5})\in A$, 
$\HD (J(\langle \g _{1},  \ldots , \g_{5}\rangle ))=\delta (\g _{1},\ldots ,\g_{5})<2.$  
\end{ex}
\begin{ex}
\label{ex:nkun}
There are infinitely many analogues of Sierpinski gasket or Pentakun which are called 
Hexakun, Heptakun, Octakun and so on (see \cite[page 119]{Ki}). 
As in Example~\ref{ex:Pentaatc}, for each such analogue, 
we obtain similar results on the family of small perturbations of the system of the analogue.  
\end{ex}
\begin{figure}[htbp]
\caption{(From left to right) Snowflake, Pentakun}    
\ \ \ \ 
\includegraphics[width=1.8cm,width=1.8cm]{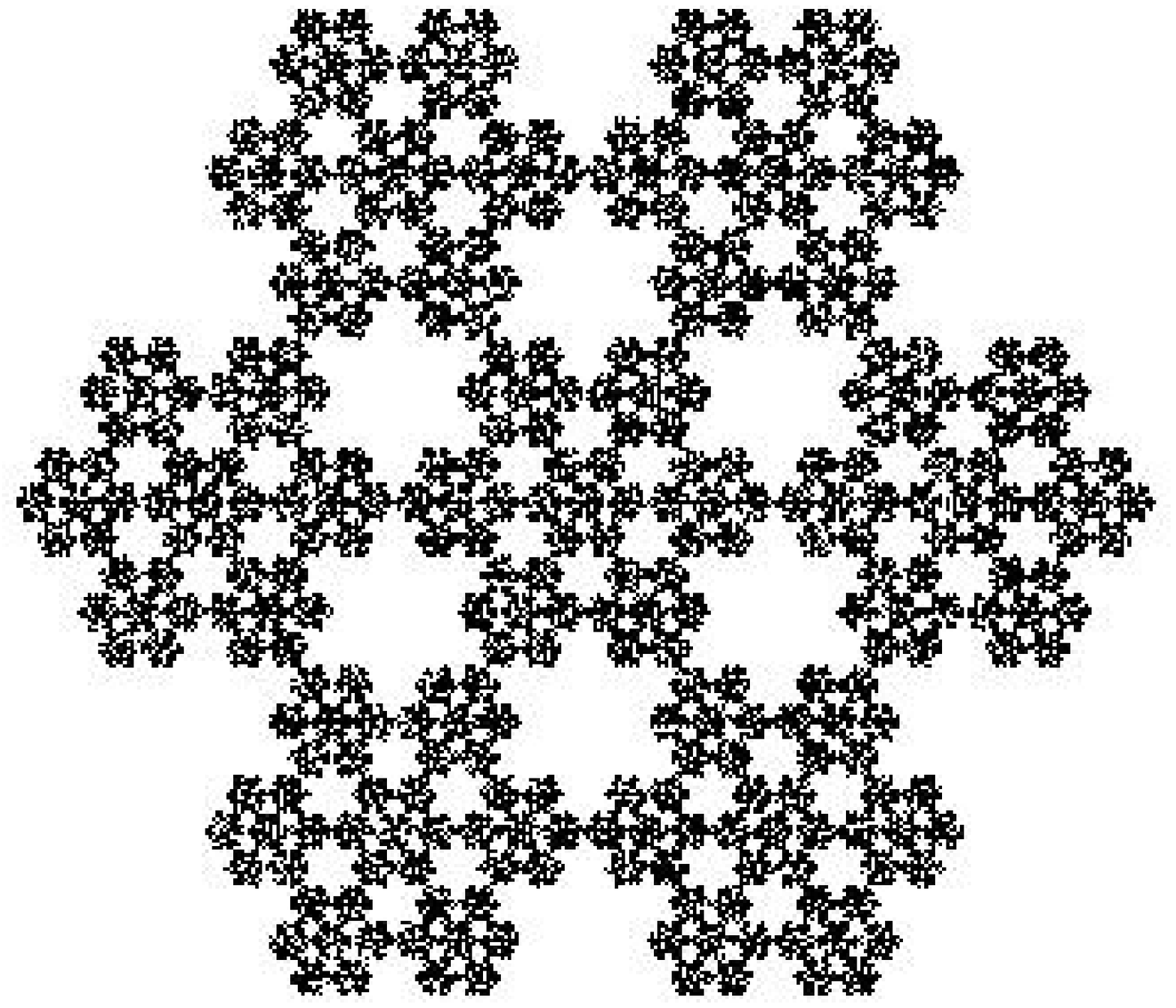}
\includegraphics[width=2.3cm, width=2.3cm,angle=17.5,totalheight=2.7cm]{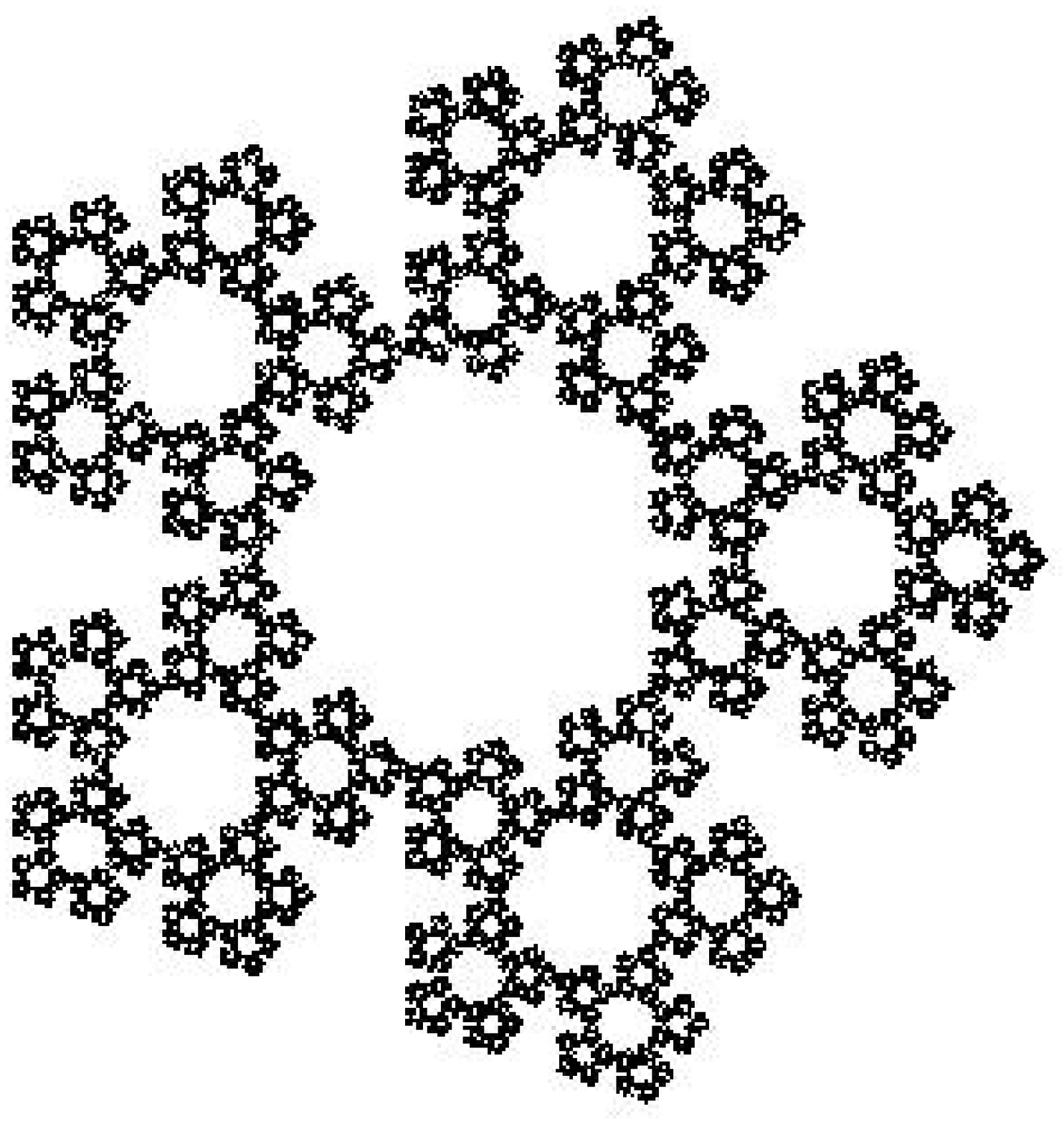}
\label{fig:snowflake2}
\end{figure}
\begin{rem}
\label{r:autcci}
Regarding Examples~\ref{ex:m2autatc}--\ref{ex:nkun}, 
even if we replace ``$\mbox{{\em Aut}}(\CC )$'' by 
``$\mbox{{\em Aut}}(\CCI )$'', 
we obtain similar results by using Lemma~\ref{l:atcemb}.

\end{rem}

As we see in Examples~\ref{ex:m2autatc}--\ref{ex:nkun} and Remark~\ref{r:autcci}, 
we have many examples to which we can apply Theorem~\ref{t:autc}. 
\section{Remarks}
\label{Remarks}
We finally give a remark. 
\begin{rem}
\label{r:gdCIFS}
We can prove similar results to those in sections~\ref{Results}, \ref{Applications} 
(especially Theorems~\ref{t:tcdimj}, \ref{t:stcmain}, Proposition~\ref{p:atctc}, Lemma~\ref{l:hfdiff}, 
Theorem~\ref{t:autc}) 
for a family 
$\{ \Phi ^{\l }\} _{\l \in U}=\{ \{ \varphi _{i}^{\l }\} _{i\in I} \} _{\l \in U} $ of 
hyperbolic conformal iterated function systems (CIFSs) 
 on an open subset $V$ of  
$\RR ^{p} (p\in \NN )$ without the open set condition, 
where $\varphi _{i}^{\l }:V\rightarrow V$ is a contracting conformal map, and 
$U$ is a bounded open subset of $\RR ^{d}, d\geq p.$  
For each $\l \in U$, we consider the limit set $J(\Phi ^{\l })$  of $\Phi ^{\l }.$ 
In the above setting, the definition of the transversality condition is modified such that 
the right hand side of (\ref{eq:tc}) is replaced by $C_{1}r^{p}$. 
The definition of the strong transversality condition is modified such that 
the right hand side of (\ref{eq:stc}) is replaced by 
$C_{1}'r^{p-d}.$ If $p=2$ and each $\varphi _{i}^{\l }$ is a holomorphic map, 
then we can define ``analytic transversality family'' just like
Definition~\ref{d:analtc}.    
The number $``2''$ (which represents the dimension of the phase space $\CCI $)  
in results of the previous sections are replaced by the number $p.$ 
These results will be stated and will be proved in the authors'
upcoming paper \cite{su}.  
\end{rem}

\end{document}